\documentclass{amsart}

\usepackage[latin1]{inputenc}
\usepackage[english]{babel}
\usepackage{enumerate}
\usepackage{graphicx}

\usepackage[colorlinks=false,backref=false,pagebackref=false,pdftex,pdfauthor={Denis Bonheure, Ederson M. dos Santos, Hugo Tavares},pdftitle={Hamiltonian systems_ a survey}]{hyperref}

\usepackage{amsthm}
\usepackage{amsmath,amsfonts,amssymb}

\usepackage{latexsym,hyperref}
\usepackage{comment}

\newcommand{\IR}{{\mathbb R}}

\newcommand{\IZ}{{\mathbb Z}}

\usepackage[showonlyrefs]{mathtools}
\mathtoolsset{showonlyrefs=true}

\newcommand{\R}{{\mathbb R}}

\newcommand{\N}{{\mathbb N}}

\newcommand{\I}{{\mathcal I}}
\newcommand{\J}{{\mathcal J}}
\newcommand{\G}{{\mathcal G}}

\newcommand{\Ncal}{\mathcal{N}}

\newcommand {\nin}{\noindent}
\newcommand {\menos}{\backslash}
\newcommand {\gd}{\displaystyle}

\newcommand {\con}{\subset}

\newcommand {\rt}{\rightarrow}

\newcommand{\eps}{\varepsilon}
\newcommand{\spann}{\text{span}}

\newcommand{\ip}[2]{\langle #1,  #2 \rangle}

\numberwithin{equation}{section}

\newtheorem{theorem}{Theorem}[section]
\newtheorem{corollary}[theorem]{Corollary}
\newtheorem{lemma}[theorem]{Lemma}
\newtheorem{proposition}[theorem]{Proposition}

\theoremstyle{definition}
\newtheorem{definition}[theorem]{Definition}

\newtheorem{remark}[theorem]{Remark}

\begin{document}

\title[Hamiltonian elliptic systems: a guide to variational frameworks]{Hamiltonian elliptic systems: a guide to variational frameworks}

\dedicatory{Ao Miguel Ramos, que tanto nos ensinou}


%
\author{Denis Bonheure}
\address{Denis Bonheure \newline \indent D{\'e}partement de Math{\'e}matique \newline \indent  Universit{\'e} libre de Bruxelles \newline \indent CP 214,  Boulevard du Triomphe, B-1050 Bruxelles, Belgium}
\email{denis.bonheure@ulb.ac.be}

\author{Ederson Moreira dos Santos}
\address{Ederson Moreira dos Santos \newline \indent Instituto de Ciências Matemáticas e de Computação \newline \indent Universidade de São Paulo \newline \indent
Caixa Postal 668, CEP 13560-970 - S\~ao Carlos - SP - Brazil}
\email{ederson@icmc.usp.br}

\author{Hugo Tavares}
\address{Hugo Tavares \newline \indent  Center for Mathematical Analysis, Geometry and Dynamical Systems  \newline \indent Mathematics Department \newline \indent Instituto Superior T\'ecnico, Universidade de Lisboa \newline \indent Av. Rovisco Pais, 1049-001 Lisboa, Portugal}
\email{htavares@math.ist.utl.pt}
\date{\today}




\begin{abstract}
Consider a Hamiltonian elliptic system of type
\begin{equation*}
\left\{
\begin{array}{ll}
-\Delta u=H_{v}(u,v) & \text{ in } \Omega\\
 -\Delta v=H_{u}(u,v) & \text{ in } \Omega\\
 u,v=0  & \text{ on } \partial \Omega
 \end{array}
\right.
\end{equation*}
where $H$ is a power-type nonlinearity, for instance \[H(u,v)= |u|^{p+1}/(p+1)+|v|^{q+1}/(q+1),\] having subcritical growth, and $\Omega$ is a bounded domain of $\R^N$, $N\geq 1$. The aim of this paper is to give an overview of the several variational frameworks that can be used to treat such a system.  Within each approach, we address existence of solutions, and in particular of ground state solutions. Some of the available frameworks are more adequate to derive certain qualitative properties; we illustrate this in the second half of this survey, where we also review some of the most recent literature dealing mainly with symmetry, concentration, and multiplicity results. This paper contains some original results as well as new proofs and approaches to known facts.
\end{abstract}




\maketitle
\tableofcontents

\section{Introduction}

Consider the problem
\begin{equation}\label{eq:main_system}
\left\{
\begin{array}{ll}
-\Delta u=H_{v}(u,v) & \text{ in } \Omega\\
 -\Delta v=H_{u}(u,v) & \text{ in } \Omega\\
 u,v=0  & \text{ on } \partial \Omega
 \end{array}
\right.
\end{equation}
where the coupling between the two equations is made through a Hamiltonian $H$ of the form $H(u,v) = |u|^{p+1}/(p+1) + |v|^{q+1}/(q+1)$, and $\Omega\subset \R^N$ is a \emph{bounded} domain, $N\geq 1$. In the literature these systems are usually referred to as elliptic systems of Hamiltonian type. It is also said that the equations are strongly coupled, in the sense that $u\equiv 0$ if and only if $v\equiv 0$; moreover, as we will see along this paper, many other properties are shared by the components of each solution pair.

The study of such system can be made through the use of variational methods. Unlike in the case of gradient systems where the choice of the energy functional associated to the problem is straightforward, in the case of Hamiltonian systems like \eqref{eq:main_system} there are several variational approaches available, each one with its advantages and disadvantages. The aim of this paper is to give an overview of several of these variational frameworks emphasizing that even if almost all of them are suitable to obtain existence and multiplicity theorems, some of them are more adequate to derive certain qualitative properties of the solutions. We also review some of the recent literature, complementing and updating in this way the surveys \cite[Section 3]{deFigueiredo} and \cite[Section 4]{Ruf} with only a few overlaps.  For instance, one of our main interests consists in the variational characterization of ground state solutions and on Nehari type approaches. These topics are not covered in \cite{deFigueiredo,Ruf}.  We also emphasize that in comparison to \cite{deFigueiredo,Ruf}, we focus on the simplest case where $H$ is a sum of pure powers in order to grasp the main ideas and to avoid too much technicalities.

As we already stated, we will focus on the model case 
\begin{equation}\label{Hamiltonian}
H(u,v) = \frac{1}{p+1} |u|^{p+1} + \frac{1}{q+1} |v|^{q+1},
\end{equation}%
so that the system becomes
\begin{equation}\label{eq:main_system_particularcase}
\left\{
\begin{array}{ll}
-\Delta u=|v|^{q-1}v & \text{ in } \Omega\\
-\Delta v=|u|^{p-1}u & \text{ in } \Omega\\
 u,v=0  & \text{ on } \partial \Omega.
 \end{array}
\right.
\end{equation}
The assumptions on the \emph{positive} powers $p$ and $q$ will be discussed in a while.
Formally, the equations in \eqref{eq:main_system} are the Euler-Lagrange equations of the action functional
\begin{equation}\label{eq:usual_action_functional}
(u,v)\mapsto \int_\Omega \ip{\nabla u}{\nabla v}\, dx-\int_\Omega H(u,v) \, dx.
\end{equation}
We will use the notation 
\begin{equation}\label{eq:quadraticpart}
\mathcal Q(u) =  \int_\Omega \ip{\nabla u}{\nabla v}\, dx
\end{equation}
to denote the quadratic part of the functional, while $\langle \cdot,\cdot \rangle$ denotes the canonical inner product of $\R^N$. An important question is to decide in which space the functional should be defined. A first natural choice could be to work with $(u,v)\in H^1_0(\Omega)\times H^1_0(\Omega)$. In order to define the functional in $H^1_0(\Omega)\times H^1_0(\Omega)$, we need to assume that
\begin{equation*}\label{eq:p+1,q+1<2^*}
(p+1)(N-2),\ (q+1)(N-2)\leq 2N,
\end{equation*}
whereas the strict inequality, for $N\geq 3$, is required in order to get compactness properties. However, as was simultaneously observed in \cite{deFigueiredoFelmer,HulshofvanderVorst}, this is too restrictive; indeed, the correct notion of \emph{subcriticality} associated to \eqref{eq:main_system} is 
\begin{equation}\label{eq:p_and_q_subcritical}
\frac{1}{p+1}+\frac{1}{q+1}>\frac{N-2}{N},
\end{equation}
while \emph{criticality} corresponds to $(p,q)$ lying on the so called \emph{critical hyperbola}:
\begin{equation}\label{eq:critical_hyperbola}
\frac{1}{p+1}+\frac{1}{q+1}=\frac{N-2}{N}.
\end{equation}

\begin{center}
\includegraphics[scale=.26]{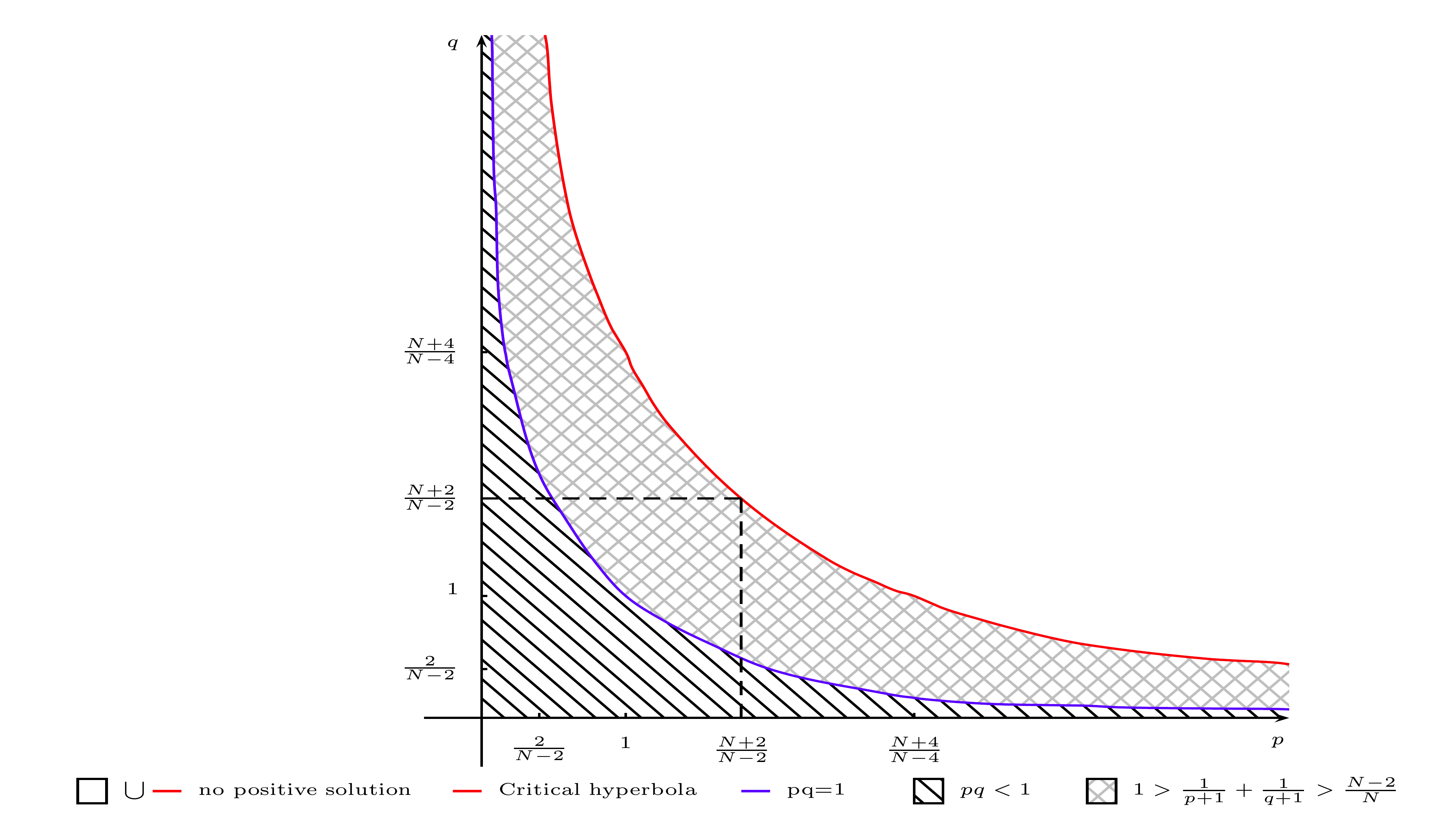}
\end{center}
\vspace{-.2cm}

Following  the aforementioned papers, we can motivate this fact at least in two different ways. First, if $q=1$, the system \eqref{eq:main_system_particularcase} reduces to the fourth order problem
\begin{equation*}
\left\{
\begin{array}{ll}
\Delta^2 u=|u|^{p-1}u & \text{ in } \Omega\\
 u,\Delta u=0  & \text{ on } \partial \Omega,
 \end{array}
\right.
\end{equation*}
whose critical exponent $p+1$ is given by $2N/(N-4)$, which is larger than $2N/(N-2)$. Observe that this is consistent with the choice of $q=1$ in \eqref{eq:critical_hyperbola}. On the other hand, the critical hyperbola also arises in the generalized Poho\u zaev identity due to Pucci and Serrin \cite{pucci-serrin}, Clément et al. \cite{ClementdeFigueiredoMitidieri}, Mitidieri \cite{Mitidieri}, van der Vorst \cite{vanderVorst1}, Peletier and van der Vorst \cite{PeletiervanderVorst}; in case of \eqref{eq:main_system_particularcase}, this identity reads as
\[
\left(\frac{N}{p+1}-\alpha\right)\int_\Omega |u|^{p+1}\, dx+\left(\frac{N}{q+1}-(N-2-\alpha)\right)\int_\Omega |v|^{q+1}\, dx=\int_{\partial \Omega} \frac{\partial u}{\partial \nu}\frac{\partial v}{\partial \nu}\, d\sigma
\]
for every $\alpha>0$. By choosing first $\alpha=N/(p+1)$, one shows that the system \eqref{eq:main_system_particularcase} does not have positive solutions on star shaped domains if $(p,q)$ lies on or above the critical hyperbola, namely if
\[
\frac{1}{p+1}+\frac{1}{q+1}\leq\frac{N-2}{N}.
\]
The previous arguments show that one should aim at working with $(p,q)$ satisfying \eqref{eq:p_and_q_subcritical}. However, under such assumption, it may happen that for instance (for $N\geq 3$) $p<\frac{N+2}{N-2}<q$, and thus the action functional may not be well defined on $H^1_0(\Omega)\times H^1_0(\Omega)$. This fact is the first reason why, in the literature, several, though equivalent, variational approaches are considered.

One of the facts that will come out of our exposition is that the more general and meaningful notion of \emph{superlinearity} is not $p,q>1$, both nonlinearities superlinear, but rather
\[
\frac{1}{p+1}+\frac{1}{q+1}<1\quad \text{or, equivalently, } \quad pq>1.
\]
This allows for instance $p<1<q$. Both the conditions $pq>1$ and \eqref{eq:p_and_q_subcritical} are strongly related with the strong coupling in \eqref{eq:main_system_particularcase}; the ideia is that one of the exponents can go ``slightly'' outside the interval $(1,2^*-1)$, as long as the other one compensates it.

\medbreak

In the first sections of this paper, namely from  Section \ref{section:direct_approaches} to Section \ref{sec:Reduced Functional}, we overview several variational frameworks which have been used in the literature to deal with Hamiltonian systems. All these approaches can also be used under Neumann boundary conditions. Here, for simplicity, we have decided to deal only with homogeneous Dirichlet conditions like in \eqref{eq:main_system}. Let us describe the content of these sections. The bibliographic references to each method and result can be found in the corresponding section.

We start in Section \ref{section:direct_approaches} by reviewing two possible frameworks built directly in the functional \eqref{eq:usual_action_functional}. For this reason, we shall call them \emph{direct approaches}. The first one consists in using the Sobolev spaces $W^{1,s}\times W^{1,\frac{s-1}{s}}$, for some suitable $s>1$. For $s\neq 2$ this is not a Hilbert space and hence this approach is rarely used to prove existence results. Nevertheless, it allows one to give a definition of ground state solution for every $(p,q)$ subcritical, and it is also useful when proving energy estimates. The second direct approach deals with fractional Sobolev spaces, with the definitive advantage of providing an Hilbertian framework. 

It will become clear from the direct approaches that another difficulty when dealing with \eqref{eq:usual_action_functional} is the fact that $\mathcal Q$, its quadratic part, is strongly indefinite in the sense that it is positive and negative respectively in two infinite dimensional subspaces which split the function space in two. Moreover, \eqref{eq:usual_action_functional} does not have a mountain pass geometry; in particular, the origin $(0,0)$ is not a local minimum. Instead, one has to rely in other linking theorems of more complicated nature. Another related issue is the fact that the usual Nehari manifold is not suitable to describe the ground state level. 

One alternative to get rid of the indefinite character of \eqref{eq:usual_action_functional} is to use the \emph{dual method}, which we describe in Section \ref{sec:DualMethod}. In an informal basis, the method consists in taking the inverse of the Laplace operator, rewriting the system as
\[
(-\Delta)^{-1}(|v|^{q-1}v)=u,\qquad (-\Delta)^{-1}(|u|^{p-1}u)=v.
\]
and defining $w_1=|u|^{p-1}u$, $w_2= |v|^{q-1}v$, which leads to
\begin{equation*}
(-\Delta)^{-1}w_2=|w_1|^{\frac{1}{p}-1}w_1,\qquad (-\Delta)^{-1}w_1=|w_2|^{\frac{1}{q}-1}w_2.
\end{equation*}
The associated energy functional, defined in a suitable product of Lebesgue spaces, has a mountain pass geometry.

Finally, other possibility is to reduce the problem to a scalar one (cf. Section \ref{sec:QuartaOrdem} and \ref{sec:Reduced Functional}). In Section \ref{sec:QuartaOrdem} we explain the \emph{reduction by inversion}, which heuristically consists in taking $v:=|\Delta u|^{\frac{1}{q}-1}(-\Delta u)$ and replacing it in the second equation of \eqref{eq:main_system_particularcase}, leading to the single equation problem of higher order
\begin{equation*}
\left\{
\begin{array}{rcll}
\Delta \left( | \Delta u |^{\frac{1}{q} -1}\Delta u \right) & = & | u |^{p-1}u & \hbox{in}\,\,\Omega,\\
u , \Delta u & = & 0 & \hbox{on}\,\,\partial \Omega.
\end{array}
\right.
\end{equation*}
This approach allows to deal with the sublinear case $pq<1$ as well, and reduces the problem of ground state solutions to the easier study of finding solutions which achieve the best constant of a related Sobolev embedding.

In Section \ref{sec:Reduced Functional}, for the case $p,q>1$, we introduce a Nehari type manifold of infinite codimension in order to characterize with a minimization problem the ground state level. Moreover, by exploiting the properties of \eqref{eq:usual_action_functional} on the pairs of type $(u,u)$ and $(u,-u)$, one can find, for each $u$, a function $\Psi_u$ so that the energy \eqref{eq:usual_action_functional} calculated on $(u+\Psi_u,u-\Psi_u)$ once again displays a mountain pass geometry, and its critical points (in $u$) correspond to solutions of the original system. This approach can be thought as being a \emph{Lyapunov-Schmidt type reduction}.

\medbreak

By using either the dual or the reduction by inversion method, one can prove in a relatively easy way positivity and symmetry properties for ground state solutions. However, it seems that this result does not follow easily with the other methods. Hence, one of the interesting things about all these approaches is that each method is more suitable to prove certain properties of the solutions. We illustrate this in a deeper way in the second part of the paper, from Section \ref{sec:more_on_symmetry} on, where we survey some recent literature, highlighting for each stated result the most suitable framework. In Section \ref{sec:more_on_symmetry} we combine the reduction by inversion approach with some arguments based on polarization of functions to prove symmetry properties of ground state solutions for two classes of systems. In particular, we solve an open problem, cf. \cite[p. 451]{BonheureSantosRamosTAMS}, about the radial symmetry of ground state solutions of a system posed on $\R^N$. For the so called Hénon-type system, we prove the foliated Schwarz symmetry of the ground state solution, as well as we present a result about symmetry breaking. Section \ref{sec:concentration} reviews the existing concentration results available for \eqref{eq:main_system}; there, the chosen method is the Lyapunov-Schmidt type reduction. After that, in Section \ref{sec:MultiplicityResults} we show how to obtain infinitely many solutions (by three different methods: the two reductions and a Galerkin type method) in both the symmetric case \eqref{eq:main_system_particularcase} as well as in the perturbation from symmetry problem. Finally, the last section is about sign-changing solutions; Subsection \ref{sec:lens}  deals with a very recent result of existence and symmetry properties of least energy nodal solutions via dual method, while Subsection \ref{sec:infinitely_sign_changing} is about the existence of infinitely many sign-changing solutions of \eqref{eq:main_system_particularcase} via the Lyapunov-Schmidt type reduction. 

To sum up, one can say in conclusion that it seems that the direct approaches are harder to apply to \eqref{eq:main_system} and have been less used in the past, mainly because it is hard to deal directly with the strongly indefinite functional \eqref{eq:usual_action_functional}. The dual method seems to be more adapted to prove sign and symmetry results; the reduction by inversion to prove sign, symmetry and multiplicity results, while the Lyapunov-Schmidt type reduction is useful in proving concentration and multiplicity results.
\medbreak

We finish this introduction by stressing that, although this paper is mainly a survey, it contains some original results, proofs, and computations that have not appeared elsewhere. As an example we refer to:
\begin{enumerate}[-]
\item the proof that the standard Nehari manifold cannot be used to define the ground state level (Proposition \ref{nehari nao funciona});
\item the fact that the functional associated to the dual method in Section \ref{sec:DualMethod} satisfies the Palais Smale condition (Proposition \ref{psconditionphi});
\item a simple proof for the radial symmetry of ground state solutions by using the dual method and under the mere hypothesis \eqref{eq:(p,q)_for_Ederson_part}, which includes cases with $p<1$ or $q<1$ (Theorem \ref{thm:qualitativeprop_dual});
\item comprehensive proofs of the several characterizations of least energy level in Section \ref{sec:Reduced Functional};
\item we solve an open problem, cf. \cite[p. 451]{BonheureSantosRamosTAMS}, about the radial symmetry of ground state solutions of a system posed on $\R^N$ (Theorem \ref{Th:radialsymmetryRN});
\item with respect to the existing bibliography, we prove the concentration results in Section \ref{sec:concentration} under more general assumptions on the nonlinearities.
\end{enumerate}






\noindent \textbf{Notations.} We will denote the $L^r$-norm by $\|u\|_{r}:=\left(\int_\Omega |u|^r\, dx\right)^{1/r}$. We will always assume $N\geq 1$, except it is specifically mentioned, and define $2^*=+\infty$ if $N=1,2$; $2^*=2N/(N-2)$ otherwise.


\section{Direct approaches}\label{section:direct_approaches}

In this section we present two approaches built directly on the action functional \eqref{eq:usual_action_functional}. In Subsection \ref{subsec:W^1s} we use the spaces $W^{1,s}\times W^{1,\frac{s-1}{s}}$, while in Subsection \ref{subset:Fractional} we deal with fractional Sobolev spaces. We also make some remarks concerning \emph{least energy solutions}. In order to simplify the presentation, throughout this section we focus on the model case \eqref{Hamiltonian}, so that the system in consideration is \eqref{eq:main_system_particularcase}.
%

\subsection{The $W^{1,s}\times W^{1,\frac{s}{s-1}}$ framework} \label{subsec:W^1s}

Having in mind the goal of finding a space in which \eqref{eq:usual_action_functional} is well defined for $(p,q)$ lying bellow the critical hyperbola, following for instance \cite[Section 1]{ClementvanderVorst} (see also \cite[Section 2]{BonheureSantosRamosJFA}), we observe that for every $s>1$,
\[
\left|\int_\Omega \langle \nabla u,\nabla v\rangle dx \right| \le \|\nabla u\|_{L^{s}(\Omega)}\|\nabla v\|_{L^{\frac{s}{s-1}}(\Omega)}.
\]
Therefore the quadratic part $\mathcal Q$ of \eqref{eq:usual_action_functional} is well defined on the product $W^{1,s}_0(\Omega)\times W^{1,\frac{s}{s-1}}_0(\Omega)$, and if for some $s>1$,  one has the embeddings 
\begin{equation}\label{eq:W^s_0_embeddings}
W^{1,s}_0(\Omega)\hookrightarrow   L^{p+1}(\Omega),\qquad W^{1, \frac{s}{s-1}}_0(\Omega)\hookrightarrow L^{q+1}(\Omega),
\end{equation}
the integral of the Hamiltonian in \eqref{eq:usual_action_functional} is finite. Let us suppose without loss of generality that $p\geq q$. One easily sees that the previous embeddings are continuous and compact whenever $s>1$ is such that
\[
(N-s)(p+1)<sN,\qquad \left(N-\frac{s}{s-1}\right)(q+1)<\frac{sN}{s-1},
\]
or, equivalently,
\[
N(p+1)<s(N+p+1),\qquad s((N-1)(q+1)-N)<N(q+1).
\]
There are now two possibilities: either $(N-1)(q+1)-N\leq 0$ and we can take any suitable large $s$, or $q+1>N/(N-1)$ and we can choose $s$ satisfying
\[
1<\frac{N(p+1)}{N+p+1}<s<\frac{N(q+1)}{(N-1)(q+1)-N}
\]
under the mere assumption that
\begin{equation}\label{eq:(p,q)_for_W^1,s_0}
p,q>0,\qquad \frac{1}{p+1}+\frac{1}{q+1}>\frac{N-2}{N}. \tag{H1}
\end{equation}
\vspace{-.5cm}
 \begin{center}
 \includegraphics[scale=.26]{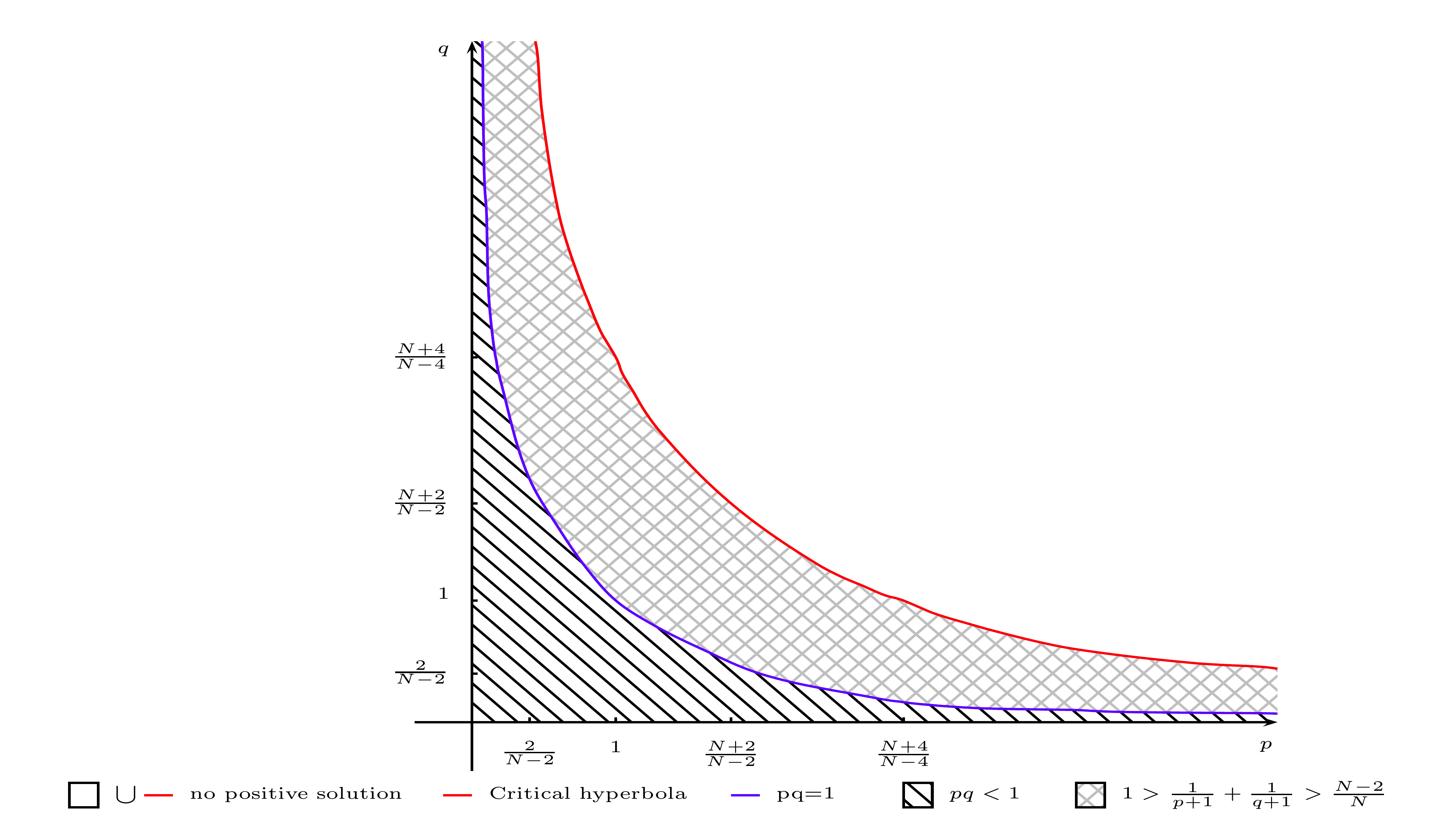}
 \end{center}
\vspace{-.2cm}

In conclusion, for $(p,q)$ satisfying \eqref{eq:(p,q)_for_W^1,s_0}, we can choose $s>1$ so that the embeddings \eqref{eq:W^s_0_embeddings} are continuous and compact, and in particular we can define the action functional $\mathcal G_{s}:W^{1, s}_{0}(\Omega)\times W^{1, \frac{s}{s-1}}_{0}(\Omega)\to \R$ by
\begin{equation}\label{indefh3}
(u,v)\mapsto \mathcal G_{s}(u,v) = \int_{\Omega}\langle \nabla u,\nabla v\rangle dx-\int_{\Omega}H(u,v) dx.
\end{equation}
In this framework, a \emph{weak solution} of \eqref{eq:main_system} is a critical point of $\mathcal G_s$, i.e.
a couple $(u,v)\in W^{1, s}_{0}(\Omega)\times W^{1, \frac{s}{s-1}}_{0}(\Omega)$ such that
\[
\G_{s}'(u,v)(\varphi,\psi)=\int_{\Omega}\left( \langle \nabla u, \nabla \psi\rangle +
\langle \nabla v, \nabla \varphi\rangle -  H_{u}(u,v)\varphi -  H_{v}(u,v)\psi \right) dx= 0,
\]
for every $(\varphi,\psi) \in W^{1, s}_{0}(\Omega)\times W^{1, \frac{s}{s-1}}_{0}(\Omega)$. Observe that if we assume $(p+1)(N-2),(q+1)(N-2)<2N$, then we can choose $s=2$ and $H^1_0(\Omega)\times H^1_0(\Omega)$ is an agreeable framework.

It is clear that $(0,0)$ is not a local minimum of $\G_s$. Indeed, the quadratic part $\mathcal Q$ is indefinite since $\mathcal Q(u,u)$ is positive definite whereas $\mathcal Q(u,-u)$ is negative definite for $u\in H^1_0(\Omega)$. This implies that the functional $\G_s$ does not display a mountain pass geometry. Moreover, $W^{1, s}_{0}(\Omega)\times W^{1, s/(s-1)}_{0}(\Omega)$ is not a Hilbert space if $s \neq 2$, which makes linking theorems as the one by Benci and Rabinowitz \cite{BenciRabinowitz} not applicable (we refer to the next subsection for a different framework which allows the use of linking theorems). Due to this fact, it is quite involved to show existence results using directly the functional $\G_s$. We refer to \cite{deFigdoORuf} or \cite[Section 5]{Ruf} for an approach in that direction. On the other hand, once one knows that a solution actually exists (for example through other approaches), one can then use $\G_s$ to obtain energy estimates. In \cite[Section  2]{BonheureSantosRamosJFA}, for instance, this framework has proved itself to be useful in estimating the level of \emph{least energy solutions}, also called $\emph{ground state}$ solutions. These can be defined as pairs $(u,v)$ achieving
\[
c_{s}(\Omega):=\inf \{ \G_s(u,v):\ (u,v)\in W^{1,s}_0(\Omega)\times W^{1,\frac{s}{s-1}}(\Omega),\ (u,v)\neq (0,0), \  \G_s'(u,v)=0\}.
\]
A priori this level could depend on $s$, but this turns out not to be the case.  Indeed, arguing as in \cite[Proposition 2.1]{BonheureSantosRamosTAMS}, see also \cite[Theorem 1]{Sirakov} or Subsection \ref{sec:p_diferente_q} ahead, one shows with a standard bootstrap that the weak solutions of \eqref{eq:main_system} are classical solutions, so that the numbers $c_s(\Omega)$ are independent of the particular choice of $s$. Throughout this paper we will denote the ground state level simply by $c(\Omega)$.

In the case of a single equation or when dealing with gradient systems, one possible characterization of the ground state level is through the minimization of the energy functional on the so called \emph{Nehari manifold}. For Hamiltonian systems, this turns out to be \emph{unsuccessful}, and we illustrate this fact in the superlinear case $pq>1$.

If $(u,v)$ is a weak solution of \eqref{eq:main_system_particularcase}, we have $\int_\Omega |u|^{p+1}\, dx=\int_\Omega |v|^{q+1}\, dx$. Moreover, using the fact that $\G_s'(u,v)(u,v)=0$, we infer that
\[
2\int_{\Omega}\langle \nabla u, \nabla v\rangle dx  = \int_{\Omega} |u|^{p+1} dx + \int_{\Omega} |v|^{q+1} dx.
\]
Assuming $pq> 1$, we deduce that for a nontrivial weak solution $(u,v)$ of \eqref{eq:main_system_particularcase}, we have
\begin{align*}
\G_s(u,v) &= \frac{p-1}{2(p+1)} \int_\Omega  |u|^{p+1} dx + \frac{q-1}{2(q+1)} \int_\Omega  |v|^{q+1} dx\\
		&=\frac{pq-1}{(p+1)(q+1)}\int_\Omega |u|^{p+1}\, dx>0.
\end{align*}
If \eqref{eq:(p,q)_for_W^1,s_0} holds, we have all the required compactness to prove that $c_{s}(\Omega)$ is achieved
as soon as one can prove that $\G_{s}$ has at least one critical point and we therefore deduce that $c(\Omega)=c_{s}(\Omega)>0$. 

Next, we define the Nehari manifold $\Ncal_{\G_s}$ as usual by 
\[
\Ncal_{\G_s}:=\{(u,v)\in W^{1,s}_{0}(B)\times W^{1,\frac{s}{s-1}}_{0}(B)\mid (u,v) \ne (0,0)\ \text{ and }\ \G_s'(u,v)(u,v) = 0\}.
\]
In contrast with the case of a single equation or gradient systems, the origin $(0,0)$ turns out to be adherent to $\Ncal_{\G_s}$. This means $\inf_{\Ncal_{\G_s}}\G_s\le 0$, and therefore $\inf_{\Ncal_{\G_s}} \G_s$ \emph{cannot be a critical level associated to a nontrivial critical point}! Since this fact seems not so well known by the community and has been misused, we state it as a proposition for completeness.
\begin{proposition}\label{nehari nao funciona}
Assume \eqref{eq:(p,q)_for_W^1,s_0} and $pq> 1$ hold. Then $(0,0)$ is an adherent point of $\Ncal_{\G_s}$ and 
$\inf_{\Ncal_{\G_s}}\G_s \leq 0$.
\end{proposition}
\begin{proof}
Suppose for instance that $q >1$. Let $u\in C^\infty_\textrm{c}(\Omega)$ be a positive function, and $\lambda,t>0$. Then $\G_s'(tu,t\lambda u)(tu,t\lambda u)=0$ means
\[
\varphi_\lambda(t):=t^{p-1}\|u\|_{p+1}^{p+1}+t^{q-1} \lambda^{q+1}\|u\|^{q+1}_{q+1}=2\lambda \|\nabla u\|_2^2.
\]
Observe that $\varphi_\lambda(t)\to +\infty$ as $t\to +\infty$, whatever $\lambda>0$ and $p>0$ are fixed.

\medbreak

\noindent{\it Claim : for each $\lambda>0$ large enough, there exists a unique $t_\lambda>0$ such that
\begin{equation}\label{eq:tu,tlambdau_in_N}
\varphi_\lambda(t_\lambda)=2\lambda\|\nabla u\|_2^2,\ \text{ or equivalently},\  (t_\lambda u,t_\lambda \lambda u)\in \Ncal_{\G_s}.
\end{equation}}%
If $p>1$, then for each $\lambda>0$ we have $\varphi_\lambda(0)=0$. Thus the claim follows easily from the continuity of $\varphi_\lambda$. For $p=1$ one can argue in an analogous way for each $\lambda$ satisfying $2\lambda\|\nabla u\|_2^2>\|u\|_{p+1}^{p+1}$, since in such case $\varphi_\lambda(0)<2\lambda \|\nabla u\|_2^2$. Finally, when $p<1$, we can take $\lambda>0$ such that
\[
2\lambda^{\frac{pq-1}{q-p}}\|\nabla u\|_2^2>\|u\|_{p+1}^{p+1}+\|v\|_{q+1}^{q+1},
\]
which leads to $\varphi_{\lambda}(\lambda^{-(q+1)/(q-p)})<2\lambda \|\nabla u\|_2^2$, and we conclude as before.

%

\medbreak

\noindent{\it Conclusion :} the identity \eqref{eq:tu,tlambdau_in_N}  implies in particular that
\[
0<(t_\lambda \lambda)^{q-1}\|u\|_{q+1}^{q+1}\leq \frac{2\|\nabla u\|_2^2}{\lambda}\to 0 \qquad \text{ as } \lambda \to +\infty.
\]
Hence $t_\lambda \lambda\to 0$, and $\|(t_\lambda u,t_\lambda \lambda u)\|_{W^{1,s}_0 \times W^{1, \frac{s}{s-1}}_0}\to 0$ as $\lambda \to +\infty$, so that the proof is complete.

\end{proof}

In Section \ref{sec:Reduced Functional}, we will define a suitable Nehari type set (of infinite codimension). Namely, by imposing the relations 
$$\G_{s}'(u,v)(u+\phi,v-\phi)=0$$
for every direction $\phi$, we will recover that the minimum on such a set corresponds to the ground energy level. A different route will also be considered in Sections \ref{sec:DualMethod} and \ref{sec:QuartaOrdem} where we provide two other ways to recover a characterization of the ground state level as the minimum on a standard Nehari manifold.

\subsection{Using fractional Sobolev spaces} \label{subset:Fractional}

In this section we describe the variational approach based on the use of fractional Sobolev spaces, following \cite{deFigueiredoFelmer,HulshofvanderVorst}. We recall that, in order to simplify the computations, we still assume $H(u,v)=|u|^{p+1}/(p+1)+|v|^{q+1}/(q+1)$, and refer to the above mentioned papers for more general statements. The following approach will yield an existence result for $(p,q)$ such that
\begin{equation}\label{eq:p_and_q_for_E^s_approach}
p,q>0, \quad 1>\frac{1}{p+1}+\frac{1}{q+1}>\frac{N-2}{N},\quad p(N-4),q(N-4)< N+4. \tag{H2}
\end{equation}
Recall that $1/(p+1)+1/(q+1)<1$ is equivalent to $pq>1$, which corresponds to the notion of superlinearity in the context of elliptic Hamiltonian systems.

\begin{center}
\includegraphics[scale=.26]{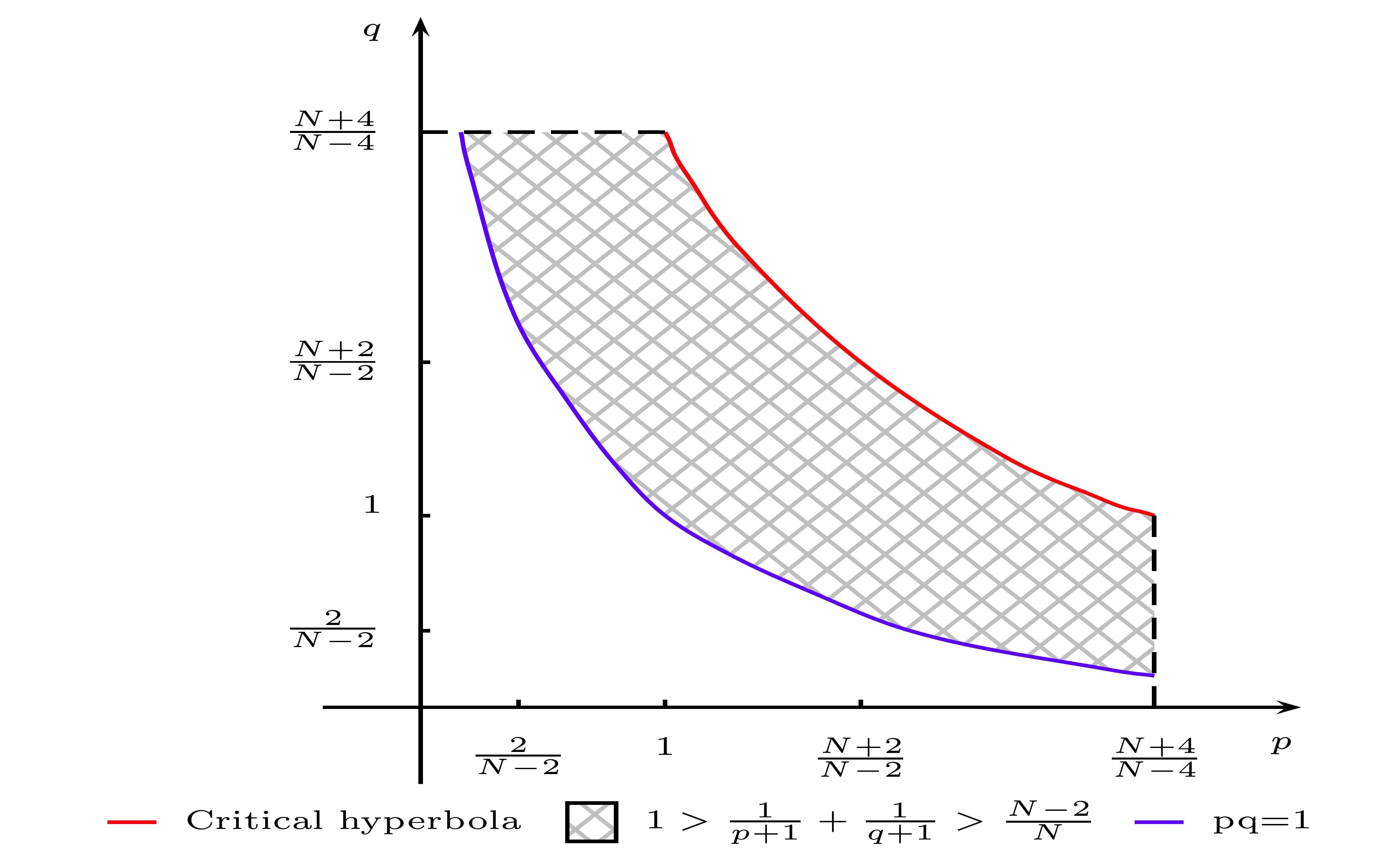}
\end{center}
\vspace{-.3cm}

\begin{remark}The references \cite{deFigueiredoFelmer,HulshofvanderVorst} were published contemporaneously, and the techniques share much similarities. However, in \cite{HulshofvanderVorst} the proof is done for the more restrictive case
\[
p,q>1, \qquad \frac{1}{p+1}+\frac{1}{q+1}>\frac{N-2}{N}
\]
(both nonlinearities need to be superlinear), while the observation that one can actually treat the more general case \eqref{eq:p_and_q_for_E^s_approach} is done in \cite{deFigueiredoFelmer}. For more precise details check the proof of {\rm Theorem \ref{thm:Existence_with_E^s}} ahead. 
\end{remark}
For $N\geq 3$, the fact that $(p,q)$ lies below the critical hyperbola may yield that (for instance) $q<2^*-1<p$. In such a case, we cannot define the action functional \eqref{eq:usual_action_functional} in $H^1_0(\Omega)\times H^1_0(\Omega)$, and the idea is to impose a priori more regularity on $u$ and less on $v$, keeping at the same time an Hilbertian framework. Having this in mind, let us introduce the fractional Sobolev spaces $E^s(\Omega)$. 

Let $(\phi_n)_n$ be the sequence of $L^2$--normalized eigenfunctions of $(-\Delta,H^1_0(\Omega))$, with corresponding eigenvalues $(\lambda_n)_n$. It is well know that each $u\in L^2(\Omega)$ coincides with its Fourier series 
$$u=\sum_{n=1}^\infty a_n \phi_n,$$ 
with $a_n:=\int_\Omega u\phi_n\, dx$. For $s>0$, we can therefore define the operator $A^s:=(-\Delta)^{s/2}: E^s(\Omega)\to L^2(\Omega)$, where 
\begin{equation}\label{eq:Es}
 E^s(\Omega)=\left\{u=\sum_{n=1}^\infty a_n \phi_n\in L^2(\Omega):\ \|u\|_{E^s(\Omega)}^2:=\sum_{n=1}^\infty \lambda_n^s a_n^2<\infty   \right\},
\end{equation}
and 
\[
A^su=A^s\left( \sum_{n=1}^\infty a_n \phi_n \right):=\sum_{n=1}^\infty \lambda_n^{s/2}a_n \phi_n.
\]
We endow $E^s(\Omega)$ with the inner product
\[
\langle u,v\rangle_{E^s(\Omega)}:=\int_\Omega A^s u A^s v\, dx,\qquad \forall u,v\in E^s(\Omega),
\]
so that $E^s(\Omega)$ is an Hilbert space with the Hilbertian norm $\|u\|_{E^s}=\|A^s u\|_2$. We denote by $A^{-s}$ the inverse of the operator $A^s$. Observe that  $E^2(\Omega)=H^2(\Omega)\cap H^1_0(\Omega)$ and $A^2 =-\Delta$, while $E^1(\Omega)=H^1_0(\Omega)$.

Suppose for the moment that 
\begin{equation}\label{eq:p,q_auxiliary_E^s}
p,q>0, \quad \frac{1}{p+1}+\frac{1}{q+1}>\frac{N-2}{N}, \quad p(N-4), q(N-4)<N+4.
\end{equation} 
Then
\[
0<N\left(\frac{1}{q+1}-\frac{N-4}{2N}\right),\qquad N\left(\frac{1}{2}-\frac{1}{p+1}\right)<2,
\]
and
\[
N\left(\frac{1}{2}-\frac{1}{p+1}\right)<N\left(\frac{1}{q+1}-\frac{N-4}{2N}\right),
\]
whence we can take $0<s<2$ such that
\[
N\left(\frac{1}{2}-\frac{1}{p+1}\right)<s<N\left(\frac{1}{q+1}-\frac{N-4}{2N}\right).
\]
This last statement is equivalent to
\[
(p+1)(N-2s)<2N,\qquad (q+1)(N-2(2-s))<2N.
\]
Thus, under this choice, we have the compact embeddings (see \cite[Theorem 1.1]{deFigueiredoFelmer}):
\[
E^s(\Omega)\hookrightarrow L^{p+1}(\Omega),\qquad E^{2-s}(\Omega)\hookrightarrow L^{q+1}(\Omega).
\]
To simplify the notation, we set $t=2-s$ and $E_s:=E^s(\Omega)\times E^t(\Omega)$. The previous embeddings imply that the energy functional
\begin{equation}\label{eq:the_use_of_s}
\I_s:E_s \to \R,\qquad  \I_s(u,v)=\int_\Omega A^s u A^{t}v\, dx-\int_\Omega H(u,v)\, dx
\end{equation}
is a well defined $C^1$--functional for $(p,q)$ as in \eqref{eq:p,q_auxiliary_E^s}. Then $(u,v)\in E_s$ is a critical point of $I_s$ if and only if
\[
 \I'_s(u,v)(\varphi,\psi)=\int_\Omega (A^s u A^{t}\psi+A^s\varphi A^{t}v-H_{u}(u,v)\varphi-H_{v}(u,v)\psi)\, dx=0,
\] 
for every $(\varphi, \psi) \in E_s$ that is, $(u,v)$ is a solution of
\[
-\Delta u=H_{v}(u,v) \text{ in } E^{-t}(\Omega),\qquad -\Delta v=H_{u}(u,v) \text{ in } E^{-s}(\Omega).
\]
This is the notion of \emph{weak solution} in this context. It is proved in \cite[Theorem 1.2]{deFigueiredoFelmer} that weak solutions are \emph{strong solutions}, in the sense that
\[
u\in W^{2,\frac{q+1}{q}}(\Omega)\cap W^{1,\frac{q+1}{q}}_0(\Omega),\qquad v\in W^{2,\frac{p+1}{p}}(\Omega)\cap W^{1,\frac{p+1}{p}}_0(\Omega)
\]
and they satisfy the system \eqref{eq:main_system} pointwise for a.e. $x\in \Omega$. By using a bootstrap argument and elliptic regularity theory \cite{GilbargTrudinger}, see also Subsection \ref{sec:p_diferente_q} ahead, one proves in a standard way that weak solutions are in fact classical solutions.  

In order to obtain the existence of nontrivial solutions under \eqref{eq:p_and_q_for_E^s_approach} we need some preliminaries. First observe that the functional $\I_s$ may be written in the form 
\begin{equation*}\label{eq:functional:alternative_form_with_L}
\I_s(u,v)=\frac{1}{2}\langle L_s(u,v),(u,v)\rangle_{E_s}-\int_\Omega H(u,v)\, dx,
\end{equation*}
where $L_s:E_s\to E_s$ is the self-adjoint bounded linear operator defined by the condition
\begin{equation*}
\langle L_s(u,v),(\varphi,\psi) \rangle_{E}=\int_\Omega (A^s uA^{t}\psi+A^s\varphi A^{t} v)\, dx,
\end{equation*}
having the explicit formula
\[
L_s(u,v)=(A^{-s}A^{t}v,A^{-t}A^su).
\]
The space $E_s$ decomposes in $E_s^+\oplus E_s^-$, with
\[
E_s^+=\{(u,A^{-t}A^s u):\  u\in E^s(\Omega)\}, \qquad  E_s^{-}=\{(u,-A^{-t}A^s u):\ u\in E^s(\Omega)\},
\]
writing $(u,v) \in E_s$ as
\[
(u,v) = \left( \frac{u + A^{-s}A^t v}{2}, \frac{v + A^{-t}A^{s}u}{2}\right) + \left( \frac{u - A^{-s}A^tv}{2}, \frac{v - A^{-t}A^su}{2} \right).
\]
Observe that both $E_s^+$ and $E_s^-$ are infinite dimensional, and the quadratic part $\int_\Omega A^s uA^t v\, dx$ is positive on $E_s^+$, negative on $E_s^-$. In the literature, this type of geometry is referred to as \emph{strongly indefinite}. 

\begin{lemma}\label{I_s_satisfies_PS}
Under \eqref{eq:p_and_q_for_E^s_approach}, the functional $\I_s$ satisfies the Palais-Smale condition.
\end{lemma}
\begin{proof}
 Let $(u_n,v_n)\in E_s$ be so that $\I_s(u_n,v_n)$ is bounded and $\I_s'(u_n,v_n)\to 0$. Then we have
\begin{multline*}
0<\frac{pq-1}{p+q+2}\int_\Omega\left(\frac{|u_n|^{p+1}}{p+1}+\frac{|v_n|^{q+1}}{q+1}\right)\, dx\\
=\I_s(u_n,v_n)-\I_s'(u_n,v_n)\left(\frac{q+1}{p+q+2}u_n,\frac{p+1}{p+q+2}v_n\right)\leq C+ \epsilon_n\|(u_n,v_n)\|_{E_s},
\end{multline*}
with $0 < \eps_n \rightarrow 0$ as $n\to \infty$. Thus
\begin{multline*}
\|(u_n,v_n)\|^2_{E_s} = \int_\Omega (|u_n|^pA^{-s}A^tv_n+|v_n|^qA^{-t}A^su_n)\, dx\\
+\I_s'(u_n,v_n)(A^{-s}A^tv_n,A^{-t}A^su_n)\\
\leq\left(\int_\Omega |u_n|^{p+1}\, dx\right)^\frac{p}{p+1}\|A^{-s}A^t v_n\|_{p+1}\\
+\left(\int_\Omega |v_n|^{q+1}\, dx\right)^\frac{q}{q+1}\|A^{-t}A^s u_n\|_{q+1} +\epsilon_n\|(u_n,v_n)\|_{E_s}\\
\leq  C\|(u_n,v_n)\|_{E_s} \left(C+ \epsilon_n\|(u_n,v_n)\|_{E_s} \right)^{\frac{p}{p+1}} \\
+  C\|(u_n,v_n)\|_{E_s} \left(C+ \epsilon_n\|(u_n,v_n)\|_{E_s} \right)^{\frac{q}{q+1}} +\epsilon_n\|(u_n,v_n)\|_{E_s} \\
\leq C\|(u_n,v_n)\|_{E_s}+ \epsilon_n \|(u_n,v_n)\|_{E_s}+ \epsilon_n \|(u_n,v_n)\|_{E_s}^2.
\end{multline*}
Since $\{(u_n,v_n)\}_n$ is bounded in $E_s$, up to a subsequence, $(u_n,v_n)$ weakly converges to some $(u,v)\in E_s$. The convergence is actually strong and this can be deduced from a careful analysis of the convergence 
\[
\I_s'(u_n,v_n)(A^{-s}A^t(v_n-v),A^{-t}A^s(u_n-u))\to 0.
\]
\end{proof}

\begin{theorem}\label{thm:Existence_with_E^s}
Take $(p,q)$ satisfying \eqref{eq:p_and_q_for_E^s_approach}. Then \eqref{eq:main_system} admits a nontrivial classical solution.
\end{theorem}

\begin{proof}[Sketch of the proof]
We follow \cite{deFigueiredoFelmer}, where more general nonlinearities are considered.

\medbreak

\noindent{\it Step 1.} Definition of the set $S$. Given $\rho>0$, by using the embeddings \eqref{eq:W^s_0_embeddings}, we have that, for $\|u\|_{E^s}=\rho$,
\[
\I_s(\rho^qu,\rho^p A^{-t}A^su)\geq \rho^{p+q+2}-C\rho^{(p+1)(q+1)}=\rho^{p+q+2}(1-C\rho^{pq-1})
\]
for some $C>0$ independent of $\rho$. Define the set
\[
S=S_\rho=\{(\rho^qu,\rho^pA^{-t}A^su):\ \|u\|_{E^s}=\rho\}.
\]
Then there exists a constant $\alpha>0$ such that $\I_s|_{S}\geq \alpha>0$ whenever $\rho>0$ is taken sufficiently small.

\medbreak

\noindent{\it Step 2.} Definition of the set $Q$. Let $w$ be any eigenfunction of $(-\Delta, H^1_0(\Omega))$. Given constants $\sigma,M>0$, define the set
\[
Q=Q_{\sigma,M}=\{(\sigma^q(tw+\phi),\sigma^pA^{-t}A^s(tw-\phi)): \ 0\leq t\leq \sigma,\ 0\leq \|\phi\|_{E^s}\leq M\}.
\]
Observing that
\begin{multline*}
\I_s(\sigma^q(tw+\phi),\sigma^pA^{-t}A^s(tw-\phi))=\sigma^{p+q}t^2 \|w\|_{E^s}^2-\sigma^{p+q}\|\phi\|_{E^s}^2\\
-\frac{\sigma^{q(p+1)}}{p+1}\int_\Omega |tw+\phi|^{p+1}\, dx-\frac{\sigma^{p(q+1)}}{q+1}\int_\Omega |A^{-t}A^s(tw-\phi)|^{q+1}\, dx,
\end{multline*}
and $pq>1$, it can be proved that $\I_s|_{\partial Q}\leq 0$ for sufficiently large  $\sigma,M>0$ (cf. \cite[Section 3]{deFigueiredoFelmer}).

\medbreak

\noindent{\it Step 3.} Conclusion. It can be proved that $Q$ and $S$ \emph{link}\footnote{In the sense of equation (3.7) in \cite{Felmer}.}, and one can apply the linking theorem of Benci and Rabinowitz \cite[Theorem 0.1]{BenciRabinowitz} in a version due to Felmer \cite[Theorem 3.1]{Felmer} (see also \cite{HulshofvanderVorst} where, under the additional assumption that $p,q>1$, \cite[Theorem 0.1]{BenciRabinowitz} is applied directly with $S:=\partial B_\rho(0)\cap E^+$, $Q:=\R(w,A^{-t}A^sw)\oplus E^-$).
\end{proof}

Once we have the existence of at least one solution, the existence of a solution with least energy follows from a compactness argument.

\begin{corollary}\label{coro:_least_energy_achieved}
Take $(p,q)$ satisfying \eqref{eq:p_and_q_for_E^s_approach}. Then the ground state level
\[
c(\Omega)=\inf\{\I_s(u,v):\ (u,v)\in E_s,\ (u,v)\neq 0,\ \I_s'(u,v)=0\}
\]
is achieved and positive.
\end{corollary}

\begin{proof}
First, observe that arguing as in the previous subsection, if $(u,v)\neq (0,0)$ and $\I_s'(u,v)=0$, we obtain
\begin{equation}\label{eq:groundstate_c>0}
\int_\Omega |u|^{p+1}\, dx=\int_\Omega |v|^{q+1}\, dx \ \ \text{ and } \ \ \I_s(u,v)=\frac{pq-1}{(p+1)(q+1)}\int_\Omega|u|^{p+1}\, dx>0.
\end{equation}
Hence we infer that $c(\Omega)\geq 0$. Suppose without loss of generality that $p\geq q$, so that $p>1$ does hold. From the identity $\I_s'(u,v)(0,A^{-t}A^su)=0$, we infer that
\begin{align*}
\|u\|_{E^s}^2&=\int_\Omega |v|^{q-1}vA^{-t}A^s u\, dx \leq \left(\int_\Omega |v|^{q+1}\, dx\right)^\frac{q}{q+1}\|A^{-t}A^s u\|_{q+1}\\
						&\leq C\|u\|_{p+1}^\frac{q(p+1)}{q+1}\|A^{-t}A^su\|_{q+1}\leq C \|u\|_{E^s}^{\frac{q(p+1)}{q+1}+1}
\end{align*}
for some $C>0$. In particular, there exists $\kappa>0$ such that for every nontrivial solution $(u,v)$, 
\begin{equation}\label{eq:bound_from_bellow_inE^s}
\|u\|_{E^s}^\frac{pq-1}{q+1}\geq \kappa 
\end{equation}
Now take a minimizing sequence $(u_n,v_n)$ at the level $c(\Omega)$, that is
\[
\I_s(u_n,v_n)\to c(\Omega),\qquad I_s'(u_n,v_n)=0.
\]
Since $\I_s$ satisfies the Palais-Smale condition (see Lemma \ref{I_s_satisfies_PS}), we have (up to a subsequence) $(u_n,v_n)\to (u,v)$ in $E^s(\Omega)\times E^t(\Omega)$. Moreover, \eqref{eq:bound_from_bellow_inE^s} implies that $(u,v)\neq (0,0)$, and it is a critical point of $\I_s$, at the critical level $c(\Omega)$. From \eqref{eq:groundstate_c>0}, we have that $c(\Omega)>0$.
\end{proof}

We end this section by observing that, at the price of dealing with modified nonlinearities, we could also have worked with the Sobolev space $H^1_0(\Omega)\times H^1_0(\Omega)$. In fact, introducing the isometric isomorphism 
$$B_s: H^1_0(\Omega)\to E^{s}(\Omega) : u\mapsto B_s(u) =A^{-s}\circ A^1(u),$$ 
the functional $\widetilde \I_s: H^1_0(\Omega)\times H^1_0(\Omega)\to \IR$ defined by 
\[
\widetilde \I_s(u,v)=\int_{B}\langle \nabla u,\nabla v\rangle dx-\int_{B}H(B_su,B_{2-s}v) dx
\]
is a well defined $C^1$-functional for $(p,q)$ satisfying \eqref{eq:p_and_q_for_E^s_approach}. Then $(u,v)\in H^1_0(\Omega)\times H^1_0(\Omega)$ is a critical point of $\widetilde \I_s$ if and only if
\begin{multline*}
\widetilde \I_{s}'(u,v)(\varphi,\psi)=\int_{\Omega} \left(\langle \nabla u, \nabla \psi\rangle + \langle \nabla v, \nabla \varphi\rangle \right)\, dx \\ - \int_{\Omega} \left( H_{u}(B_su,B_{2-s}v)B_s\varphi -H_{v}(B_su,B_{2-s}v)B_{2-s}\psi\right)\, dx  = 0,
\end{multline*}%
for every $(\varphi,\psi) \in H^{1}_{0}(\Omega)\times H^{1}_{0}(\Omega)$, or, equivalently, if and only if
$(B_su,B_{2-s}v)\in E^{s}(\Omega)\times E^{2-s}(\Omega)$ is a critical point of $\I_{s}$. This slightly different strategy was introduced and used in \cite[Section 5]{BonheureRamos} as it is convient to work in a framework where $u$ and $v$ belong to the same functional space, especially when one wants to use the reduction approach of Section \ref{sec:Reduced Functional}.
\begin{remark}
One downsize of the approach presented in this subsection is that it does not allow to treat the cases where $p(N-4)\geq N+4$ or $q(N-4)\geq N+4$, as we cannot find  $0<s<2$ so that the functional $\I_s$ is well defined on $E^s(\Omega)\times E^t(\Omega)$, with $t = 2-s$. Moreover, by using directly this approach, it is not clear how to show that ground state solutions are signed nor if they enjoy symmetry properties. These questions will be considered in the following two sections via other methods.
\end{remark}

\section{The dual method}\label{sec:DualMethod}
In this section we describe some applications of the so called {\it dual variational principle} of Clarke and Ekeland \cite{clarke, ClarkeEkeland} to prove the existence and to study qualitative properties of ground state solutions to the system
\eqref{eq:main_system}. 
%
We still deal with pure power nonlinearities as in the model Hamiltonian \eqref{Hamiltonian}, assuming that the couple $(p,q)$ satisfies
\begin{equation}\label{eq:(p,q)_for_Ederson_part}
p, q> 0, \qquad 1 > \frac{1}{p+1} + \frac{1}{q+1} > \frac{N-2}{N}. \tag{H3}
\end{equation}

\begin{center}
\includegraphics[scale=.26]{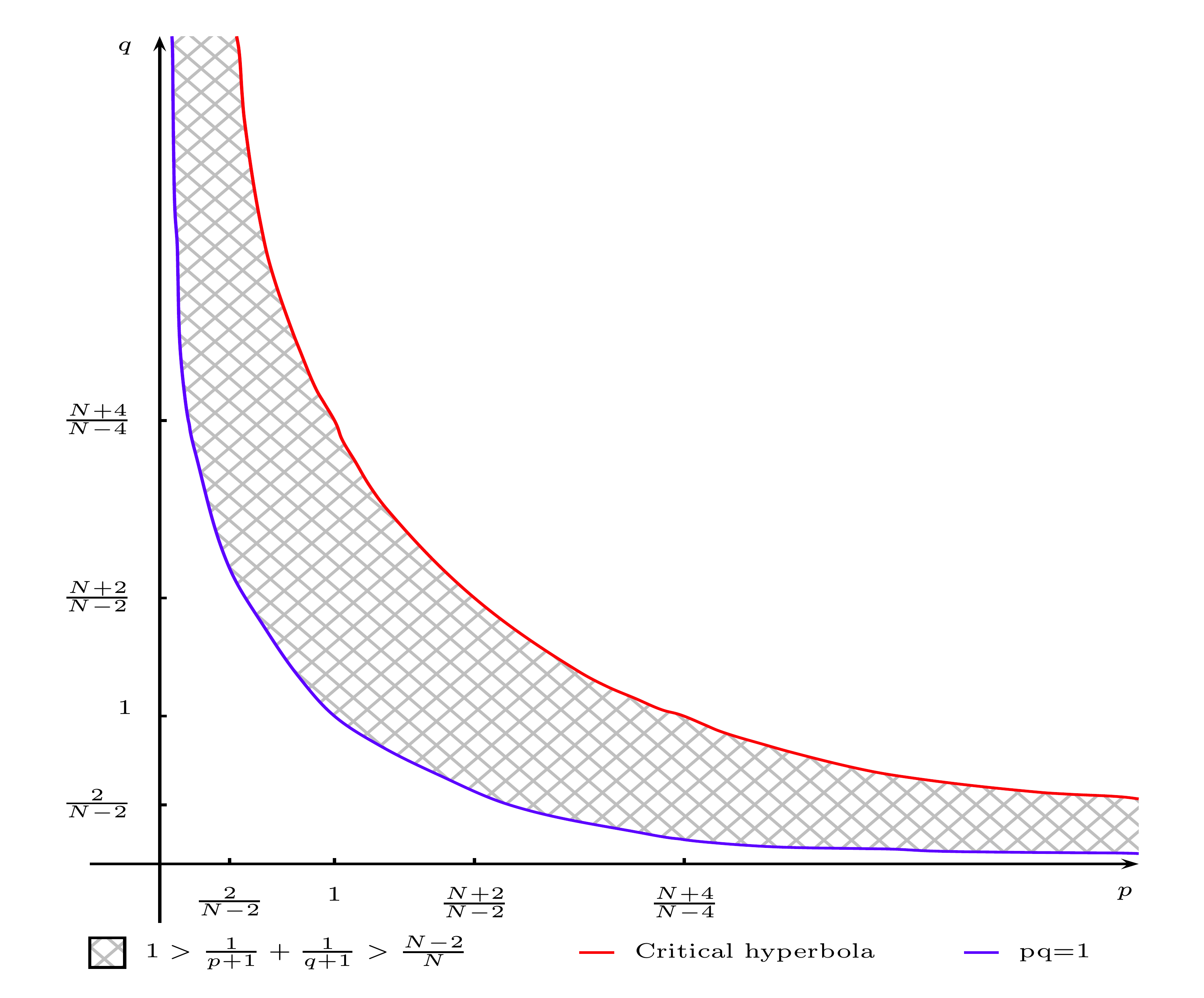}
\end{center}
\vspace{-.3cm}

As far as we know, the application of the dual variational principle to strongly coupled elliptic systems such as \eqref{eq:main_system} goes back to Clément and van der Vorst \cite{ClementvanderVorst}. One of the features in \cite{ClementvanderVorst} is that the authors assumed the mere hypothesis \eqref{eq:(p,q)_for_Ederson_part}, which includes cases with $p\leq1$ or $q\leq 1$. We mention that Alves and Soares \cite{AlvesSoares} also applied the dual variational method to treat singularly perturbed systems\footnote{Compare \eqref{system claudianor} with \cite[eq. (1.3)]{AlvesSoares}. In fact, in order to apply the dual variational method to \eqref{system claudianor}, the two potentials on the left hand sides of \eqref{system claudianor} must be equal as follows from \cite[line 5 p. 114]{AlvesSoares}. 
} of the form
\begin{equation}\label{system claudianor}
\left\{
\begin{array}{l}
- \eps^2 \Delta u + V(x) u = |v|^{q-1}v \quad \text{in} \quad \R^N,\\
- \eps^2 \Delta v + V(x) v = |u|^{p-1}u \quad \text{in} \quad \R^N,
\end{array}
\right.
\end{equation}
assuming \eqref{eq:(p,q)_for_Ederson_part} and the extra assumptions $p> 1$ and $q>1$; cf. \cite{SirakovSoares} for a more general system with a similar superlinear assumption. In this section we follow some of the ideas in \cite{ClementvanderVorst}, but we present a simpler approach. Our main concern is to show how the dual variational method transforms the strongly indefinite structure of the functional associated to problem \eqref{eq:main_system} into a problem whose functional displays a mountain pass geometry. We simplify the arguments of \cite{ClementvanderVorst}, especially with respect to compactness. In \cite{ClementvanderVorst}, the authors use a variant of the mountain pass theorem of Ambrosetti and Rabinowitz \cite{AmbrosettiRabinowitz} due to Bartolo, Benci and Fortunato \cite{BartoloBenciFortunato,BenciFortunato}, which holds for functionals that satisfy the so called Cerami condition, see e.g. \cite{Cerami}. Instead, we prove that the functional associated to this approach indeed satisfies the Palais-Smale condition and so we apply the standard version of the mountain pass theorem \cite{AmbrosettiRabinowitz}. Finally, we stress that more general systems can be considered in this framework and we refer to \cite[Theorem 1.1]{BenciFortunato} and \cite[Theorem 3.2]{ClementvanderVorst}. 

We now illustrate how problem \eqref{eq:main_system} can be treated in a dual formulation. For instance, consider the slightly more general problem
\begin{equation}\label{eq:more general system}
\left\{
\begin{array}{l}
 - \Delta u + c(x) u = |v|^{q-1}v \quad \text{in} \quad \Omega,\\
 -\Delta v + c(x) v = |u|^{p-1}u \quad \text{in} \quad \Omega,\\
 u, v = 0 \quad \text{on} \quad \partial \Omega,
 \end{array}
 \right.
\end{equation}
where the function $c$ is such that $L := -\Delta + c(x)I : W^{2,r}(\Omega)\cap W^{1,r}_0(\Omega) \rt L^r(\Omega)$ is a topological isomorphism for every $1 < r < \infty$. For example, this holds when $c\in C^1(\overline{\Omega})$ is nonnegative and of course one could assume more general conditions. 

We illustrate this method by using it in order to prove that:
\begin{enumerate}[(i)]
\item there exists a ground state solution to \eqref{eq:more general system};
\item any ground state solution $(u,v)$ of \eqref{eq:more general system} is signed, that is, either $(u^+,v^+)=(0,0)$ or  $(u^-,v^-)=(0,0)$ (or, equivalently, $uv>0$ in $\Omega$);
\item any ground state solution of \eqref{eq:more general system} is radially symmetry in case $\Omega$ is a ball and $c\equiv0$. 
\end{enumerate}

For that, we define $\phi_p, \phi_q: \R \rt \R$ by
\[
\phi_p(t) = |t|^{p-1}t, \quad \phi_q(t) = |t|^{q-1}t, \quad t \in \R.
\]
The dual method consists in taking the inverse of the operator $L$, rewriting the system \eqref{eq:more general system} as
\[
u=L^{-1}(|v|^{q-1}v)=L^{-1}(\phi_q(v)),\qquad v=L^{-1}(|u|^{p-1}u)=L^{-1}(\phi_p(u))
\]
and introducing the new variables $f=|u|^{p-1}u=\phi_p(u)=Lv$, $g=|v|^{q-1}v=\phi_q(v)=Lu$, leading to the system
\begin{equation}\label{euler-lagrange}
L^{-1}f = |g|^{\frac{1}{q}-1}g=\phi_q^{-1}(g), \qquad L^{-1}g = |f|^{\frac{1}{p}-1}f=\phi_p^{-1}(f).
\end{equation}

Then, fixing the notation $K=L^{-1}$, we define the operator $T:X\to X^*$, for appropriate $X$, through the identity
\begin{equation}\label{operatorT}
\langle T(f_1,g_1), (f_2,g_2)\rangle = \int_\Omega \left(f_2 K g_1 + g_2  K f_1\right)dx.
\end{equation}
Since $\langle T(f_1,g_1), (f_2,g_2)\rangle= \langle T(f_2,g_2), (f_1,g_1)\rangle$, it follows that the equations in  \eqref{euler-lagrange} appear as the Euler-Lagrange equations associated to the action functional
\begin{equation}\label{phi}
\Phi(f,g) = \int_\Omega \left(\frac{p}{p+1}|f|^{\frac{p+1}{p}} + \frac{q}{q+1}|g|^{\frac{q+1}{q}}\right) dx -\dfrac{1}{2}\langle T(f,g), (f,g)\rangle.
\end{equation}
We will show that $(f,g)$ is a critical point of $\Phi$ defined in an adequate space if, and only if, $(u,v) = (Kg, Kf) = ( \phi_ p^{-1}(f), \phi_ q^{-1}(g))$ is a classical solution of \eqref{eq:more general system}, and that the least energy level associated to \eqref{eq:more general system} corresponds to the mountain pass critical level of $\Phi$. Thus, the dual method allows one to avoid the strongly indefinite character that is present in the direct approaches.

\subsection{Variational framework}

Throughout this section we assume that \eqref{eq:(p,q)_for_Ederson_part} holds and we set 
\[
X:=L^{\frac{p+1}{p}}(\Omega)\times L^{\frac{q+1}{q}}(\Omega).
\]
Assuming that $c \in C^1(\overline{\Omega})$ is nonnegative, we consider the linear operator
\[
L := -\Delta + c(x)I,
\] 
and recall that we denote its inverse by $K$. We first infer from \eqref{eq:(p,q)_for_Ederson_part}, the $W^{2,r}$-regularity for second order elliptic operators as in \cite[Theorem 9.15]{GilbargTrudinger}, and the classical Sobolev embeddings, that the operator $T: X \rt X^*$ defined through \eqref{operatorT} is a linear compact operator.
Moreover $\Phi \in C^1(X, \R)$ and $\Phi ' = \Psi - T$, where
\[
\langle \Psi(f_1,g_1), (f_2, g_2) \rangle = \int_\Omega \left(|f_1|^{\frac{1}{p}-1}f_1f_2 + |g_1|^{\frac{1}{q}-1}g_1g_2\right)dx,
\]
for every  $(f_1, g_1), (f_2, g_2) \in X$. Then, from the classical Riesz representation theorem for Lebesgue spaces, we see that $\Psi: X \rt X^*$ is a homeomorphism. 

Throughout this section we will constantly use the identity
\[
\int_\Omega f K g dx = \int_\Omega g K f dx,
\]
which is the consequence of a simple integration by parts.

\begin{proposition}\label{propositionregularitydual}
Assume that \eqref{eq:(p,q)_for_Ederson_part} holds. Then $(f,g)\in X$ is a critical point of $\Phi$ if, and only if, $(u,v) = (Kg, Kf) = ( \phi_ p^{-1}(f), \phi_ q^{-1}(g))$ is a classical $C^2(\overline{\Omega})$-solution of \eqref{eq:more general system}.
\end{proposition}
\begin{proof}
Let $(f,g) \in X$ be a critical point of $\Phi$. Then, for every $(\varphi, \psi) \in X$, we have
\begin{equation}\label{criticalpointdual}
0 = \int_{\Omega}\left( |f|^{\frac{1}{p}-1}f\varphi + |g|^{\frac{1}{q}-1}g\psi\right) dx - \int_{\Omega}\left( \varphi Kg + \psi Kf\right)\, dx.
\end{equation}
Define $u = Kg$ and $v = Kf$. From the standard $W^{2,r}$-regularity for the second order elliptic operator $L$, see for instance \cite[Theorem 9.15 and Lemma 9.17]{GilbargTrudinger}, it follows that $u \in W^{2, \frac{q+1}{q}}(\Omega) \cap W^{1,\frac{q+1}{q}}_0(\Omega)$, $v \in W^{2, \frac{p+1}{p}}(\Omega) \cap W^{1,\frac{p+1}{p}}_0(\Omega)$ and therefore \eqref{criticalpointdual} reads
\[
0 = \int_\Omega \left(|Lv|^{\frac{1}{p}-1}Lv \, \varphi + |Lu|^{\frac{1}{q}-1}Lu \, \psi\right) dx - \int_\Omega \left(\varphi u + \psi v \right)dx \quad \forall \, (\varphi, \psi) \in X.
\]
From this identity, we deduce that $(u,v)$ is a strong solution of
\begin{equation}\label{systemregularity}
Lv = |u|^{p-1}u, \quad Lu = |v|^{q-1}v \quad \text{in} \: \Omega, \qquad u,v=0 \quad \text{ on }  \partial\Omega.
\end{equation}
Then, we can proceed as in \cite[Theorem 1.1]{Ederson2008} to conclude that $(u,v)$ is a classical $C^2(\overline{\Omega})$ solution of \eqref{systemregularity}.

The converse implication is even easier and we omit it here.
\end{proof}

Now we recall a result that helps proving that $\Phi$ satisfies the Palais-Smale condition.

\begin{lemma}[{\cite[Lemma 3.1]{Ederson2008}}]\label{lemmaps}
Let $X$ be a Banach space and $\Phi \in C^1(X, \R )$ be such that
\begin{itemize}
\item[(i)] any Palais-Smale sequence of $\Phi$ is bounded;
\item[(ii)] for all $u \in X$, 
\begin{equation}\label{decomposition}
\Phi'(u) = \Psi(u) + S(u),
\end{equation}
\nin where $\Psi:X \rt X^*$ is a homeomorphism and $S: X \rt X^*$ is a continuous map such that  $(S(u_n))$ has a converging subsequence for every bounded  sequence $(u_n)$ in $X$.
\end{itemize}
Then $\Phi$ satisfies the Palais-Smale condition.
\end{lemma}
\begin{proof}
 Take a Palais-Smale sequence $(u_n)_{n}$ . By (i), the sequence is bounded and $\Psi(u_n) +
S(u_n) = \Phi'(u_n) \rt 0$ in $X^{*}$. Let $v_n = S(u_n)$. By (ii) there exists a subsequence $(v_{n_k})_{k}\subset(v_n)_{n}$ such that $v_{n_k} \rt v$  in $X^*$ for some $v\in X^{*}$. Therefore
\[
u_{n_k} = \Psi^{-1}(\Phi'(u_{n_k} ) - v_{n_k} ) \rt \Psi^{-1}(-v).
\]
\end{proof}

\begin{remark}
The operators $\Psi$ and $S$ appearing in the decomposition \eqref{decomposition} are not necessarily linear. The condition (ii) related to the operator $S$ is satisfied in the case $X$ is a reflexive Banach space and $S:X \rt X^*$ is a compact linear operator, which is the case in this section. We refer to \cite[Section 3]{Ederson2008} for an example where $\Psi$ and $S$ are both nonlinear operators.
 \end{remark}

\begin{proposition}\label{psconditionphi}
 Assume \eqref{eq:(p,q)_for_Ederson_part}. Then the functional $\Phi: X\to  \R$, defined by \eqref{phi}, satisfies the Palais-Smale condition.
\end{proposition}
\begin{proof}
Observe that $X$ is a reflexive Banach space, $\Phi \in C^1(X, \R)$ is such that $\Phi' = \Psi - T$ where $\Psi: X \rt X^*$ is a homeomorphism and $T: X \rt X^*$ is a linear compact operator. Lemma \ref{lemmaps} implies that all we need to prove is that any Palais-Smale sequence of $\Phi$ is bounded.

Let $(f_n, g_n)$ be a Palais-Smale sequence of $\Phi$. Then, there exist $C>0$ and a sequence $(\eps_n)_{n}$ of positive numbers such that for every $n\in\N$, 
\begin{equation}\label{boundedcondps}
|\Phi(f_n, g_n )| \leq C \quad \text{and} \quad \|\Phi'(f_n, g_n)\|_{X^*} \leq \eps_n.
\end{equation}

Set 
\[
p' = \dfrac{\sqrt{(q-p)^2 +4} - (q-p)}{2} \quad \text{and} \quad q' = \dfrac{\sqrt{(q-p)^2 +4} + (q-p)}{2}.
\]Then observe that the straight line passing through $(p', q')$ and $(p,q)$ has its slope equal to $1$, $p>p'>0$, $q>q'>0$ and
\[
p'q'= 1, \quad \text{that is}, \quad \frac{1}{p'+1} + \frac{1}{q'+1} =1.
\]
Since for every $(f, g) \in X$ we have
\[
\frac{1}{2} \langle T(f, g), (f, g) \rangle - \left\langle T(f, g), \left( \frac{f}{p'+1}, \frac{g}{q'+1} \right)  \right\rangle = 0, 
\]
it follows from \eqref{boundedcondps} and the identity
\begin{multline*}
 \Phi'(f_n, g_n)\left(\frac{f_n}{p'+1}, \frac{g_n}{q'+1}\right) - \Phi(f_n, g_n) = \\\left( \frac{1}{p'+1} - \frac{1}{p+1}\right) \|f_n\|^{\frac{p+1}{p}} + \left( \frac{1}{q'+1} - \frac{1}{q+1} \right) \|g_n\|^{\frac{q+1}{q}}
\end{multline*}
that $(f_n, g_n)$ is a bounded sequence in $X$. Therefore, by Lemma \ref{lemmaps}, $\Phi$ satisfies the Palais-Smale condition.
\end{proof}

\begin{proposition}\label{mountain pass geometry}
Assume \eqref{eq:(p,q)_for_Ederson_part} holds. Then $\Phi$ has a local minimum at $(0,0)$ and a mountain pass geometry around $(0,0)$.
\end{proposition}
\begin{proof}
Indeed $\Phi(0,0) = 0$ and 
\begin{multline}\label{estimateMP}
\Phi(f,g) = \int_\Omega \left(\frac{p}{p+1}|f|^{\frac{p+1}{p}} + \frac{q}{q+1}|g|^{\frac{q+1}{q}}\right) dx - \frac{1}{2}\int_\Omega\left( f Kg + g Kf\right)\, dx\\
\geq \frac{p}{p+1}\|f\|_{\frac{p+1}{p}}^{\frac{p+1}{p}} + \frac{q}{q+1}\|g\|_{\frac{q+1}{q}}^{\frac{q+1}{q}} - \frac{1}{2} \left( \|K(g)\|_{p+1} \|f\|_{\frac{p+1}{p}} + \|K(f)\|_{q+1} \|g\|_{\frac{q+1}{q}} \right)\\
\geq \frac{p}{2(p+1)}\|f\|_{\frac{p+1}{p}}^{\frac{p+1}{p}} + \frac{q}{2(q+1)} \|g\|_{\frac{q+1}{q}}^{\frac{q+1}{q}} - \frac{1}{2} \left( \frac{\|K(g)\|_{p+1}^{p+1}}{p+1} + \frac{\|K(f)\|_{q+1}^{q+1}}{q+1} \right)\\
 \geq \frac{p}{2(p+1)}\|f\|_{\frac{p+1}{p}}^{\frac{p+1}{p}} + \frac{q}{2(q+1)} \|g\|_{\frac{q+1}{q}}^{\frac{q+1}{q}} - C \left( \|f\|_{\frac{p+1}{p}}^{q+1} + \|g\|_{\frac{q+1}{q}}^{p+1} \right).
\end{multline}
Then, since $\frac{1}{p+1} + \frac{1}{q+1} < 1$, which is equivalent to $pq>1$, it follows that $\frac{p+1}{p} < q+1$ and $\frac{q+1}{q} < p+1$. Therefore, we have established that $(0,0)$ is a local minimum of $\Phi$ and that there exist $r> 0$ and $b>0$ such that
\[
\Phi(f, g) \geq b \quad \forall \, (f,g) \in X \quad \text{such that} \quad \|(f,g)\|_X =r.
\]
Moreover, let $\varphi_1$ be the positive eigenfunction such that $\int_\Omega \varphi_1^2 dx =1$, associated to the first eigenvalue $\lambda_1$ of $(-\Delta + c(x)I, H^1_0(\Omega))$. Then
\[
\Phi(t \varphi_1, t^{\frac{q(p+1)}{p(q+1)}} \varphi_1) = \left(\frac{p}{p+1}\|\varphi_1\|_{\frac{p+1}{p}}^{\frac{p+1}{p}} + \frac{q}{q+1}\|\varphi_1\|_{\frac{q+1}{q}}^{\frac{q+1}{q}}\right)  t^{\frac{p+1}{p}}- \frac{t^{\frac{2pq + p +q}{p(q+1)}}}{\lambda_1}
\]
and therefore, since $pq>1$, we can choose $t_0>0$ large enough such that \linebreak $\Phi(t_0\varphi_1, t_0^{\frac{q(p+1)}{p(q+1)}} \varphi_1) < 0$ and $\|(t_0\varphi_1, t_0^{\frac{q(p+1)}{p(q+1)}} \varphi_1)\|_X > r$.
\end{proof}


From Propositions \ref{psconditionphi} and \ref{mountain pass geometry} and the classical montain pass theorem \cite{AmbrosettiRabinowitz} we know that
\begin{equation}\label{mountainpassleveldual}
\inf_{\gamma \in \Gamma} \max_{t \in [0,1]} \Phi(\gamma(t)) 
\end{equation}
is a positive critical value of $\Phi$, where
\[
\Gamma = \{ \gamma \in C([0,1], X):\ \gamma(0)= 0 \ \text{and} \ \Phi(\gamma(1)) <0 \}.
\]

Next we consider a Nehari type manifold associated to $\Phi$, namely
\begin{equation}\label{neharidual}
\mathcal{N} _\Phi= \left\{ (f,g) \in X:\ (f,g) \neq (0,0) \ \text{and}\ \Phi'(f,g)\left(f,\frac{q(p+1)}{p(q+1)}g\right) =0\right\}.
\end{equation}
It will become clear below why one needs to add the multiplier $\frac{q(p+1)}{p(q+1)}$ to the condition in the Nehari set. This is related to the fact that the powers of $f$ and $g$ in the energy $\Phi$ are different; observe that the multiplier is such that $\frac{q(p+1)}{p(q+1)} \frac{q+1}{q}=\frac{p+1}{p}$.
Observe moreover that if $(f,g) \in \mathcal{N}_\Phi$ then
\begin{equation}\label{eq:Phi>0}
\Phi(f,g) = \frac{pq-1}{p(q+1) + q(p+1)} \int_\Omega \left(\frac{p}{p+1}|f|^{\frac{p+1}{p}} + \frac{q}{q+1}|g|^{\frac{q+1}{q}} \right)dx > 0.
\end{equation}

\begin{definition}
A critical point $(f,g)$ of $\Phi$ is called a least energy critical point if $\Phi(f,g)$ is the smallest value among the nontrivial critical values of $\Phi$. 
\end{definition}

Using the previous notations, one can easily check that for every $(f,g)$ critical point of $\Phi$ it holds
\[
\Phi(f,g)= \int_\Omega \langle \nabla u,\nabla v\rangle \, dx-\int_\Omega \left(\frac{1}{p+1}|u|^{p+1}+\frac{1}{q+1}|v|^{q+1}\right)\, dx.
\]
Thus the least energy critical level defined through $\Phi$ coincides with the one defined in the previous section, namely we have that
\[
c(\Omega)=\inf \left\{\Phi(f,g):\ (f,g)\in X,\ (f,g)\neq 0,\ \Phi'(f,g)=0 \right\}.
\]
Before we give two other characterizations of the least energy level, we state and prove some properties of $\Ncal_\Phi$.

\begin{lemma}\label{lemma:properties_of_N_Phi}
The Nehari set $\Ncal_\Phi$ associated to the functional $\Phi$ has the following properties.
\begin{enumerate}[(i)]
\item There exists $R> 0$ such that $\|(f,g)\|_X \geq R$ for every $(f,g)\in \mathcal{N}_\Phi$.
\item The set $\mathcal{N}_\Phi$, as defined by \eqref{neharidual}, is a $C^1$-manifold on $X$ of codimension one. 
\item $\mathcal{N}_\Phi$ is a natural constraint to $\Phi$ in the sense that
\[
(f,g)\in \Ncal_\Phi,\ \Phi'|_{\Ncal_\Phi}(f,g)=0   \qquad  \Longrightarrow\qquad     \Phi'(f,g)=0.
\]
\end{enumerate}
\end{lemma}

\begin{proof}
\emph{Item (i)} From the definition of $\Phi$ we see that $(f,g) \in \mathcal{N}_\Phi$ if, and only if, $(f,g) \neq (0,0)$ and 
\[
\int_\Omega \left( |f|^{\frac{p+1}{p}} + \frac{q(p+1)}{p(q+1)}|g|^{\frac{q+1}{q}} \right) dx - \int_\Omega \left( f Kg +\frac{q(p+1)}{p(q+1)} g Kf \right)dx = 0. 
\]

So, if $(f,g) \in \mathcal{N}_\Phi$ then
\begin{multline*}
 \|f\|_{\frac{p+1}{p}}^{\frac{p+1}{p}} + \frac{q(p+1)}{p(q+1)}\|g\|_{\frac{q+1}{q}}^{\frac{q+1}{q}}  = \int_\Omega\left( f Kg +\frac{q(p+1)}{p(q+1)} g Kf \right)dx \\
 \leq \|f\|_{\frac{p+1}{p}} \|Kg\|_{p+1} + \frac{q(p+1)}{p(q+1)}\|g\|_{\frac{q+1}{q}} \|Kf\|_{q+1}\\
 \leq \frac{p}{p+1} \|f\|_{\frac{p+1}{p}}^{\frac{p+1}{p}} + \frac{1}{p+1} \|Kg\|_{p+1}^{p+1} +\frac{q(p+1)}{p(q+1)}\left( \frac{q}{q+1}\|g\|_{\frac{q+1}{q}}^{\frac{q+1}{q}} + \frac{1}{q+1}\|Kf\|_{q+1}^{q+1}\right)
\end{multline*}
and so
\[
\frac{1}{p+1}  \|f\|_{\frac{p+1}{p}}^{\frac{p+1}{p}} + \frac{q(p+1)}{p(q+1)} \frac{1}{(q+1)}  \|g\|_{\frac{q+1}{q}}^{\frac{q+1}{q}} \leq C \left( \|f\|_{\frac{p+1}{p}}^{q+1} +  \|g\|_{\frac{q+1}{q}}^{p+1}\right).
\]
From the last inequality and since $q+1 > \frac{p+1}{p}$, $p+1> \frac{q+1}{q}$, it follows that there exists $R> 0$ such that
\[
\|(f,g)\|_X \geq R \quad \forall \, (f,g) \in \mathcal{N}_\Phi.
\]
\medbreak

\noindent\emph{Item (ii).} We set $\Lambda: X \menos\{(0,0)\} \rt \R$ by
\[
\Lambda (f,g) = \int_\Omega \left( |f|^{\frac{p+1}{p}} + \frac{q(p+1)}{p(q+1)}|g|^{\frac{q+1}{q}} \right) dx - \int_\Omega\left( f Kg + \frac{q(p+1)}{p(q+1)}g Kf \right)dx.
\]
In view of item (i), it is enough to prove that $0$ is a regular value of $\Lambda$. \linebreak First observe that if $\Lambda(f,g) =0$ then
\[
\int_\Omega \left( |f|^{\frac{p+1}{p}} + \frac{q(p+1)}{p(q+1)}|g|^{\frac{q+1}{q}} \right) dx = \int_\Omega \left(f Kg + \frac{q(p+1)}{p(q+1)} g Kf\right) dx > 0.
\]
On the other hand,
\begin{multline*}
\Lambda'(f,g)(f_1, g_1) = \int_\Omega \left( \frac{p+1}{p}|f|^{\frac{1}{p}-1}ff_1 + \frac{p+1}{p}|g|^{\frac{1}{q}-1}gg_1 \right) dx\\ - \left( 1 + \frac{q(p+1)}{p(q+1)} \right)\int_\Omega \left(f_1 Kg + g_1 Kf \right)dx.
\end{multline*}
Therefore, if $\Lambda(f,g) = 0$ then
\[
\Lambda'(f,g)\left( \frac{p}{p+1}f, \frac{q}{q+1}g \right) = \frac{1- pq}{(p+1)(q+1)} \int_\Omega \left( |f|^{\frac{p+1}{p}} + \frac{q(p+1)}{p(q+1)}|g|^{\frac{q+1}{q}} \right) dx <0
\]
and so $0$ is a regular value of $\Lambda$.

\medbreak
\noindent \emph{Item (iii).} If $(f,g)$ is a critical point of $\Phi'|_{\Ncal_\Phi}$ then there exists $\lambda \in \R$ such that
\[
\Phi'(f,g)(f_1, g_1)= \lambda \Lambda'(f,g)(f_1, g_1) \quad \forall \, (f_1, g_1) \in X.
\]
So, in particular, for $(f_1, g_1) = \left( f, \frac{q(p+1)}{p(q+1)}g \right)$ and by the definition of $\mathcal{N}$
\begin{align*}
0 &= \Phi'(f,g)\left( f, \frac{q(p+1)}{p(q+1)}g \right) = \lambda \Lambda'(f,g) \left( f, \frac{q(p+1)}{p(q+1)}g \right) \\
    &= \lambda \frac{(pq-1)}{p(q+1)} \int_\Omega \left(f Kg + \frac{q(p+1)}{p(q+1)} g Kf \right)dx \\
    &= \lambda \frac{(pq-1)}{p(q+1)} \int_\Omega \left(|f|^{\frac{p+1}{p}} + \frac{(pq-1)}{p(q+1)} |g|^{\frac{q+1}{q}} \right)dx,
\end{align*}
and so $\lambda = 0$.

\end{proof}

We can now prove the following equivalent characterizations of the least energy level $c(\Omega)$.

\begin{theorem}
We have that
\begin{align*}
c(\Omega)&=\inf_{\gamma \in \Gamma} \max_{t \in [0,1]} \Phi(\gamma(t)) \\
		&=\mathop{\inf_{(f,g) \in X}}_{\langle T(f,g), (f,g)\rangle>0} \sup_{t\geq0}\Phi(tf, t^{\frac{q(p+1)}{p(q+1)}}g)\\
		&=\inf_{(f,g)\in \Ncal_\Phi}\Phi(f,g)>0
\end{align*}
is attained.
\end{theorem}
\begin{proof}
First we observe that every nontrivial critical point $(f,g)$ of $\Phi$ is such that $(f,g) \in \mathcal{N}_\Phi$. Take the mountain pass level defined by \eqref{mountainpassleveldual}. Then, since it is a critical value of $\Phi$, it is clear that
\[
\inf_{(f,g) \in \mathcal{N}_{\Phi}}\Phi(f,g) \leq \inf_{\gamma \in \Gamma} \max_{t \in [0,1]} \Phi(\gamma(t)).
\]
On the other hand, given $(f,g) \in X$ such that $\langle T(f,g), (f,g) \rangle>0$ we consider the maps
\begin{equation}\label{eq:gamma(t)=}
 \gamma(t, (f,g)):=(t f, t^{\frac{q(p+1)}{p(q+1)}}g), \qquad \theta(t, (f,g)):=\Phi(\gamma(t,(f,g))).
 \end{equation}
 By a direct computation,
 $$
 \theta(t,(f,g))=At^{\frac{p+1}{p}}-B t^{\frac{p(q+1)+ q(p+1)}{p(q+1)}},$$
 where $A:=\int_\Omega \frac{p}{p+1}|f|^{\frac{p+1}{p}}+\frac{q}{q+1}|g|^{\frac{q+1}{q}} dx>0$, $B:=\frac{1}{2}\langle T(f,g), (f,g) \rangle>0$. We observe that
 $1<\frac{p+1}{p}<\frac{p(q+1)+ q(p+1)}{p(q+1)}$, since $pq>1$. It follows that $\theta(t,(f,g))\to -\infty$ as $t\to +\infty$ and that there exists a unique point $t_0>0$ such that $\theta'(t_0, (f,g))=0$; such a point $ t_0$ is a strict global maximum of the map $\theta(\cdot,(f,g))$. Moreover, $\theta(t,(f,g)) \in \mathcal{N}_\Phi$ for $t> 0$ if, and only if, $t= t_0$ and so,
 \begin{align*}
 \inf_{(f,g) \in \mathcal{N}} \Phi(f,g) &\leq \inf_{\gamma \in \Gamma} \max_{t \in [0,1]} \Phi(\gamma(t))\\
 							& \leq \mathop{\inf_{(f,g) \in X}}_{\langle T(f,g), (f,g) \rangle>0} \sup_{t\geq0}\Phi(t f, t^{\frac{q(p+1)}{p(q+1)}}g) = \inf_{(f,g) \in \mathcal{N}} \Phi(f,g). 
 \end{align*}
 The remaining properties are now a standard consequence of \eqref{eq:Phi>0} and Lemma~\ref{lemma:properties_of_N_Phi}
 \end{proof}

\begin{remark}
The shape of the map $\gamma$ in \eqref{eq:gamma(t)=} is due to the different powers of $f$ and $g$ in the expression of the energy $\Phi$, and also justifies the definition of the Nehari manifold in \eqref{neharidual}. Moreover, one needs to restrict the study of $\gamma$ to $(f,g)$ satisfying $\langle T(f,g), (f,g) \rangle>0$, as otherwise we would get $\theta(t,(f,g))\to +\infty$.
\end{remark}

\subsection{Sign and symmetry properties}\label{subset:Sign_and_symmetry}

One of the advantages of using the dual method is that it becomes quite straightforward to prove qualitative properties like sign and symmetry for least energy solutions.

We denote by $B_R$ the open ball in $\IR^N$ of radius $R$ centered at the origin and, for a given function $f \in C(\overline{B_R})$, $f\geq 0$, we denote by $f^*$ the Schwarz symmetric function associated to $f$, namely the radially symmetric, radially non increasing function, equi-measurable with $f$.

\begin{theorem}\label{thm:qualitativeprop_dual}
The following two properties hold.
 \begin{enumerate}[(i)]
\item Any least energy critical point $(f,g)$ of $\Phi$ is such that $f>0$ and $g>0$ in $\Omega$, or $f<0$ and $g<0$ in $\Omega$.
\item In case $\Omega$ is a ball and $c\equiv0$, any positive least energy critical point of $\Phi$ is Schwarz symmetric, that is, $f = f^*$ and $g = g^*$. 
\end{enumerate}
In particular, the corresponding $(u,v)$ solution to \eqref{eq:more general system} satisfies the same properties.
\end{theorem}

For the proof of Theorem \ref{thm:qualitativeprop_dual} we first need to recall a result on the properties of the Schwarz symmetrization. The first conclusion in the lemma below can be found in \cite[Theorem 1]{Talenti} and the second one is a particular case of \cite[Theorem 1]{AlvinoLionsTrombetti}.

\begin{lemma}\label{simetrizacao}
Let $B_R\con \R^N$, $N\geq 1$, be the open ball centered at the origin with radius $R>0$. Let $f \in C(\overline{B_R})$, $f\geq 0$, and $u,w$ satisfy
\[
\begin{array}{rl}
\left\{
\begin{array}{rcll}
-\Delta u & = &f & \hbox{in}\,\,B_R,\\
u & = & 0 &\hbox{on}\,\, \partial B_R,
\end{array}
\right.\,\,\,
\left\{
\begin{array}{rcll}
-\Delta w & = &f^* & \hbox{in}\,\,B_R,\\
w& = & 0 &\hbox{on}\,\, \partial B_R.
\end{array}
\right.
\end{array}
\]
Then $u^* \leq w$ in $B_R$. Furthermore,
\[
|\{ u^* < w \}| = 0 \,\,\,\hbox{if and only if}\,\,\,f = f^*.
\]
\end{lemma}

\begin{proof}[Proof of {\rm Theorem \ref{thm:qualitativeprop_dual}}] 
 \noindent\emph{Assertion (i).} It follows from the maximum principle for the second order elliptic operator $L$ that
\begin{equation}\label{comparisonphi}
 \Phi(|f|,|g|)\leq \Phi(f,g)\qquad \forall \, (f,g)\in X.
 \end{equation}
 Moreover, equality holds if and only if either $f\geq 0$ and $g\geq 0$ or else $f\leq 0$ and $g\leq 0$. So, given a least energy critical point $(f,g)$ of $\Phi$, for some $t_0>0$ we have that
 \begin{multline*}
 c(\Omega)\leq \sup_{t\geq0}  \Phi(t|f|, t^{\frac{q(p+1)}{p(q+1)}}|g|)= \Phi(t_0|f|, t_0^{\frac{q(p+1)}{p(q+1)}}|g|)\leq  \Phi(t_0f, t_0^{\frac{q(p+1)}{p(q+1)}}g)\\
  \leq \sup_{t\geq0}  \Phi(tf, t^{\frac{q(p+1)}{p(q+1)}}g) =  \Phi(f,g)=c(\Omega).
 \end{multline*}
According to \eqref{comparisonphi} and the uniqueness of $t_0$, it follows that $t_0=1$, $\Phi(|f|, |g|)=\Phi(f,g)$, and $(f^+, g^+)=(0,0)$ or $(f^-, g^-)=(0,0)$. Then we apply Proposition \ref{propositionregularitydual} and the strong maximum principle to conclude that $f>0$ and $g>0$ in $\Omega$, or $f< 0$ and $g<0$ in $\Omega$.   

\medbreak
\noindent\emph{Assertion (ii).} Assume $\Omega$ is a ball and $c \equiv0$. Let $f,g$ be a positive least energy critical point of $\Phi$. From Proposition \ref{propositionregularitydual} we know that $f,g$ are $C(\overline{\Omega})$ functions and $f,g>0$ in $\Omega$. So, we have
\begin{equation}\label{comparisonphisimetrizada}
 \Phi(t f^*,t g^*)\leq \Phi(tf,tg)\qquad \forall \, t>0,
 \end{equation}
because
 \begin{equation}\label{+detalhesdes}
 \int f^*Kg^* dx \geq \int f^* (Kg)^* dx \geq \int f Kg dx,
 \end{equation}
where the first inequality is given by Lemma \ref{simetrizacao} and the second is the Hardy-Littlewood inequality, cf. \cite{HardyLittlewoodPolya,Talenti}. Moreover, by Lemma \ref{simetrizacao}, the identity holds at the first inequality of \eqref{+detalhesdes} if and only if $f=f^*$ and $g =g^*$.  On the other hand, as in the proof of item (i), there exists $t_0>0$ such that
 \begin{multline*}
 c(\Omega)\leq \sup_{t\geq0}  \Phi(tf^*, t^{\frac{q(p+1)}{p(q+1)}}g^*)= \Phi(t_0f^*, t_0^{\frac{q(p+1)}{p(q+1)}}g^*)\leq  \Phi(t_0f, t_0^{\frac{q(p+1)}{p(q+1)}}g)\\
  \leq \sup_{t\geq0}  \Phi(tf, t^{\frac{q(p+1)}{p(q+1)}}g) =  \Phi(f,g)=c(\Omega).
 \end{multline*}
According to \eqref{comparisonphisimetrizada} and the uniqueness of $t_0$, it follows that $t_0=1$ and $f=f^*$ and $g =g^*$.
\end{proof}

\begin{remark} The positivity - assertion (i) - in the more restrictive case $p,q>1$ was observed by Alves et al. \cite{Alvesetal} with a similar argument.
\end{remark}

\begin{remark}\label{rem:Troy}
Assuming that $\Omega = B_R(0)$, $c(x)=c$ is a positive constant, and $p,q \geq 1$, the ground state solutions $(u,v)$ of \eqref{eq:more general system} are Schwarz symmetric. This follows from a more general result of Troy \cite{Troy} under the additional assumption that $p,q\geq 1$, which ensures that the nonlinearities $s\mapsto |s|^{p-1}s, \;|s|^{q-1}s$ are Lipchitz continuous. This result is based on the moving plane method, once it is know that $u$ and $v$ are positive. However, the approach based on symmetrization techniques is more direct and natural for ground state solutions of \eqref{eq:more general system} with $c=0$ and allows us to treat more general powers. 
\end{remark}


\section{A reduction by inversion}\label{sec:QuartaOrdem}

Let $\Omega$ be a smooth bounded domain in $\R^N$ with $N \geq 1$. We consider the system
\begin{equation}\label{S}
\left\{
\begin{array}{rcll}
-\Delta u & = & \left| v\right|^{q-1} v & \hbox{in} \,\, \Omega, \\
-\Delta v & = & \left| u\right|^{p-1} u & \hbox{in} \,\, \Omega, \\
u, v &= & 0 & \hbox{on} \,\, \partial \Omega,
\end{array}
\right.
\end{equation}
under the same hypothesis made in Subsection \ref{subsec:W^1s}, namely 
\begin{equation}\label{eq:(p,q)_for_Ederson_part2}
p,q >0,\qquad \gd{\frac{1}{p+1} + \frac{1}{q+1} > \frac{N- 2}{N}},  \tag{H1}
\end{equation}
and we mention that in this part we closely follow some of the procedures in \cite{BonheureSantosRamosTAMS}.

\begin{center}
\includegraphics[scale=.26]{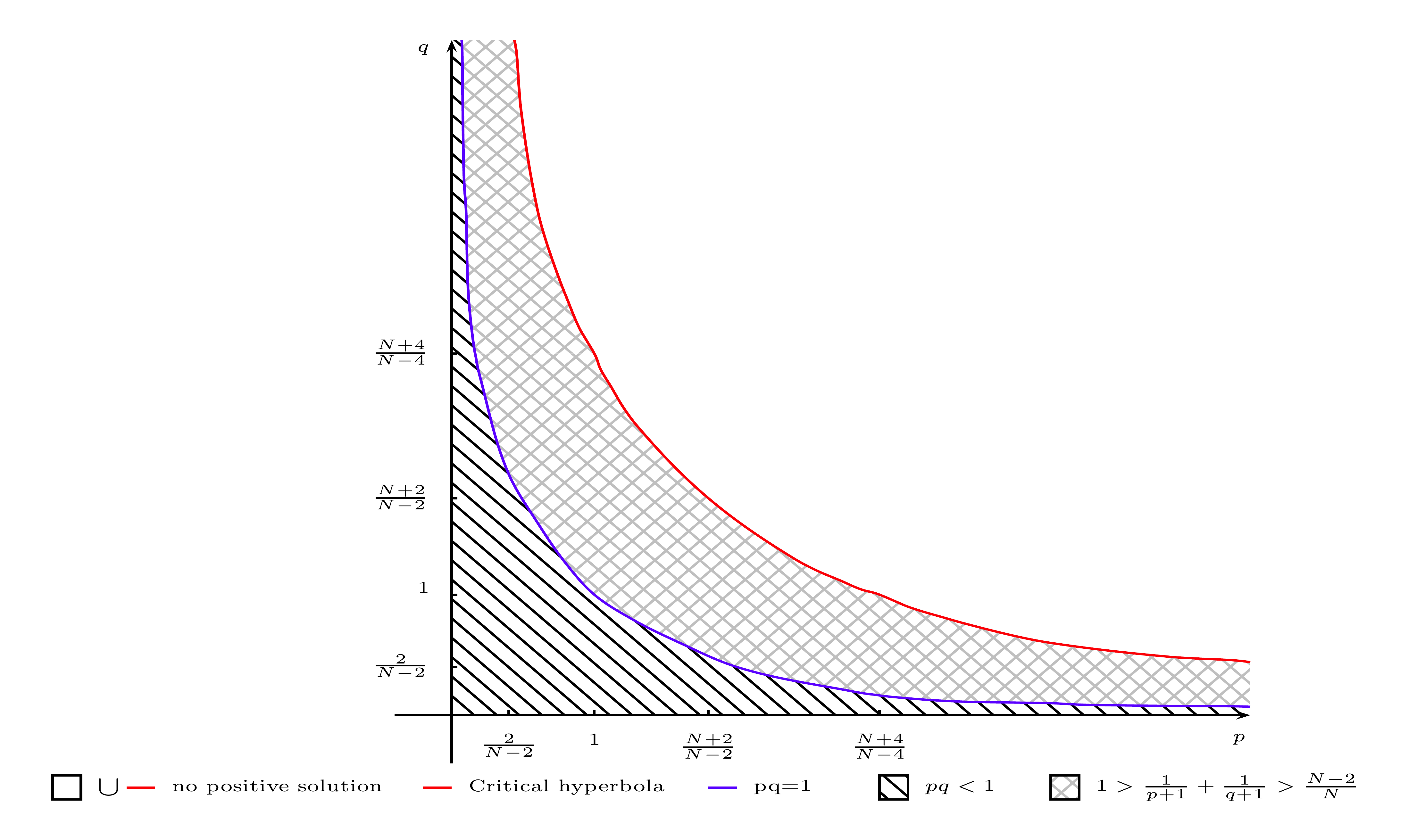}
\end{center}
\vspace{-.3cm}

In this section we reduce \eqref{S} to a fourth order equation. Indeed, see \cite[Theorem 1.1]{Ederson2008} and Proposition \ref{Reg1} hereafter, it is known that \eqref{S} is equivalent to
\begin{equation}\label{ENL}
\left\{
\begin{array}{rcll}
\Delta \left( | \Delta u |^{\frac{1}{q} -1}\Delta u \right) & = & | u |^{p-1}u & \hbox{in}\,\,\Omega\\
u , \Delta u & = & 0 & \hbox{on}\,\,\partial \Omega,
\end{array}
\right.
\end{equation}
in the sense that weak solutions of \eqref{ENL} correspond to classical solutions of \eqref{S}.

The idea of such a reduction goes back at least to P.-L.~Lions \cite{lions85}, see also \cite{ClementMitidieri,ClementFelmerMitidieri,hulshof96,doo-ub,Wang}. It turns out that \eqref{eq:(p,q)_for_Ederson_part2}, the hypothesis for subcriticality for \eqref{S}, is the right hypothesis to ensure a subcritical variational framework for dealing with the single equation \eqref{ENL}. We also mention that a more general class of Hamiltonian systems can be treated by this approach, see for instance \cite{doo-ub}.

\begin{definition}\label{solfraca}
Assume \eqref{eq:(p,q)_for_Ederson_part2}. Let $E = W^{2,\frac{q+1}{q}}(\Omega) \cap W^{1, \frac{q+1}{q}}_0(\Omega)$ be endowed with the norm
\[
\|u \|_E = \gd{\left( \int_{\Omega} | \Delta u |^{\frac{q+1}{q}} dx \right)^{\frac{q}{q+1}}}, \,\, u \in E.
\]
We say that $u \in E$ is a {\it weak solution} of \eqref{ENL} if
\[
\int_{\Omega} | \Delta u |^{\frac{1}{q}-1}\Delta u \Delta v \, dx = \int_{\Omega}| u |^{p-1}u v\, dx, \quad \forall \, v \in E.
\]
\end{definition}

So,  weak solutions of \eqref{ENL} are precisely the critical points of the $C^1(E, \R)$ functional $J: E\to\R$ defined by
\[
J(u) = \frac{q}{q+1} \int_{\Omega} \left| \Delta u \right|^{\frac{q+1}{q}}\, dx - \frac{1}{p+1} \int_{\Omega}\left| u \right|^{p+1}\, dx.
\]

Besides being more direct, another advantage of this approach is that it allows us to treat \eqref{S} in both cases $pq>1$ and $pq<1$ (superlinear and sublinear). Moreover, as we will see, it transforms the search for least energy solutions to \eqref{S} in the search of functions that realize the best constant for the embedding of $E$ into $L^{p+1}(\Omega)$ (cf. Lemma \ref{equivalencia} ahead). The price to pay with this approach is that one has to deal with a fourth order problem.

In case $ (p+1)(N-2), (q+1)(N-2) < 2N$, we recall that we can also define the weak solutions for \eqref{S} as the critical points of the $C^1 (H^1_0(\Omega) \times H^1_0(\Omega), \R)$ functional
\[
I(u,v) := \I_1(u,v)= \int_{\Omega} \langle \nabla u, \nabla v \rangle dx - \int_{\Omega} \left( \frac{|u |^{p+1}}{p+1} + \frac{|v|^{q+1}}{q+1} \right)dx, \,\,\, u,v \in H^1_0(\Omega).
\]

In order to clarify our presentation and justify our definition of ground state solutions we begin with some regularity results.

\begin{proposition}\label{Reg1} Assume that \eqref{eq:(p,q)_for_Ederson_part2} holds. Let $u \in E$ and set $v:= |\Delta u |^{\frac{1}{q} -1} (-\Delta u)$. The following statements are equivalent:
\begin{enumerate}[(i)]
\item $u$ is a critical point of $J$.
\item $u,v \in W^{2,s}(\Omega)$ for all $1 \leq s < \infty$ and $(u,v)$ is a strong solution of {\rm (\ref{S})}.
\item $u \in C^{2,\alpha}(\overline{\Omega})$ and $v \in C^{2,\beta}(\overline{\Omega})$ is a classical solution of \eqref{S} with: $\alpha =q$ if $0 < q <1$, and any $\alpha \in (0,1)$ if $q\geq 1$;  $\beta =p$ if $0 < p <1$, and any $\beta \in (0,1)$ if $p\geq 1$.
\end{enumerate}
In any such case, we have that $J(u) = I(u,v)$.
\end{proposition}

In the case when both $p$ and $q$ are subcritical in the $H^1$--sense, we have the following.

\begin{proposition}\label{Reg2} Assume $(p+1)(N-2), (q+1)(N-2) < 2N$. Let $u,v \in H^1_0(\Omega)$. The following statements are equivalent:\\
{\rm (i) } $(u,v)$ is a critical point of $I$.\\
{\rm (ii) } $ u$ is a critical point of $J$ and $v = |\Delta u |^{\frac{1}{q} -1} (-\Delta u)$.\\
In any such case, we have that $J(u) = I(u,v)$.
\end{proposition}

The proofs for Propositions \ref{Reg1} and \ref{Reg2} are a bit technical and we decide to omit them in this survey; for the interested reader we indicate \cite[Theorem 1.1]{Ederson2008} and \cite[Appendix A]{BonheureSantosRamosTAMS}. The identity $J(u) = I(u,v)$ is obtained in a straightforward way - see the arguments in \cite[eq. (4.9)]{BonheureSantosRamosTAMS} - once we know that weak solutions are indeed regular.

We make a preliminary remark in the case when $pq =1$, in which \eqref{ENL} becomes an eigenvalue problem. Let
\begin{equation}
\lambda_{1,q}: = \inf \left\{ \frac{\int_{\Omega}| \Delta u |^{\frac{q+1}{q}}dx}{\int_{\Omega}| u |^{\frac{q+1}{q}}dx}: u \in E \menos\{ 0\}\right\}.
\end{equation}
Clearly, if $\lambda_{1,q}>1$ then \eqref{ENL} has no nontrivial weak solutions. Moreover, in general $J(u)=0$ for any such weak solution $u \in E$; in particular, the value $J(u)$ does not distinguish weak solutions of \eqref{ENL} in the case when $pq=1$. In virtue of this remark, and since we will be dealing with least energy solutions of \eqref{ENL}, in the sequel we always assume that $pq\neq 1$. Supported by the regularity results stated above, we can now introduce the definition of ground state solution in this context.
\begin{definition}
Assume \eqref{eq:(p,q)_for_Ederson_part2} and $pq \neq 1$. We say that $u \in E \menos \{ 0\}$ is a {\it ground state solution} for \eqref{S} if $J$ attains its smallest nonzero critical value at $u$. 
\end{definition}

As $I(u,v)=J(u)$ for all solutions, then this notion turns out to be equivalent to all the other definitions of ground state (least energy) solutions in this survey, that is
\[
c(\Omega)=\inf \{J(u):\ u\in E,\ u\neq 0, \ J'(u)=0\}.
\]

The next theorem is Theorem \ref{thm:qualitativeprop_dual} (i) for which we will provide an alternative proof.

\begin{theorem}\label{teoexistence}
Assume \eqref{eq:(p,q)_for_Ederson_part2} and $pq \neq 1$. Then \eqref{S} has a ground state solution. Moreover, any ground state solution $(u,v)$ of \eqref{S} is such that $uv > 0$ in $\Omega$.
\end{theorem}

In connection with Theorem \ref{teounicidadesub} below, we point out that \eqref{eq:(p,q)_for_Ederson_part2} holds in case $pq< 1$. In this case, we have a uniqueness result.
\begin{theorem}\label{teounicidadesub}
Assume $pq< 1$. Then {\rm (\ref{S})} has a unique positive solution which is precisely the (positive) ground state solution.
\end{theorem}

In case $\Omega$ is a ball, we have a symmetry result for ground state solutions. Moreover, the uniqueness result also holds in the superlinear case. The next theorem is partially a consequence of Theorem \ref{thm:qualitativeprop_dual} (ii) but we present an alternative proof. 
\begin{theorem}\label{teounicidadebola}
Assume \eqref{eq:(p,q)_for_Ederson_part2}, $pq\neq 1$ and $\Omega = B_R(0)$. Then \eqref{S} has, up to sign, a unique ground state solution. Furthermore, by letting $u>0$ and $v : = | \Delta u |^{\frac{1}{q}-1}(-\Delta u) >0$ we have that both $u$ and $v$ are radially symmetric and radially decreasing with respect to the origin.
\end{theorem}

We mention that in \cite[Theorem 1.2 (i)]{dalmasso}, the author proves the existence of a radial positive solution of \eqref{S} under hypotheses \eqref{eq:(p,q)_for_Ederson_part2}, $pq \neq 1$ and $\Omega = B_R(0)$. Such a result is extended by combining Theorems \ref{teoexistence}-\ref{teounicidadesub}-\ref{teounicidadebola}. The proof of Theorem \ref{teounicidadebola} in based on an extension of \cite[Theorem 1]{FerreroGazzolaWeth}, which deals with the case $p=1$. As for the uniqueness properties above, they turn out to be straightforward consequences of the results in \cite{dalmasso2000, dalmasso}, once the remaining properties (positivity and symmetry) are established.

We will prove these results via the reduction by inversion method in the next subsection. Recall that the existence, sign, and symmetry results had already been proved for $pq>1$ (more precisely under \eqref{eq:(p,q)_for_Ederson_part}) via the dual method in the previous section. Here we decide to reprove these facts with this approach in order to clarify its advantages in the study of Hamiltonian systems, and also in particular to highlight its flexibility by dealing with the more general case \eqref{eq:(p,q)_for_Ederson_part2}, $pq\neq 1$ (without any extra effort). We will also provide other characterizations of $c(\Omega)$.

\subsection{Proof of the existence, sign, and symmetry results}
\label{boundeddomain}
In this part we prove Theorems \ref{teoexistence}-\ref{teounicidadesub}-\ref{teounicidadebola}.

In the sequel we assume that \eqref{eq:(p,q)_for_Ederson_part2} holds and $pq \neq 1$. We denote by $\mathcal{N}_J$ the Nehari manifold associated to the functional $J$, namely
\[
\mathcal{N}_J := \left\{ u \in E:\ u\neq 0\ \text{ and }\  J'(u)u = 0 \right\},
\]
and introduce the minimization problems
\begin{equation}\label{minimizacaoJ}
c_J:= \inf_{u \in \mathcal{N}_J} J (u)
\end{equation}
and
\begin{equation}\label{alphapq}
\alpha_{p,q} : = \inf \left\{ \int_{\Omega} | \Delta u |^{\frac{q+1}{q}} dx \,:\, u \in E, \, \|u\|_{p+1}^{p+1} =1\right\}.
\end{equation}
Note that if it is achieved, $1/(\alpha_{p,q})^{q/(q+1)}$ is the optimal constant for the embedding of $E$ into $L^{p+1}(\Omega)$.

We start by observing that given $u \in E\menos\{ 0 \}$ there exists a unique $t = t(u) > 0$ such that $t(u) u \in \mathcal{N}_J$, which is explicitly given by
\begin{equation}\label{raio}
t(u) = \left( \frac{\| u \|_E^{\frac{q+1}{q}}}{ \| u \|_{p+1}^{p+1}} \right)^{\frac{q}{pq-1}}.
\end{equation}
Now, let $u \in \mathcal{N}_J$. Then $0 = J'(u)u  = \| u \|_E^{\frac{q+1}{q}} - \| u \|_{p+1}^{p+1}$, and therefore
\begin{equation}\label{energiaJ}
J(u) = \frac{q}{q+1} \| u \|_E^{\frac{q+1}{q}} - \frac{1}{p+1} \| u \|_{p+1}^{p+1} = \frac{pq-1}{(p+1)(q+1)} \| u \|_E^{\frac{q+1}{q}}.
\end{equation}
Furthermore,
\begin{equation}\label{quociente}
\gd{\frac{\|u \|_E^{\frac{q+1}{q}}}{\| u \|_{p+1}^{\frac{q+1}{q}}} = \frac{\|u \|_E^{\frac{q+1}{q}}}{\|u \|_E^{\left(\frac{q+1}{q}\right)^2 \frac{1}{p+1}}} = \left( \frac{(p+1)(q+1)}{pq-1} J(u) \right)^{\frac{pq-1}{q(p+1)}}. }
\end{equation}

\begin{lemma}\label{equivalencia}
Assume \eqref{eq:(p,q)_for_Ederson_part2} and $pq\neq 1$. Then the minimization problems \eqref{minimizacaoJ} and \eqref{alphapq} are equivalent in the sense that:
\begin{enumerate}[(i)]
\item Given a minimizing sequence $(u_n) \con \mathcal{N}_J$ for \eqref{minimizacaoJ}, $(\| u_n \|_{p+1}^{-1} u_n)$ is a minimizing sequence for \eqref{alphapq}.
\item Given a minimizing sequence $(\overline{u}_n)$ for \eqref{alphapq}, $(\| \overline{u}_n \|_E^{\frac{q+1}{pq-1}} \overline{u}_n) \con \mathcal{N}_J$ is a minimizing sequence for \eqref{minimizacaoJ}.
\item We have the equality
\begin{equation}\label{relacao}
\gd{c_J = \frac{pq-1}{(p+1)(q+1)} \alpha_{p,q}^{\frac{q(p+1)}{pq-1}}.}
\end{equation}
\item The optimal constant  $\alpha_{p,q}$ is attained if and only if $c_J$ is attained. In addition, if $\overline{u}$ is a solution for \eqref{alphapq}, then $\| \overline{u} \|_E^{\frac{q+1}{pq -1}} \overline{u} = \alpha_{p,q}^{\frac{q}{pq- 1}} \overline{u}$ is a solution for \eqref{minimizacaoJ}. Conversely, if $u$ is a solution for \eqref{minimizacaoJ}, then $\|u \|^{-1}_{p+1}u$ is a solution for \eqref{alphapq}.
\end{enumerate}
\end{lemma}
\begin{proof}
Let $(u_n) \con \mathcal{N}_J$ be a minimizing sequence for \eqref{minimizacaoJ}. Then, by \eqref{quociente},
\begin{multline}\label{lado1}
\alpha_{p,q} \leq \lim_{n \rt \infty} \frac{\| u_n \|_E^{\frac{q+1}{q}}}{\| u_n \|_{p+1}^{\frac{q+1}{q}}} = \lim_{n \rt \infty} \left( \frac{(p+1)(q+1)}{pq-1} J(u_n) \right)^{\frac{pq-1}{q(p+1)}} \\= \left( \frac{(p+1)(q+1)}{pq-1} c_J \right)^{\frac{pq-1}{q(p+1)}}.
\end{multline}
On the other hand, let $(\overline{u}_n)$ be a minimizing sequence for \eqref{alphapq}. Then, by \eqref{raio}, $(\| \overline{u}_n \|^{\frac{q+1}{pq-1}} \overline{u}_n ) \con \mathcal{N}_J$ and so, by \eqref{energiaJ},
\begin{multline}\label{lado2}
c_J \leq \lim_{n \rt \infty} J( \| \overline{u}_n \|_E^{\frac{q+1}{pq-1}} \overline{u}_n ) = \frac{pq-1}{(p+1)(q+1)} \lim_{n \rt \infty} \|\overline{u}_n \|_E^{\frac{q(p+1)}{pq -1}\frac{q+1}{q}} \\ = \frac{pq-1}{(p+1)(q+1)} \alpha_{p,q}^{\frac{q(p+1)}{pq -1}}.
\end{multline}
The proof for (i)-(iii) follows from \eqref{lado1}-\eqref{lado2}.

Now, suppose that $\overline{u} \in E$ is such that $\| \overline{u} \|_{p+1} =1$ and $\alpha_{p,q} = \| \overline{u}\|_E^{\frac{q+1}{q}}$. Then, by \eqref{raio}, $\| \overline{u} \|_E^{\frac{q+1}{pq-1}} \overline{u} = \alpha_{p,q}^{\frac{q}{pq-1}} \overline{u} \in \mathcal{N}_J$. Furthermore, for every $u \in \mathcal{N}_J$ we see from \eqref{quociente} that
\begin{align*}
 \left( \frac{(p+1)(q+1)}{pq-1} J(\alpha_{p,q}^{\frac{q}{pq-1}} \overline{u} ) \right)^{\frac{pq-1}{q(p+1)}} & = \alpha_{p,q} & \\ & \leq \frac{\|u \|_E^{\frac{q+1}{q}}}{\| u\|_{p+1}^{\frac{q+1}{q}}} = \left( \frac{(p+1)(q+1)}{pq-1} J(u) \right)^{\frac{pq-1}{q(p+1)}},
\end{align*}
that is, $J(\alpha_{p,q}^{\frac{q}{pq-1}} \overline{u} ) \leq J(u)$. Therefore, $J(\alpha_{p,q}^{\frac{q}{pq-1}} \overline{u} ) = c_J$.

Conversely, suppose that $u \in \mathcal{N}_J$ is such that $J(u) = c_J$. Then, by \eqref{quociente} and \eqref{relacao},
\[
\frac{\| u \|_E^{\frac{q+1}{q}}}{\| u \|_{p+1}^{\frac{q+1}{q}} } = \left( \frac{(p+1)(q+1)}{pq-1} J(u) \right)^{\frac{pq-1}{q(p+1)}} = \left( \frac{(p+1)(q+1)}{pq-1} c_J \right)^{\frac{pq-1}{q(p+1)}} = \alpha_{p,q}.
\]
This completes the proof of (iv).
\end{proof}

\begin{lemma}\label{alphapqatingido}
Assume \eqref{eq:(p,q)_for_Ederson_part2} and $pq \neq 1$. Then the optimal constant  $\alpha_{p,q}$ is attained, i.e. there exists $u \in E$ such that $\| u \|_{p+1} =1$ and $\| u \|_E^{\frac{q+1}{q}} = \alpha_{p,q}$.
\end{lemma}
\begin{proof}
It is a straightforward consequence of the fact that $E$ is compactly embedded into $L^{p+1}(\Omega)$, since $\Omega$ is a bounded smooth domain.
\end{proof}

Our next lemma shows that the minimization problem \eqref{minimizacaoJ} is a natural method for finding ground state solutions for \eqref{S}, namely we show that $c(\Omega)=c_J$.

\begin{lemma}\label{neharisolucao}
Assume \eqref{eq:(p,q)_for_Ederson_part2} and $pq\neq 1$. If $u \in \mathcal{N}_J$ is such that $J(u) = c_J$ then $u$ is a ground state solution for \eqref{S}. Conversely, if $u$ is a ground state solution for \eqref{S} then $J(u) = c_J$.
\end{lemma}
\begin{proof} Let $G(u)=J'(u)u$. Then $G'(u)u=\frac{1-pq}{q} \| u \|_E^{\frac{p+1}{p}} \neq 0$ for every $u\in \mathcal{N}_J$, and the first conclusion follows from the Lagrange multiplier theorem. In particular, thanks also to Lemmas \ref{equivalencia} (iv) and \ref{alphapqatingido} we have that there exists $\bar{u}\in \mathcal{N}_J$ such that $J(\bar{u})=c_{J}$ and $J'(\bar{u})=0$, and this yields our second conclusion.
\end{proof}

\begin{lemma}\label{positividade}
Let $u \in\mathcal{N}_J$ be such that $J(u)=c_{J}$. Then $u>0$ and $-\Delta u > 0$ in $\Omega$, or else $u<0$ and $-\Delta u < 0$ in $\Omega$.
\end{lemma}
\begin{proof} We infer from Lemma \ref{neharisolucao} and Proposition \ref{Reg1}, that the couple $(u,v)$ with $v := | \Delta u |^{\frac{1}{p}-1}(-\Delta u)$ classically solves the problem \eqref{S} and we have that $u,v\in C^{2, \alpha}(\overline{\Omega})$ for a suitable $\alpha \in (0,1)$. By using the strong maximum principle, we will be done if we show that $-\Delta u$ does not change sign in $\Omega$.

Now, we use an argument that goes back at least to van der Vorst \cite{vanderVorstdie}. Namely, let $w\in E$ be such that $-\Delta w=|\Delta u|$, so that $-\Delta(w\pm u)\geq 0$. Arguing by contradiction, suppose $-\Delta u$ does change sign in $\Omega$. Then $-\Delta(w\pm u)\neq 0$ and the strong maximum principle implies that $w>|u|$. Then, using also Lemma \ref{equivalencia} (iv), we have that
\[
\gd{ \int _{\Omega}\left| \Delta \left( \frac{w}{\| w \|_{p+1}} \right) \right|^{\frac{q+1}{q}} dx = \int_{\Omega} \left(\frac{| \Delta u |^{\frac{q+1}{q}}}{\| w \|_{p+1}^{\frac{q+1}{q}}}\right) dx
<\int_{\Omega} \left(\frac{| \Delta u |^{\frac{q+1}{q}}}{\| u \|_{p+1}^{\frac{q+1}{q}}}\right) dx = \alpha_{p,q}. }
\]
This contradicts the definition of $\alpha_{p,q}$ and completes the proof.\end{proof}

Before we pass to the proof of the main theorems of this section recall Lemma \ref{simetrizacao}, and the definition of the Schwarz symmetrization $f^\ast$ of a function $f\in C(\overline{B_R})$ at the beggining of Subsection \ref{subset:Sign_and_symmetry}.

\begin{proof}[Proof of {\rm Theorems \ref{teoexistence}-\ref{teounicidadesub}-\ref{teounicidadebola}}] The conclusion in Theorem \ref{teoexistence} follows from Lemma \ref{equivalencia} (iv), Lemma \ref{alphapqatingido}, Lemma \ref{neharisolucao} and Lemma \ref{positividade}. The uniqueness property of Theorem \ref{teounicidadesub} is then a direct consequence of \cite[Theorem 3]{dalmasso2000}. As for Theorem \ref{teounicidadebola}, once the radial symmetry is established the uniqueness of the ground state follows from \cite[Theorem 1.1 (i)]{dalmasso}.
 Now, let $u \in E$ be a ground state solution for \eqref{S} such that $u, -\Delta u>0$ in $\Omega$, and set $f := -\Delta u\ \in C(\overline{B_R})$. Let $w$ be such that $-\Delta w=f^*$ in $B_R$, $w=0$ on $\partial B_R$. In order to complete our proof we must show that $f=f^*$. Arguing by contradiction, suppose $f\neq f^*$. It follows then from Lemma \ref{simetrizacao} that $\|w\|_{p+1}>\|u^*\|_{p+1}$. Thus, using also Lemma \ref{equivalencia} (iv), we have that
\begin{multline*}
 \int_{B_R}\left| \Delta\left(\frac{w}{| w |_{p+1}}\right) \right|^{\frac{q+1}{q}}dx = \int_{B_R} \left( \frac{| \Delta u |^{\frac{q+1}{q}}}{| w |_{p+1}^{\frac{q+1}{q}}} \right)dx \\ < \int_{B_R} \left( \frac{| \Delta u |^{\frac{q+1}{q}}}{| u^* |_{p+1}^{\frac{q+1}{q}}} \right)dx = \int_{B_R} \left| \Delta\left(\frac{u}{| u |_{p+1}}\right) \right|^{\frac{q+1}{q}}dx = \alpha_{p,q}.
\end{multline*}
This contradicts the definition of $\alpha_{p,q}$ and completes the argument.\end{proof}


\section{A Lyapunov-Schmidt type reduction}\label{sec:Reduced Functional}

In Sections \ref{section:direct_approaches} and \ref{sec:DualMethod}, we have presented several approaches where we look at solutions or critical points as couples $(u,v)$ in a product of two functional spaces. In Section \ref{sec:QuartaOrdem}, we have reduced the system to a scalar equation which, as a price to pay, leads to an increase of the order of the problem. In this section, we again reduce the problem to a single equation or equivalently to the existence of critical points of a scalar functional, but without increasing the order. This reduction can be thought as an infinite dimensional Lyapunov-Schmidt reduction.

The approach of this section, for the case $p,q>1$, consists in making the most of the saddle geometry of the functional \eqref{eq:usual_action_functional} and in using the corresponding decomposition of the functional space. 
It should be noted that \eqref{eq:usual_action_functional} has a mountain pass geometry when restricted to the space of pairs with equal components $H^+:=\{(u,u)\}$, while it is concave when restricted to $H^-=\{(u,-u)\}$. This allows us to prove that for each $(u,u)$, there exists a unique $(\Psi_u,-\Psi_u)$ so that $(u+\Psi_u,u-\Psi_u)$ maximizes the energy functional; more importantly, as a function of $u$, the energy evaluated at such type of points has a mountain pass geometry, critical points correspond to solutions of the system, and one can apply the classical theory to such reduced functional. This rather simple idea will allow to substitute the saddle geometry of $\I_s$ or $\G_{s}$ at the origin by a mountain-pass geometry for a scalar functional. 

\subsection{Preliminaries}
Before introducing the reduced functional in Subsection \ref{subset:reducedFunctional}, we will make some preliminary considerations. We aim at working with
\begin{equation}\label{eq:p_and_q_reduzidoGERAL}
p,q>1, \qquad \frac{1}{p+1}+\frac{1}{q+1}>\frac{N-2}{N}. \tag{H4}
\end{equation}
However, first we will make some preliminary considerations in the (apparently) more restrictive case
\begin{equation}\label{eq:p_and_q_reduzido}
p,q>1,  \qquad (p+1)(N-2),(q+1)(N-2)\leq 2N. \tag{H4'}
\end{equation}

\begin{center}
\includegraphics[scale=.26]{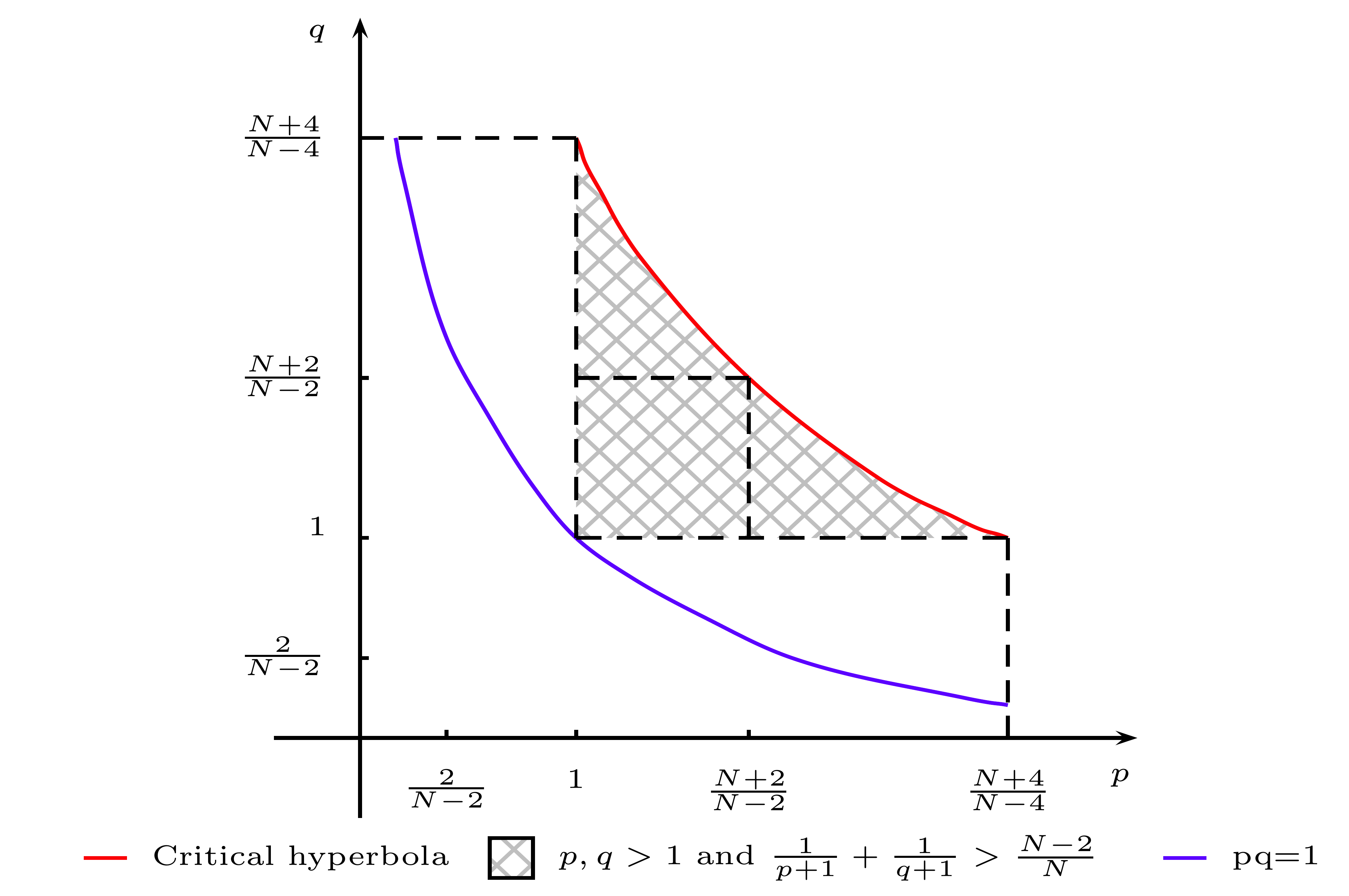}
\end{center}

\vspace{-.1cm}
Ahead in Subsection \ref{sec:p_diferente_q} we will justify why one does not lose generality by imposing that \emph{both} $p$ and $q$ are subcritical in the $H^1_0(\Omega)$-sense. This fact is not obvious, and it is related to a priori bounds on the $L^\infty$--norm of some families of solutions to appropriately truncated problems. Moreover, in Remark \ref{rem:p=1_ou_q=1} we will make some observations about the case $pq>1$ with either $p=1$ or $q=1$.

Under \eqref{eq:p_and_q_reduzido} we can use the $E^s\times E^t$ framework with $s=t=1$ (cf. Subsection \ref{subset:Fractional}), and work in $H^1_0(\Omega)\times H^1_0(\Omega)$. We have already seen that the ground state level is achieved and positive (cf. Corollary \ref{coro:_least_energy_achieved}). Here we will provide two other variational characterizations for this critical level, check Propositions \ref{prop:gs_caract2} and \ref{prop:gs_caract3} ahead. 

We will deal with \eqref{eq:main_system} of the particular type
\begin{equation}\label{eq:system_with_f_and_g}
\left\{
\begin{array}{ll}
-\Delta u=g(v) & \text{ in } \Omega,\\
 -\Delta v=f(u) & \text{ in } \Omega,\\
 u,v=0  & \text{ on } \partial \Omega.
 \end{array}
\right.
\end{equation}
Since we will deal with truncations of the functions $f$ and $g$ in Subsection \ref{sec:p_diferente_q}, we need to consider nonlinearities which are not necessarily pure powers. Following \cite[Section 4]{BonheureSantosRamosJFA} and \cite{RamosTavares}, we assume that the $C^1$-functions  $f,g:\R\to \R$ satisfy the following conditions.
\begin{itemize}
\item[($fg1$)] $f(s)=\textrm{o}(s)$, $g(s)=\textrm{o}(s)$ as $s\to 0$;
\end{itemize}
either
\begin{itemize}
\item[($fg2$)] there exist $p,q$ satisfying \eqref{eq:p_and_q_reduzidoGERAL} and $C>0$ such that
\[
|f(s)|\leq C(1+|s|^p),\quad |g(s)|\leq C(1+|s|^{q}) \qquad \forall s\in \R,
\]
\end{itemize}
or
\begin{itemize}
\item[($fg2'$)] there exist $p,q$ satisfying \eqref{eq:p_and_q_reduzido} and $C>0$ such that
\[
|f'(s)|\leq C(1+|s|^{p-1}),\quad |g'(s)|\leq C(1+|s|^{q-1}) \qquad \forall s\in \R;
\]
\end{itemize}
and
\begin{itemize}
\item[($fg3$)] there exists $\delta>0$ such that 
\[
0<(1+\delta)f(s)s\leq f'(s)s^2,\quad  0<(1+\delta)g(s)s\leq g'(s)s^2 \qquad \forall s\neq 0.
\]
\end{itemize}
Observe that ($fg3$) implies that $f(s)\geq f(1)s^{1+\delta}$ (for $s\geq 1$), and $f(s)\leq f(-1)|s|^{1+\delta}$ (for $s\leq -1$), and the same holds for $g$; this yields in particular that $1+\delta\leq p$. In this subsection and in the following, we will assume ($fg1$)--($fg2'$)--($fg3$), and consider the more general case ($fg1$)--($fg2$)--($fg3$) in Subsection \ref{sec:p_diferente_q}.

Under assumptions $(fg1),(fg2'),(fg3)$, the following energy functional $\I=\I_1:H^1_0(\Omega)\times H^1_0(\Omega)\to \R$,
\[
\I(u,v)=\int_\Omega \langle \nabla u,\nabla v\rangle\, dx-\int_\Omega F(u)\, dx-\int_\Omega G(v)\, dx
\]
is well defined and of class $C^2$ (where $F(s):=\int_0^s f(\xi)\, d\xi$, $G(s):=\int_0^s g(\xi)\, d\xi$). In order to simplify notations, we will denote from now on $H:=H^1_0(\Omega)\times H^1_0(\Omega)$, which splits in $H^+\oplus H^-$, with
\[
H^+=\{(\phi,\phi):\ \phi\in H^1_0(\Omega)\},\qquad H^-=\{(\phi,-\phi):\ \phi\in H^1_0(\Omega)\},
\]
writing each $(\varphi, \psi) \in H^1_0(\Omega)\times H^1_0(\Omega)$ as
\[
(\varphi, \psi) = \left( \frac{\varphi + \psi}{2}, \frac{\varphi + \psi}{2} \right) + \left( \frac{\varphi - \psi}{2}, \frac{\psi - \varphi}{2}\right).
\]
A \emph{weak solution} corresponds to a critical point of $\I$, and the ground state level is given by
\begin{equation}\label{eq:least_energy_level}
c(\Omega)=\inf\{I(u,v)\ \mid\ u,v\in H^1_0(\Omega),\ (u,v)\neq (0,0),\ \I'(u,v)=0\}.
\end{equation} 
Before providing other characterizations of this level, we start with the following results which in the $H^1_0\times H^1_0$ case generalize Lemma \ref{I_s_satisfies_PS} and Theorem \ref{thm:Existence_with_E^s}.

\begin{lemma}\label{lemma:I_PalaisSmale}
The functional $\I$ satisfies the Palais-Smale condition.
\end{lemma}

\begin{proof}  In Lemma \ref{I_s_satisfies_PS} we presented the proof of this fact in the case of $f$ and $g$ being pure powers. The general case follows the same line, proving first that a Palais-Smale sequence $(u_n,v_n)$ satisfies
\[
\frac{\delta}{2(2+\delta)}\int_\Omega (f(u_n)u_n+g(v_n)v_n)\, dx=\textrm{O}(1)+\textrm{o}(1)\|(u_n,v_n)\|_H,
\]
and showing afterwards the existence of $\kappa>0$ such that
\begin{align*}
\|(u_n,v_n)\|_H^2=&\int_\Omega (f(u_n)u_n+g(v_n)v_n)\, dx +\I'(u_n,v_n)(v_n,u_n)\\
			\leq& \frac{1}{2}\|(u_n,v_n)\|^2+\kappa\|v_n\|_{p+1}\left( \int_\Omega f(u_n)u_n\, dx\right)^{\frac{p}{p+1}}\\
			&+\kappa \|u_n\|_{q+1}\left( \int_\Omega g(v_n)v_n\, dx \right)^\frac{q}{q+1}+\textrm{o}(1)\|(u_n,v_n)\|_H.
\end{align*}
In the last inequality it is used the fact that, for any given $\varepsilon>0$, there exists $C>0$ such that $|f(s)|^{p/(p+1)}\leq Cf(s)s$ and $|g(s)|^{q/(q+1)}\leq Cg(s)s$ whenever $|s|\geq \varepsilon$.
\end{proof}

\begin{lemma}\label{lemma:Existence_result}
Given $w,z$ such that $w\neq -z$, there exists a nontrivial critical point $(u,v)$ of $\I$ such that
\begin{equation}\label{eq:I(u,v)leq_BenciRab}
\I(u,v)\leq \mathop{\sup_{t\geq 0}}_{\phi\in H^1_0(\Omega)} \I(t(w,z)+(\phi,-\phi)).
\end{equation}
\end{lemma}

\begin{proof}
\noindent \emph{Step 1.} Given $(u,u)\in H^+$, from ($fg1$)--($fg2'$) we deduce the existence of $C>0$ such that
\[
\I(u,u)=\|u\|^2_{H^1_0}-\int_\Omega F(u)\, dx-\int_\Omega G(u)\, dx\geq \frac{1}{2}\|u\|^2_{H^1_0}-C\|u\|_{L^{p+1}}^{p+1}-C\|u\|_{L^{q+1}}^{q+1},
\]
whence
\[
\I\geq \alpha \qquad \text{ on } \partial B_\rho(0)\cap H^+,
\]
for some $\rho,\alpha>0$. 

\medbreak

\noindent \emph{Step 2.} Take $(w,z)\in H$ with $w\neq -z$. We claim that
\[
\I(u,v)\leq 0 \text{ whenever } (u,v)\in H^-, \text{ or } (u,v)\in \R^+(w,z) + H^- \text{ with } \|(u,v)\|=R,
\]
for $R$ sufficiently large. Is is clear that $\I(u,-u)\leq 0$ for every $u\in H^1_0(\Omega)$, while
\begin{multline*}
\I(tw+\phi,tz-\phi)=t^2\int_\Omega \ip{\nabla w}{\nabla z}\,  dx-\|\phi\|_{H^1_0}^2+t\int_\Omega \ip{\nabla \phi}{\nabla (z-w)}\\
				-\int_\Omega F(tw+\phi)\, dx-\int_\Omega G(tz-\phi)\, dx\\
				\leq  a_1t^2-\frac{1}{2}\|\phi\|_{H^1_0}^2- a_2\int_\Omega |tw+\phi|^{2+\delta}\, dx-a_3\int_\Omega |tz-\phi|^{2+\delta}\, dx-a_4\\
				\leq  a_1t^2-\frac{1}{2}\|\phi\|_{H^1_0}^2-a_5t^{2+\delta}\int_\Omega |w+z|^{2+\delta}\, dx-a_3\leq 0
\end{multline*}
if either $t$ or $\|\phi\|_{H^1_0}$ is sufficiently large, so we have proved the claim. 

\medbreak 

\noindent \emph{Step 3.} One can conclude applying directly Benci-Rabinowitz's linking theorem \cite[Theorem 0.1]{BenciRabinowitz}. Observe that inequality \eqref{eq:I(u,v)leq_BenciRab} is a direct consequence of the minimax procedure presented in \cite{BenciRabinowitz}.

\end{proof}

As discussed in Section \ref{section:direct_approaches}, it is not true that $c(\Omega)$ can be obtained as the infimum on the usual Nehari manifold of $\I$. Based on the structure of the functional and on Lemma \ref{lemma:Existence_result}, the suitable Nehari-type set to work with is the following:
\begin{equation}\label{eq:Abstract_Nehari}
\Ncal_\I:=\left\{(u,v)\in H\ \mid \begin{array}{c} (u,v)\neq 0, \ \I'(u,v)(u,v)=0,\\ \I'(u,v)(\phi,-\phi)=0,\ \forall \phi\in H^1_0(\Omega)\end{array}\right\}.
\end{equation}
This set was first introduced by Pistoia and Ramos \cite[p. 4]{PistoiaRamosDirichlet}, using ideas from a previous paper of Ramos and Yang \cite{RamosYang}. However, the proofs there required more restrictive assumptions on the nonlinearities. Our observations in this section follow mostly the ideas of Ramos and Tavares \cite{RamosTavares}; see also \cite[Section  5]{TavaresThesis} for a more detailed version.

\begin{lemma}\label{lemma:N_is_manifold}
The set $\Ncal_\I$ is a submanifold of $H$ having infinite codimension. Moreover, $\Ncal_\I$ is a natural constraint for $\I$ in the sense that
\[
(u,v)\in \Ncal_\I,\ \I|_{\Ncal_\I}'(u,v)=0\qquad  \Longrightarrow\qquad \I'(u,v)=0.
\]
\end{lemma}
\begin{proof}
Elements in $\Ncal_\I$ are zeros of the map
\[
K:H\to \R\oplus H^-,\quad K(u,v):=(\I'(u,v)(u,v),P\I'(u,v)),
\]
where $P:H\to H^-$ denotes the orthogonal projection. For any $(u,v)\in \Ncal_\I$, its derivative $K'(u,v):H\to \R\oplus H^-$ is given by
\[
K'(u,v)(\xi,\eta)=(\I'(u,v)(\xi,\eta)+\I''(u,v)(u,v)(\xi,\eta),PI''(u,v)(\xi,\eta)).
\]
Let us now focus on $K'(u,v)$ restricted to the subspace $Z:=\spann \{(u,v)\}\oplus H^-$ of $H$, which we can be identified with $\R\oplus H^-$.

\medbreak

\noindent \emph{Step 1.} $K'(u,v)$ is one to one. Take $(tu+\phi,tv-\phi)\in Z$ such that $K'(u,v)(tu+\phi,tv-\phi)=0$. Then 
\[
\I''(u,v)(u,v)(tu+\phi,tv-\phi)=0,\qquad P\I''(u,v)(tu+\phi,tv-\phi)=0,
\]
and in particular
\[I''(u,v)(tu+\phi,tv-\phi)(tu+\phi,tv-\phi)=0,
\]
that is
\begin{multline*}
2t^  2\int_\Omega \ip{\nabla u}{\nabla v}\, dx+2t\int_\Omega \ip{\nabla \phi}{\nabla(v-u)}\, dx-2\int_\Omega|\nabla \phi|^2\, dx\\
-\int_\Omega f'(u)(tu+\phi)^2\, dx-\int_\Omega g'(v)(tv-\phi)^2 \, dx=0.
\end{multline*}
Adding and subtracting the quantities $\int_\Omega \frac{f(u)}{u}\phi^2\, dx$, $\int_\Omega \frac{g(v)}{v}\phi^2\, dx$ and using the identities $I'(u,v)(t^2u,t^2v)=0$, $I'(u,v)(t\phi,-t\phi)=0$, we obtain
\begin{multline*}
\int_\Omega\left(\frac{f(u)}{u}-f'(u)\right) (tu+\phi)^2\, dx+\int_\Omega \left(\frac{g(v)}{v}-g'(v)\right) (tv-\phi)^2\, dx\\
-\int_\Omega \left(\frac{f(u)}{u}\phi^2+\frac{g(v)}{v}\phi^2\right)-2\int_\Omega |\nabla \phi|^2 \, dx=0
\end{multline*}
Taking into account ($fg3$), we have $\phi\equiv 0$ and $t = 0$.

\medbreak

\noindent \emph{Step 2.} As $1<p,q<2^*-1$, one can check that $Id- K'(u,v)$ is a compact operator. Then by the Fredholm alternative theorem $K'(u,v)|_Z$ being one to one yields that $K'(u,v)$ is onto, and thus $\Ncal_\I$ is a manifold.

\noindent Step 3. If $I'|_{\Ncal_\I}(u,v)=0$, then according to the Lagrange multiplier rule there exist $\lambda\in \R$ and $\Psi\in H^1_0(\Omega)$ such that
\begin{align*}
\I'(u,v)(\xi,\eta)&= \lambda \I'(u,v)(\xi,\eta)+\lambda \I''(u,v)(u,v)(\xi,\eta)+\I''(u,v)(\Psi,-\Psi)(\xi,\eta)\\
	   &= \I'(u,v)(\lambda \xi,\lambda \eta)+\I''(u,v)(\lambda u+\Psi,\lambda v-\Psi)(\xi,\eta)
\end{align*}
for every $(\xi,\eta)\in H$. By taking $(\xi,\eta)=(\lambda u+\Psi,\lambda v-\Psi)$, we have
\[
\I''(u,v)(\lambda u+\Psi,\lambda v-\Psi)(\lambda u+\Psi,\lambda v-\Psi)=0
\]
and thus, as in the Step 1 of this proof, $\lambda=0,\Psi\equiv 0$, and the conclusion follows.
\end{proof}

We can now state a different characterization for the least energy level $c(\Omega)$.

\begin{proposition}\label{prop:gs_caract2}
Assume $(fg1)$--$(fg2')$--$(fg3)$ hold. Then
\[
c(\Omega)=\inf_{\Ncal_\I}\I>0, 
\]
which is achieved.
\end{proposition}
\begin{proof}
First of all observe that
\[
\I(u,v)=\frac{\delta}{2(2+\delta)}\int_\Omega (f(u)u+g(v)v)\, dx>0 \ \forall (u,v)\in \Ncal_\I,
\]
hence $\inf_{\Ncal_\I}\I\geq 0$. Taking a minimizing sequence $\{(u_n,v_n)\}_n$, by Ekeland's variational principle we can suppose without loss of generality that $(u_n,v_n)$ is a Palais Smale sequence. Thus from Lemma \ref{lemma:I_PalaisSmale} we obtain the existence of $(u,v)\in H$ such that, up to a subsequence, $(u_n,v_n)\to (u,v)$ in $H$. Reasoning as in the proof of Corollary \ref{coro:_least_energy_achieved}, we conclude that $(u,v)\neq (0,0)$ achieves $\inf_{\Ncal_\I}\I$. Thus by Lemma \ref{lemma:N_is_manifold} we have $\I'(u,v)=0$, and so $\inf_{\Ncal_\I}\I= c(\Omega)$.
\end{proof}

Let us now look for a third characterization of $c(\Omega)$. Take $(u,v)\in H$, $u\neq -v$. As we saw in the proof of Lemma \ref{lemma:Existence_result}, the quantity
\[
\sup\{\I(t(u,v)+(\phi,-\phi)):\ t\geq 0,\ \phi\in H^1_0\}
\]
is finite. We want to show that it is \emph{uniquely achieved}, and that the maximum is associated to a point in $\Ncal_\I$.

\begin{lemma}\label{lemma:sup_I(t(u,v)+(phi,-phi))}
Given $(u,v)\in H$, $u\neq -v$, there exist unique $t^*>0$ and $\phi^*\in H^1_0(\Omega)$ such that
\[
\sup\{\I(t(u,v)+(\phi,-\phi)):\ t\geq 0,\ \phi\in H^1_0(\Omega)\}=\I(t^*(u,v)+(\phi^*,-\phi^*)).
\]
Moreover, $t^*$ and $\phi^*$ are uniquely characterized by
\[
\I'(t^*u+\phi^*,t^*v-\phi^*)(u,v)=0,\qquad \I'(t^*u+\phi^*,t^*v-\phi^*)(\phi,-\phi)=0,\ \forall \phi\in H^1_0(\Omega),
\]
that is
\[
(t^*u+\phi^*,t^*v-\phi^*)\in \Ncal_\I.
\]
\end{lemma}
\begin{proof}
This was essentially proved in \cite{RamosYang}, under more restrictive assumptions. Here we present a proof which works under ($fg1$)--($fg2'$)--($fg3$). Going back to the proof of Lemma \ref{lemma:Existence_result}, we deduce that
\[
\Lambda(t,\phi):=\I(tu+\phi,tv-\phi)\leq a_1t^2-\frac{1}{2}\|\phi\|_{H^1_0}^2-a_2t^{2+\delta}\int_\Omega |u+v|^{2+\delta}\, dx-a_3
\]
for some $a_1,a_2,a_3>0$, so that $s:=\sup \Lambda(t,\phi)$ is finite. Moreover, taking a maximizing sequence $(t_n,\phi_n)$, $|t_n|$ and $\|\phi_n\|_{H^1_0}$ are bounded, and (up to subsequences)
\[
t_n\to t^*,\qquad \phi_n\rightharpoonup \phi^* \text{ weakly in }H^1_0.
\]
In particular 
\[
\|\phi^*\|_{H^1_0}^2\leq \liminf \|\phi_n\|_{H^1_0}^2,\quad \int_\Omega \ip{\nabla \phi_n}{\nabla (v-u)}\, dx\to \int_\Omega \ip{\nabla \phi}{\nabla (v-u)}\, dx,
\]
and, by Fatou's lemma,
\[
\int_\Omega F(t^\ast u+\phi^\ast)\, dx\leq \liminf \int_\Omega F(t_nu+\phi_n)\, dx,
\]
and the same holds for $G$. So,
\[
s=\limsup \I(t_nu+\phi_n,t_nv-\phi_n)\leq \I(t^\ast u+\phi^\ast,t^\ast v-\phi^\ast)\leq s
\]
and thus $s$ is attained for $t^\ast>0$, $\phi^\ast\in H^1_0(\Omega)$, satisfying 
\[
(u^\ast,v^\ast):=(t^\ast u+\phi^\ast,t^\ast v-\phi^\ast)\in \Ncal_\I.
\]
As for the uniqueness of $t^\ast$, $\phi^\ast$, observe that 
\begin{align*}
\Lambda''(t^\ast,\phi^\ast)(t,\phi)(t,\phi)&=\I''(u^\ast,v^\ast)(tu+\phi,tv-\phi)(tu+\phi,tv-\phi)\\
								&\I''(u^\ast,v^\ast)(su^\ast+\psi,sv^\ast-\psi)(su^\ast+\psi,sv^\ast-\psi)
\end{align*}
for $s:=t/t^\ast$, $\psi=\phi-t\phi^\ast /t^\ast$, and so, by the computations of the proof of Lemma \ref{lemma:N_is_manifold},
\begin{multline*}
\Lambda''(t^\ast,\phi^\ast)(t,\phi)(t,\phi)\leq -\delta\int_\Omega \left(\frac{f(u^*)}{u^*}(su^*+\psi)^2+\frac{g(v^*)}{v^*}(sv^\ast-\psi)^2\right)\, dx\\ - \int_\Omega \left(\frac{f(u^\ast)}{u^\ast}\psi^2+\frac{g(v^*)}{v^*}\psi^2\right)\, dx-2\int_\Omega |\nabla \psi|^2\, dx<0,
\end{multline*}
whenever $(t,\phi)\neq (0,0)$, whence $\Lambda$ has at most one single critical point.
\end{proof}
As a consequence of this lemma, we have the following third characterization of the least energy level $c$.
\begin{proposition}\label{prop:gs_caract3}
Assume that ($fg1$)--($fg2'$)--($fg3$) hold. Then
\[
c(\Omega)=\mathop{\inf_{u,v\in H^1_0}}_{v\neq -u} \mathop{\sup_{t\geq 0}}_{\phi\in H^1_0} \I(t(u,v)+(\phi,-\phi)).
\]
\end{proposition}
We end this preliminary subsection with three observations.
\begin{remark}
With respect to {\rm Subsection \ref{subset:Fractional}}, observe that here we have more information on the ground state level under \eqref{eq:p_and_q_reduzido}. However, observe that this is a more restrictive assumption than \eqref{eq:p_and_q_for_E^s_approach}. 
\end{remark}
\begin{remark}\label{rem:p=1_ou_q=1}
If either $p=1$ and $q>1$,  or $q=1$ and $p>1$, we still have that $\I$ is a $C^2$ functional. Moreover, a closer look at the proof of {\rm Lemma \ref{lemma:N_is_manifold}} shows that its conclusions are true for such $(p,q)$'s. Thus {\rm Proposition \ref{prop:gs_caract2}} still holds in this situation. This observation has already been made in \cite[Lemma 4.1]{BonheureSantosRamosJFA}.
\end{remark}
\begin{remark}
\emph{Extension to other boundary conditions:} following step by step each proof, it is easily seen that the same conclusions hold for $(u,v)\in H^1(\Omega)\times H^1(\Omega)$, that is, we have similar variational characterizations for ground states of the problem
\[
-\Delta u+u=g(v),\quad -\Delta v+v=f(u) \text{ in } \Omega,
\]
with $\partial_\nu u=\partial_\nu v=0$ on $\partial \Omega$. The same observation is true for the system posed in the entire space $\Omega=\R^N$ (cf. \cite[eq. (4.4)]{BonheureSantosRamosTAMS}). Furthermore, thanks to {\rm Proposition \ref{prop:gs_caract3}}, it is also straightforward to show that
\[
\Omega\mapsto c(\Omega)
\]
is decreasing with respect to domain inclusion, and that
\[
\R^+\to \R^+,\qquad \lambda\mapsto c(\lambda) \qquad \text{ is increasing,}
\]
where $c(\lambda)$ denotes the ground state level of 
\[
-\Delta u+\lambda u=g(v),\quad -\Delta v+\lambda v=f(u) \text{ in } \R^N.
\]
\end{remark}


\subsection{Introduction of the reduced functional}\label{subset:reducedFunctional}

The considerations of the previous subsection motivate the following (equivalent) approach. Throughout this subsection we assume ($fg1$)--($fg2'$)--($fg3$), so that $p,q<2^*-1$. Reasoning as in the proof of Lemma \ref{lemma:sup_I(t(u,v)+(phi,-phi))}, we can conclude that for any given $(u,v)$ with $u\neq -v$, there exists a unique function $\Psi_{u,v}\in H^1_0(\Omega)$ such that
\[
\sup\{ \I(u+\phi,v-\phi):\ \phi\in H^1_0\}=\I(u+\Psi_{u,v},v-\Psi_{u,v}\},
\]
which is uniquely characterized by
\[
\I'(u+\Psi_{u,v},v-\Psi_{u,v})(\phi,-\phi)=0,\qquad \forall \phi\in H^1_0(\Omega),
\]
that is
\[
-2\Delta \Psi_{u,v}=-\Delta(v-u)+g(v-\Psi_{u,v})-f(u+\Psi_{u,v}) \qquad \text{ in } H^{-1}(\Omega).
\]
\begin{lemma} \label{lemma:Theta_is_C1}
The map $H\cap\{u\neq -v\} \to H$, $
(u,v)\mapsto \Psi_{u,v}$ is of class $C^1$.
\end{lemma}
\begin{proof} We follow the proof of \cite[Proposition 2.1]{RamosTavares}. We apply the implicit function theorem to the map $ \Theta:(H\cap\{u\neq -v\})\times H^-\to H^-$; $ \Theta((u,v),(\psi,-\psi))=P \I'((u,v)+(\psi,-\psi))$, where $P$ is the orthogonal projection of $H$ onto $H^-$. For any fixed pair $((u,v),(\psi,-\psi))$, the derivative of $\Theta$ with respect to $(\phi,-\phi)$ evaluated at $((u,v),(\psi,-\psi))$ is given by the linear map
$$(\phi,-\phi)\mapsto T(\phi,-\phi)=P \I''(u+\psi,v-\psi)(\phi,-\phi),$$ 
that is
$$\langle T(\phi,-\phi),(\varphi,-\varphi) \rangle= -2\int_\Omega \ip{\nabla\phi}{\nabla \varphi}\, dx-\int_\Omega f'(u+\psi)\phi \varphi-\int_\Omega g'(v-\psi)\phi \varphi, \quad \forall \, \phi,\varphi.$$
Since $(fg2')$ holds, we have that $Id-T$ is a compact operator. The operator $T$ is one-to-one, since if $T(\phi,-\phi)=0$, then
$$-2\|\phi\|^2=\int_\Omega f'(u+\psi)\phi^2+\int_\Omega g'(v-\psi)\phi^2\geq 0$$ and so $\phi=0$. Thus by the Fredholm's alternative theorem $T$ is also onto and hence we can apply the implicit function theorem and obtain the desired result.
\end{proof}

Denote $\Psi_{u,u}$ simply by $\Psi_u$. One can now define the \emph{reduced functional}
\begin{equation}\label{eq:reducedfunctional_J}
\J:H^1_0(\Omega)\to \R,\qquad \J(u):=\I(u+\Psi_u,u-\Psi_u),
\end{equation}
which is a $C^1$ functional by the previous lemma. Moreover,
\begin{align*}
\J'(u)\phi=\I'(u+\Psi_u,u-\Psi_u)(\phi+\Psi_u'\phi,\phi-\Psi_u'\phi)=\I'(u+\Psi_u,u-\Psi_u)(\phi,\phi).
\end{align*}
As $H=H^+\oplus H^-$, we have proved the following lemma.
\begin{lemma}\label{lem:homeo}
The map
\[
\eta:H^1_0\to H \qquad u\mapsto (u+\Psi_u,u-\Psi_u)
\]
is a homeomorphism between critical points of $\J$ and $\I$.
\end{lemma}

The introduction of this reduced functional appeared contemporarily in the papers by Ramos et al \cite{BonheureRamos,RamosTavaresZou,RamosJMAA}. The properties of $\J$, in opposition to the ones of $\I$, are similar to the ones of the energy functional in the single equation case. In fact, from the properties of the functional $\I$, it is easy to show that $\J$ satisfies the Palais-Smale condition, and that it has a mountain pass geometry. Therefore, as far as the least energy level is concerned, we have the following ``usual'' characterizations.

\begin{proposition}
Let $c(\Omega)$ be the least energy level as defined in \eqref{eq:least_energy_level}. Then
\begin{align*}
c(\Omega)&=\inf\{\J(u)\mid  u\in H^1_0(\Omega),\ u\neq 0,\  \J'(u)=0\}\\
  &=\inf_{\Ncal_{\J}} {\J}\\
  &=\inf_{u\in H^1_0\setminus \{0\}} \sup_{t\geq 0} {\J}(tu),
\end{align*}
where $\Ncal_{\J}:=\{u\in H^1_0(\Omega)\ \mid \ u\neq 0,\ {\J}'(u)u=0\}$ is the \emph{standard} Nehari manifold of the reduced functional ${\J}$.
\end{proposition}

\begin{remark}
A related reduction method was also used by Szulkin and Weth \cite{SzulkinWeth}  in order to find ground-state solutions of the equation
\[
-\Delta u+V(x) u =f(x,u),\qquad u\in H^1(\R^N).
\]
with f a superlinear, subcritical nonlinearity, f and V periodic in x, and 0 not belonging to the spectrum of $-\Delta +V$. The authors use a reduction based on the decomposition $H^1(\R^N)=E^+ \oplus E^-$ related to the positive and negative parts of the spectrum of $-\Delta +V$. The work \cite{SzulkinWeth} is inspired by Pankov \cite{Pankov}, where a manifold of type \eqref{eq:Abstract_Nehari} appeared in an independent way. We refer the reader to \cite[Section  4]{SzulkinWethsurvey} for more details on this.
\end{remark}

\subsection{The general case} \label{sec:p_diferente_q}
In this section we justify in a precise way why one does not lose generality by considering $p,q<2^*-1$ instead of $1/(p+1)+1/(q+1)>(N-2)/N$, and $(fg2')$ instead of $(fg2)$. We illustrate this in the study of ground state solutions, and at the end of this section we make more general remarks. The ideas presented here are based on \cite[Section  5]{RamosTavares}; since the proofs there are rather sketchy, we have decided to present here all the details. 

Take $(p,q)$ satisfying \eqref{eq:p_and_q_reduzidoGERAL}, supposing without loss of generality that $p\leq q$, so that $p<2^*-1$. Assume that ($fg1$)--($fg2$)--($fg3$) holds. For each $n\in \N$, consider the following $C^1$ truncation of the functions $f$ and $g$:
\[
f_n(s)=\left\{
\begin{array}{cc}
\frac{f'(n)}{pn^{p-1}}s^p + f(n)-\frac{f'(n)n}{p} &  s> n\\
f(s) 									& -n\leq s\leq n\\
\frac{f'(-n)}{pn^{p-1}}|s|^{p-1}s+f(-n)+\frac{f'(-n)n}{p}  & s<-n.
\end{array}
\right.
\]
and
\[
g_n(s)=\left\{
\begin{array}{cc}
\frac{g'(n)}{pn^{p-1}}s^p + g(n)-\frac{g'(n)n}{p} &  s> n\\
g(s) 									& -n\leq s\leq n\\
\frac{g'(-n)}{pn^{p-1}}|s|^{p-1}s+g(-n)+\frac{g'(-n)n}{p}  & s<-n.
\end{array}
\right.
\]
Observe that
\begin{equation}\label{eq:uniform_bound_of_g_n_independent}
|f_n(s)|\leq C(1+|s|^{p-1}), \qquad |g_n(s)|\leq C(1+|s|^{q-1})
\end{equation}
for some $C$ independent of $n$ (recall that $q\geq p$), and
\begin{equation*}\label{eq:uniform_bound_of_g_n_dependent}
|f_n'(s)|\leq C_n(1+|s|^{p-1}),\qquad |g_n'(s)|\leq C_n(1+|s|^{p-1}).
\end{equation*}
In particular, $f_n$ and $g_n$ satisfy $(fg2')$. For $F_n(s):=\int_0^s f_n(\xi)\, dx$ and  $G_n(s):=\int_0^s g_n(\xi)\, d\xi$, take the functional 
\[
\I_n(u,v)=\int_\Omega \langle \nabla u,\nabla v\rangle\, dx-\int_\Omega F_n(u)\, dx-\int_\Omega G_n(v)\, dx,
\]
which is well defined in $H=H^1_0(\Omega)\times H^1_0(\Omega)$. Take the auxiliary functions
\[
\tilde f(s)=\left\{
\begin{array}{cc}
\frac{f'(1)}{1+\delta}s^{1+\delta} + f(1)-\frac{f'(1)}{1+\delta} &  s> 1\\
f(s) 									& -1\leq s\leq 1\\
\frac{f'(-1)}{1+\delta}|s|^{\delta}s+f(-1)+\frac{f'(-1)}{1+\delta}  & s<-1.
\end{array}
\right.
\]
and
\[
\tilde g(s)=\left\{
\begin{array}{cc}
\frac{g'(1)}{1+\delta}s^{1+\delta} + g(1)-\frac{g'(1)}{1+\delta} &  s> 1\\
g(s) 									& -1\leq s\leq 1\\
\frac{g'(-1)}{1+\delta}|s|^{\delta}s+g(-1)+\frac{g'(-1)}{1+\delta}  & s<-1.
\end{array}
\right.
\]
It is straightforward to check that 
\[
\lambda_1 \tilde f'(s)\leq f_n'(s) \quad \text{ for every $\lambda_1<\min \left\{1,(1+\delta)\frac{f(1)}{f'(1)},-(1+\delta)\frac{f(-1)}{f'(-1)}\right\}$},
\]
and
\[
\lambda_2 \tilde g'(s)\leq g_n'(s) \quad \text{ for every $\lambda_2<\min \left\{1,(1+\delta)\frac{g(1)}{g'(1)},-(1+\delta)\frac{g(-1)}{g'(-1)}\right\}$},
\]so that
\begin{equation}\label{eq:G(s)_leq_G_n(s)}
\lambda \tilde F(s)\leq F_n(s),\quad \lambda \tilde G(s)\leq G_n(s) \qquad \text{ for } \lambda=\min\{\lambda_1,\lambda_2\}.
\end{equation}
Define $\tilde \I:H\to \R$ by
\[
\tilde \I(u,v)=\int_\Omega \langle \nabla u, \nabla v\rangle\, dx-\lambda\int_\Omega \tilde F(u)-\lambda \int_\Omega \tilde G(v)\, dx.
\]
We denote by $c_n$ and $\tilde c$ the least energy levels of $\I_n$ and $\tilde \I$ respectively, as defined before in this section - recall \eqref{eq:least_energy_level}.
\begin{lemma}\label{lemma:uniform_bounds_Energy}
If $(u_n,v_n)$ is a least energy solution for $\I_n$, then 
\[
\I_n(u_n,v_n)\leq \tilde c.
\]
Furthermore,
\[
\int_\Omega (f(u_n)u_n+g(v_n)v_n)\, dx\leq \frac{2(2+\delta)}{\delta}\tilde c.
\]
\end{lemma}
\begin{proof}
Inequality \eqref{eq:G(s)_leq_G_n(s)} yields that $\I_n(u,v)\leq \tilde \I(u,v)$ for every $(u,v)\in H$. Hence, if $(\tilde u,\tilde v)$ is a ground state for $\tilde I$, then by Proposition \ref{prop:gs_caract3} one has
\[
\tilde c=\mathop{\sup_{t\geq 0}}_{\phi\in H^1_0} \tilde \I(t(u,v)+(\phi,-\phi))\geq \mathop{\sup_{t\geq 0}}_{\phi\in H^1_0}  \I_n(t(u,v)+(\phi,-\phi))\geq c_n=\I_n(u_n,v_n)
\]
and the first conclusion follows. Now, by taking in consideration 
\[
\int_\Omega \langle \nabla u_n,\nabla v_n\rangle\, dx-\int_\Omega F_n(u_n)-\int_\Omega G_n(v_n)\, dx=I_n(u_n,v_n)\leq C
\]
and
\[
2 \int_\Omega \langle \nabla u_n,\nabla v_n\rangle\, dx-\int_\Omega (f_n(u_n)+g_n(v_n))\, dx=I_n'(u_n,v_n)(u_n,v_n)=0,
\]
and by combining this with assumption ($fg3$), we have
\[
\int_\Omega (f_n(u_n)+g_n(v_n))\, dx\leq 2\tilde c+\frac{2}{2+\delta}\int_\Omega (f_n(u_n)u_n+g_n(v_n)v_n)\, dx.
\]
As $2/(2+\delta)<1$, the result follows.
\end{proof}

Observe that \eqref{eq:p_and_q_reduzidoGERAL} implies \eqref{eq:p_and_q_for_E^s_approach}. Thus, as seen in Subsection \ref{subset:Fractional}, there exists $0<s<2$ such that 
\[
E^s(\Omega)\hookrightarrow L^{p+1}(\Omega),\qquad E^{2-s}(\Omega)\hookrightarrow L^{q+1}(\Omega).
\]
are compact embeddings. Next, we prove uniform bounds in the space $E^s(\Omega)\times E^{2-s}(\Omega)$.
\begin{lemma}\label{lemma:bounds_in_HsHt}
Under the previous notations, there exists $C>0$ (independent of $n$) such that
\[
\|u_n\|_{E^s(\Omega)}+\|v_n\|_{E^{2-s}(\Omega)}\leq C.
\]
\end{lemma}
\begin{proof}
Multiplying equation $-\Delta u_n=g_n(v_n)$ by $A^{-(2-s)}A^s u_n\in E^{2-s}(\Omega)$, we obtain
\[
\|u_n\|_{E^s(\Omega)}^2=\int_\Omega g_n(v_n)A^{-(2-s)}A^s u_n \, dx.
\]
Given $\eps>0$, by using ($fg1$) we obtain the existence of $\eps',C>0$ such that
\begin{multline*}
\|u_n\|_{E^s(\Omega)}^2 
	\leq \eps \|v_n\|_{L^2}^2+\eps C\|A^{-(2-s)}A^su_n\|_{L^{q+1}}^2\\
	+C \left(\int_{\{|v_n|\geq \eps'\}} |g_n(v_n)|^{(q+1)/q}\right)^{q/(q+1)}\|A^{-(2-s)}A^su_n\|_{L^{q+1}}.
\end{multline*}
Now, we have that
\[
\|A^{-(2-s)}A^s u_n\|_{L^{q+1}}\leq C'\| A^{-(2-s)}A^s u_n\|_{E^{2-s}(\Omega)}=C'\|A^s u_n\|_{L^2}=C'\|u_n\|_{E^s(\Omega)};
\]
Moreover, $|g(s)|\leq C|s|^q$ $\forall |s|\geq \eps'$, for some C independent of $n$ (cf. \eqref{eq:uniform_bound_of_g_n_independent}), thus $|g(s)|^{(q+1)/q}\leq C|g(s)s|$. By choosing $\eps$ sufficiently small, and recalling Lemma \ref{lemma:uniform_bounds_Energy}, we obtain
\begin{align*}
\|u_n\|^2_{E^s(\Omega)}&=\int_\Omega g_n(v_n) A^{-(2-s)}A^s u_n\, dx \\
					&\leq \frac{1}{4}\left(\|u_n\|_{E^s(\Omega)}^2+\|v_n\|_{E^{2-s}(\Omega)}^2\right)+C''\|u_n\|_{E^s(\Omega)}.
\end{align*}
Since an analogous estimate holds true for $\|v_n\|_{E^{2-s}(\Omega)}^2$, the result follows.
\end{proof}

Finally, the following estimate implies that one can indeed work without loss of generality with $p,q<2^*-1$ (actually even with $p=q$).
\begin{lemma}
Under the previous notations, there exists $C>0$ such that
\[
\|u_n\|_\infty+\|v_n\|_\infty \leq C.
\]
In particular, $(u_n,v_n)$ is a least energy solution of \eqref{eq:system_with_f_and_g} for sufficiently large $n$.
\end{lemma}
\begin{proof}
We follow the proof of \cite[Theorem 1]{Sirakov}. From the choice of $s$ and by Lemma \ref{lemma:bounds_in_HsHt}, we have that 
\[
u_n \text{ is bounded in } L^{p+1}, \quad v_n \text{ is bounded in } L^{q+1}.
\]
Thus $f(u_n)$ is bounded in $L^{(p+1)/p}$ and $g_n(v_n)$ is bounded in $L^{(q+1)/q}$, and by elliptic regularity (see for instance \cite[Th. 9.5 \& Coro. 9.17]{GilbargTrudinger})
\[
u_n \text{ is bounded in } W^{2,(q+1)/q}, \quad v_n \text{ is bounded in } W^{2,(p+1)/p}.
\]
If either $(p+1)/p\geq N/2$ or $(q+1)/q\geq N/2$, then the result follows from the embedding $W^{2,s}\hookrightarrow L^\infty$ for $s\geq N/2$. Assuming the contrary, then by the embedding $W^{2,s}\hookrightarrow L^{Ns/(N-2s)}$, we conclude that
\[
u_n \text{ is bounded in } L^{N(q+1)/(Nq-2(q+1))}, \quad v_n \text{ is bounded in } L^{N(p+1)/(Np-2(p+1))}.
\]
We now iterate this procedure. Define the sequences
\[
p_{n+1}=\frac{Nq_n}{Np-2q_n}, \quad q_{n+1}=\frac{Np_n}{Nq-2p_n},
\]
with $p_0:=q+1$, $q_0:=p+1$. Whenever $Np-2q_n,Nq-2p_n>0$, we have that
\[
u_n \text{ is bounded in } W^{2,p_{n+1}/q}, \quad v_n \text{ is bounded in } W^{2,q_{n+1}/p}.
\]
We now prove that $p_n\to +\infty$, $q_n\to +\infty$, which shows that we need to make this bootstrap procedure a finite number of times only in order to obtain the desired conclusion. The fact that $(p,q)$ is below the critical hyperbola is equivalent to $p_0<p_1$, $q_0<q_1$; by induction one can easily show that both $p_n$ and $q_n$ are strictly increasing sequences. Suppose, in view of a contradiction, that $p_n\to l_1\in \R$, $q_n\to l_2\in \R$. Then
\[
l_2=\frac{N(pq-1)}{2(q+1)}.
\]
We claim that $N(pq-1)/(2q+2)<p+1$, which is in contradiction with the fact that $p+1=q_0<l_2$. To prove the claim take, for each $p$ fixed, the function
\[
h_p(q)=\frac{N(pq-1)}{2(q+1)}, \text{ for } 1<q<\frac{2p+N+2}{pN-2p-2}.
\]
We have $h_p'>0$, hence
\[
h_p(q)<h_p\left(\frac{2p+N+2}{pN-2p-2}\right)=p+1.
\]
\end{proof}

\begin{remark}\label{rem:p=q}
For later reference, let us stress that actually from the start one could have supposed, without loss of generality, that $1<p=q<2^*-1$.
\end{remark}
From the previous considerations, we learn that the following general result holds.

\begin{theorem}\label{thm:aprioribounds}
Under assumptions $(fg1)$--$(fg2)$--$(fg3)$, let $(u_n,v_n)$ be any sequence of solutions of the truncated system
\[
-\Delta u_n=g_n(v_n),\quad -\Delta v_n=f_n(u_n),\qquad u_n,v_n\in H^1_0(\Omega).
\]
If there exists $C>0$ such that $\I_n(u_n,v_n)\leq C$, then $\|u_n\|_\infty+\|v_n\|_\infty\leq C'$ for some constant $C'$ (and thus $(u_n,v_n)$ solves \eqref{eq:system_with_f_and_g} for large $n$).
\end{theorem}
We saw that the variational characterization provided by Proposition \ref{prop:gs_caract3} implies that $\I_n(u_n,v_n)\leq C$ for ground state solutions. In general, one can imagine that if in another situation one has a suitable variational characterization for a certain energy critical level, then the uniform bound on functionals at that energy level is also easily satisfied. This justifies the statement made in the beginning of this section saying that, in general, one does not lose generality by assuming $p,q<2^\ast-1$.

\begin{remark}
In the works by Ramos et al \cite{PistoiaRamosNeumann,PistoiaRamosDirichlet,RamosAPriori,RamosSoares,RamosYang}, $L^\infty$--bounds are obtained under a different assumption for the sequence $(u_n,v_n)$, and for more restrictive classes of functions $f$ and $g$. For the nonlinearities, consider
\begin{itemize}
\item[$(\widetilde{fg2})$] There exists $p,q$ satisfying \eqref{eq:p_and_q_reduzido} and constants $l_1,l_2>0$ such that
\[
\lim_{|s|\to \infty} \frac{|f(s)|}{|s|^{p}}=l_1,\qquad \lim_{|s|\to \infty} \frac{|g(s)|}{|s|^{q}}=l_2.
\]
\end{itemize}
Then for instance in  \cite[Theorem 3.3]{RamosAPriori} it is shown that the conclusion of {\rm Theorem~\ref{thm:aprioribounds}} also holds under assumptions $(fg1),(fg2'),(fg3)$ and supposing that $(u_n,v_n)$ is a critical point of $\I_n$ such that $m_{H^-}(u_n,v_n)\leq k$ for some $k\in \N$, where $m_{H^-}$ is the relative Morse index  \cite[Section 2.4]{Abbondandolo1}, \cite[Section 1]{Abbondandolo2}:
\[
m_{H^-}(u,v):=dim_{E^-}V^-:=dim(V^-\cap (E^-)^\perp)-dim(E^-\cap (V^-)^\perp),
\]
and $V^-$ represents the negative eigenspace of $\I_n''(u,v)$. 

At this point we would like to observe that although $\I(u,v)$ has always infinite Morse index, this is not the case for the reduced functional. For instance for ground state solutions $(u,v)$, the Morse index of ${\J}(\frac{u+v}{2})$ is one, and also $m_{H^-}(u,v)=1$ (cf. \cite[Example 3.2]{RamosAPriori}). In general, one has $m_{{\J}}(u)\leq m_{H^-}(u+\Psi_u,u-\Psi_u)$, see \cite[Lemma 3.1]{RamosAPriori}.
\end{remark}

\subsection{A family of reduced functionals depending on a parameter}\label{sec:reduced_with_lambda}

In some cases, it is useful to introduce a free parameter when we reduce the functional $\I$. We will consider two such situations in Subsection \ref{sec:symmetrybreaking} and Subsection \ref{sec:pertusym} ahead. Suppose without loss of generality $1<p=q<2^*$, and consider the family of functionals 
\[
\J_{\lambda} : H^{1}_{0}(\Omega) \to \R : u\mapsto \sup\left\{ \I\left(\lambda u + \psi,u-\frac{\psi}{\lambda}\right): \psi\in H^1_0(\Omega)\right\},\quad \lambda>0.
\]
As in the case $\lambda=1$, it is easily seen that $\J_{\lambda}(w) = \I(\lambda w+\psi_{\lambda , w},w-\frac{\psi_{\lambda , w}}{\lambda})$
for some unique $\psi_{\lambda , w}\in H^1_0(\Omega)$, and that the map $\theta : H^1_0(\Omega) \to H^1_0(\Omega)$, $w \mapsto \psi_{\lambda , w}$ is of class $C^{1}$, see Lemma \ref{lemma:Theta_is_C1}. 

By definition of $\psi_{\lambda , w}$, we have that
\begin{equation}\label{eq:reducI'}
\I'(\lambda w+\psi_{\lambda , w},w-\frac{\psi_{\lambda , w}}{\lambda})(\eta,-\frac{\eta}{\lambda}) = 0
\end{equation}
for every $\eta\in H^{1}_{0}(B)$. It then follows that
\begin{equation}
\label{eq:J'}
\J_{\lambda}'(w)\xi = I'(u,v)(\lambda\xi,\xi) =  I'(u,v)(\lambda\xi +\phi,\xi-\frac{\phi}{\lambda}),
\end{equation}
where $w:= (u+\lambda v)/2\lambda $ and $\phi\in H^{1}_{0}(\Omega)$ is arbitrary. Therefore, the map
$$H^1_0(\Omega)\to H^1_0(\Omega)\times H^1_0(\Omega), \quad w \mapsto (\lambda w+\psi_{\lambda , w},w-\frac{\psi_{\lambda , w}}{\lambda})$$ provides a homeomorphism between critical points of the reduced functional $\J_{\lambda}$ and critical points of the functional $\I$. Indeed, observe that for any $(\zeta,\xi)\in H^1_{0}(\Omega)\times H^1_{0}(\Omega)$, we have
\begin{multline*}
\I'(\lambda w+\psi_{\lambda , w},w-\frac{\psi_{\lambda , w}}{\lambda})(\lambda\zeta,\xi) = \I'(\lambda w+\psi_{\lambda , w},w-\frac{\psi_{\lambda , w}}{\lambda})(\lambda\frac{\zeta-\xi}{2},-\frac{\zeta-\xi}{2}) \\ + \I'(\lambda w+\psi_{\lambda , w},w-\frac{\psi_{\lambda , w}}{\lambda})(\lambda\frac{\zeta+\xi}{2},\frac{\zeta+\xi}{2}),
\end{multline*}
so that our claim follows.

Denoting by 
\begin{equation}\label{widetildeN_lambda:=}
{\mathcal N}_{\J_{\lambda}}:=\{w\in H^1_0(\Omega), \; w\neq 0: \J'_{\lambda}(w)w=0\}
\end{equation}
the Nehari manifold associated to $\J_{\lambda}$, we can define 
$$
c_{\lambda}:=\inf_{{N}_{\J_{\lambda}}}\J_{\lambda}.
$$
\begin{lemma} \label{clambda}
We have that
\begin{equation}\label{eq:c_lambda=MP_lambda}
c_{\lambda}=\inf_{\gamma\in \Gamma}\sup_{t\in [0,1]}J_{\lambda}(\gamma(t)),
\end{equation}
where 
\[
\Gamma:=\{ \gamma\in C([0,1]; H^1_0(\Omega)): \gamma(0)=0 \; \mbox{ and } \; \J_{\lambda}(\gamma(1))<0\}.
\] 
Moreover, the level $c_{\lambda}$ does not depend on $\lambda$, and $c_\lambda=c(\Omega)$.
\end{lemma}

\begin{proof}
This is essentially proved in \cite[Proposition 2.2]{PistoiaRamosDirichlet} in a more general situation where $\lambda$ is allowed to be a non constant function. We give a more direct proof here. 

We skip the proof of \eqref{eq:c_lambda=MP_lambda}, as it follows in a standard way. Now, let  $\theta_{\lambda}: H^1_0(\Omega)\to H^1_0(\Omega)$ be given by
$$
\theta_{\lambda}(w):=\frac{\lambda+1}{2}\, w+\frac{\lambda-1}{2\lambda }\, \psi_{\lambda, w}.$$
Then, as proved in \cite[Proposition 9, Step 3]{BonheureRamos}, $\theta_{\lambda}$ is a homeomorphism and
$$
\J_{\lambda}(\theta_{\lambda}^{-1}(w))\le \J_{1}(w),$$ for every $w\in H^1_0(\Omega)$. Now, given $w\in H^1_0(\Omega)\setminus\{0\}$, take $t_0>0$ so large that $\J_{1}(t_0w)<0$ and define $\gamma(\xi):=
\theta_{\lambda}^{-1}(\xi t_0 w)$ for $0\le \xi \le 1$. Then $\gamma\in \Gamma$ and
$$
\sup_{\xi\in [0,1]}\J_{\lambda}(\gamma(\xi))\le \sup_{\xi\in [0,1]}\J_{1}(\xi t_{0} w)\le \sup_{t>0}\J_{1}(t w),$$
implying that
$$
c_{\lambda}\le \sup_{t>0}\J_{1}(tw).$$
Since $w$ is arbitrary, we conclude that $c_{\lambda}\le c_{1}$. To show that $c_{1}\le c_{\lambda}$, one proceeds in a similar way.
\end{proof}

We have deduced yet another characterization of the ground energy level \linebreak $c(\Omega)$ defined in \eqref{eq:least_energy_level}. In particular, this yields
\begin{multline*}
c(\Omega) =\inf_{v\neq -\frac{u}{\lambda}}\sup \{\I(tu+\psi,tv-\frac{\psi}{\lambda}): t\geqslant 0,\; \psi\in H^1_0(\Omega)\}\\
 =\inf\{ \I(u,v)  \mid (u,v)\in H\setminus\{(0,0)\}, \; I'(u,v)(u+\psi,v-\frac{\psi}{\lambda})=0 \quad \forall \psi\in H^1_0(\Omega)\}.
\end{multline*}

\section{More on the symmetry properties of solutions}\label{sec:more_on_symmetry}


The questions about the symmetry of the solutions of a second order elliptic equation can be tackled either using reflexion methods and moving planes arguments as in Gidas et al. \cite{GidasNiNirenberg1979}, or symmetrization techniques as in Talenti \cite{Talenti}. The moving planes method was adapted for elliptic systems as \eqref{eq:main_system} by Troy \cite{Troy}, see also \cite{BuscaSirakov2000, deFigueiredoNodea1994,Shaker} and Remark \ref{rem:Troy}. Further contribution based on symmetrization techniques for a scalar equation often rely on the Polya-Szeg\"o inequality which asserts that the gradient of a Schwarz rearranged function $u^{*}$ has a smaller $L^{2}$-norm (other quantities can be considered as well) than the original function $u$. For higher order elliptic problems, and also for Hamiltonian elliptic systems, this approach by symmetrization cannot be applied in such a direct way. Indeed, if one thinks for instance of the treatment of the system \eqref{eq:main_system_particularcase} using the reduction by inversion, the functional framework is a Sobolev space requiring the existence of two weak derivatives and one can clearly produce examples of such functions whose Schwarz symmetric rearrangement does not possess two weak derivatives anymore. We have shown in Section \ref{sec:DualMethod} and Section \ref{sec:QuartaOrdem} that the right tool to apply symmetrization technique is the comparison principle due to Talenti, see Lemma \ref{simetrizacao}. We will give more insight on the use of this principle to get both complete and partial symmetry results using polarizations. 

%

\subsection{Working with polarization} \label{workingwithpolarization}
In this section, we show how the reduction by inversion framework allows to use polarization techniques to prove complete or partial symmetry results. 

Assume $H$ is a closed half-space in $\R^N$. We denote by  $\sigma_{H}: \R^N \rt \R^N$  the reflection with respect to the boundary $\partial H$ of $H$. For simplicity, we also put $\overline{x} = \sigma_H(x)$ for $x \in \R^N$ when the underlying half-space $H$ is understood. For a measurable function $w : \R^N \rt \R$ we define the polarization $w_H$ of $w$ relative to $H$ by
\[
w_H(x) =
\left\{
\begin{array}{l}
\max\{ w(x), w(\overline{x}) \}, \quad x \in H,\\
\min\{ w(x), w(\overline{x})\}, \quad x \in \R^N \menos H.
\end{array}
\right.
\]
We also denote $\overline{w}(x):=w(\overline{x})$. 

We consider the set $\mathcal{H}$ of all closed half-spaces $H$ in $\R^N$ such that $0 \in \partial H$. Given an unitary vector $e \in \R^N$, we denote by $\mathcal{H}_{e}$ the set of all closed half-spaces $H \in \mathcal{H}$ such that $e \in int(H)$ and we denote by $\mathcal{H}_*$ the set of all closed half-spaces $H$ in $\R^N$ such that $0 \in int(H)$.

Recall that a function $f:\R^N\to \R$ is said to be foliated Schwarz symmetric with respect to a unitary vector $e\in \R^N$ if it is axially symmetric with respect to the axis $\R e$ and nonincreasing in the polar angle $\theta= \arccos( x\cdot  e) \in [0, \pi]$. 

We mention that, up to our knowledge, the link between polarization and foliated Schwarz symmetry appeared first in \cite{SmetsWillem};  cf.  \cite[Theorem 2.6]{BartschWethWillem} for further results about the foliated Schwarz symmetry of least energy solutions of some second order elliptic equations with radial data. We also mention that some precursory works, as \cite{Ahlfors, Baernstein1, Baernstein2, BrockSolynin}, brought to light the relation between polarizations and rearrangements in many different settings. 

To our purposes we recall, without proving, the following useful characterization of a symmetric function by means of polarizations, and refer to the survey \cite{WethSurvey} for more details on the subject.

\begin{proposition}[]\label{Prop:polarizationXsymmetry} ${}$
\begin{enumerate}[(i)]
\item \cite[Proposition 2]{FerreroGazzolaWeth} Let $f \in C(\R^N)$. Then $f$ is Schwarz symmetric (with respect to the origin) if, and only if, $f = f_{H}$ for every $H \in \mathcal{H}_*$.
\item \cite[Lemma 2.6]{SmetsWillem}, \cite[Lemma 17]{BerchioGazzolaWeth} Let $f \in C(\R^N)$ and $e \in \R^N$ an unitary vector. Then $f$ is foliated Schwarz symmetric with respect to $e$ if, and only if, $f = f_H$ for every $H \in \mathcal{H}_{e}$.
\end{enumerate}
 \end{proposition}

The following result is essentially due to \cite{BrockSolynin} but we provide a rather elementary proof. Here, $B$ stands for the open ball in $\R^N$ centered at the origin and with radius one.

\begin{lemma}\label{laplaciano polarizado}
Let $f \in L^t(B)$, $1 < t  < \infty$, and $H \in \mathcal{H}$. Let $u$ and $v$ be the strong solutions of
\[
\left\{
\begin{array}{l}
-\Delta u  = f,  \quad -\Delta v  = f_H \quad in \quad $B$,\\
u, v=0 \quad{on} \quad \partial B.
\end{array}
\right.
\]
Then $v=v_H$ in $B $ and $v\geq u_H$ in $H\cap B$. Moreover, 
\begin{equation}\label{lemma_brock1}
\int_{B}u\varphi\, dx \leq \int_{B}v\varphi_H\, dx, \qquad \forall \varphi\in L^{\infty}(B).
\end{equation}
In particular, if $ f\geq 0$,
\begin{equation}\label{lemma_brock2}
\int_{B}u^s\varphi\, dx \leq \int_{B}v^s\varphi_H\, dx, \qquad \forall \varphi\in L^{\infty}(B), \; \varphi\geq 0, \; s>1.
\end{equation}
\end{lemma}
\begin{proof}
Without loss of generality, we can assume that $f$ is smooth. Since $-\Delta (v-\overline{v})=f_H-\overline{f_H}\geq 0$ in $H\cap B$ we deduce from the maximum principle that $v\geq \overline{v}$ in $H\cap B$. So $v=v_H$ in $B$. On the other hand, since, by definition, $f-f_H=\overline{f_H}-\overline{f}$, we have that $-\Delta ( \overline{u}+u-\overline{v}-v)=\overline{f}+f-f_H-\overline{f_H}=0$. It follows that $\overline{u}+u=\overline{v}+v $ in $B$; in particular, $u=v$ on $\partial H\cap B$. Then, by observing that $-\Delta(v-u)=f_H-f\geq 0$ in $H\cap B$ and $ -\Delta(v-\overline{u})=f_H-\overline{f}\geq 0$ in $H\cap B$, we conclude that $v\geq u$ in $H\cap B$ and $v\geq \overline{u} $ in $H\cap B$, that is $v\geq u_H$ in $H\cap B$.

Now, given $\varphi \in L^{\infty}(B)$, we must derive the inequality
$$
\int_{B}u\varphi \, dx=\int_{H\cap B}(u \varphi+\overline{u}\, \overline{\varphi})\, dx \leq \int_{H\cap B}(v \varphi_{H}  +  \overline{v}\, \overline{\varphi_H})\, dx=\int_{B}v \varphi_H\, dx.$$
By replacing $\overline{v}=u+\overline{u}-v$ and $\overline{\varphi_H}=\varphi+\overline{\varphi}-\varphi_H$ in the above expression, we find that the inequality reads 
$$
\int_{H\cap B}   [(\varphi_H-\varphi)(v-\overline{u})+(\varphi_H-\overline{\varphi})(v-u) ]\, dx\geq 0.$$
Clearly, this holds true since each of the four terms in parenthesis  is non-negative, and this establishes \eqref{lemma_brock1}.

Finally, in case $f\geq 0$, since moreover, $v=v_H$, the property \eqref{lemma_brock2} is a consequence of \eqref{lemma_brock1}, as follows from \cite[Lemma 9.1]{BrockSolynin} applied to the map $j(r)=r^{s}$. \end{proof}

We will also need a similar version of the above lemma for the case of the whole space $\R^N$. First, we recall that $-\Delta + I : W^{2,t}(\R^N) \rt L^t(\R^N)$ is an isomorphism for every $1 < t < \infty$. 


\begin{lemma}\label{laplaciano polarizadoSRN}
Let $1 < t  < \frac{N}{2}$ and set $r >0$ such that $\frac{t-1}{t} + \frac{1}{r} = \frac{N-2}{N}$. Let $r'>1$ such that $\frac{1}{r} + \frac{1}{r'}=1$.  Let $f \in L^t(\R^N)$, and $H$ any half-space in $\R^N$. Let $u$ and $v$ be the strong solutions of
\[
-\Delta u + u  = f,  \quad -\Delta v + v = f_H \quad in \quad \R^N.
\]
Then $v=v_H$ in $\R^N$ and $v\geq u_H$ in $H$. Moreover, 
\begin{equation}\label{lemma_brock1SRN}
\int_{\R^N}u\varphi\, dx \leq \int_{\R^N}v\varphi_H\, dx, \qquad \forall \, \varphi \in L^{r'}(\R^N).
\end{equation}
In particular, if $ f\geq 0$,
\begin{equation}\label{lemma_brock2SRN}
\int_{\R^ N}u^r\varphi\, dx \leq \int_{\R^N}v^r\varphi_H\, dx, \qquad \forall \, \varphi \in L^{\infty}(\R^N), \,\, \varphi\geq 0.
\end{equation}
\end{lemma}
\begin{proof}
We observe that $u,v \in W^{2,t}(\R^N) \hookrightarrow L^r(\R^N)$. The proof is similar to the proof of Lemma \ref{laplaciano polarizado} and so will be omitted here. 
\end{proof}

\subsection{Symmetry results for a system in $\R^N$ using polarization arguments}\label{subsec:SignSymmetryRN}
In this part we consider the system
\begin{equation}\label{problemRN}
\left\{
\begin{array}{l}
-\Delta u + u  = |v|^{q-1}v \quad \text{in} \quad \R^N,\\
-\Delta v + v  = |u|^{p-1}u \quad \text{in} \quad \R^N,
\end{array}
\right.
\end{equation}
and we assume that the pair $(p,q)$ satisfies the hypothesis \eqref{eq:(p,q)_for_Ederson_part}, which for convenience we recall here
\begin{equation}\label{eq:(p,q)_for_Ederson_partSRN}
p, q> 0, \qquad 1 > \frac{1}{p+1} + \frac{1}{q+1} > \frac{N-2}{N}.  \tag{H3} 
\end{equation}

It is proved in \cite[Theorems 1.8 and 1.9]{BonheureSantosRamosTAMS} that \eqref{problemRN} has a ground state solution and that any ground state solution of \eqref{problemRN} has definite sign, that is, $u,v>0$ in $\R^N$ or $u, v<0$ in $\R^N$ (for short, $uv>0$ in $\R^N$). It was also proved in \cite{BonheureSantosRamosTAMS} that at least one positive ground state solution of \eqref{problemRN} is Schwarz symmetric. However, it was left as an open problem whether every ground state solution of \eqref{problemRN} has radial symmetry or not. Here we give a positive answer to this question.

\begin{theorem}\label{Th:radialsymmetryRN}
Assume \eqref{eq:(p,q)_for_Ederson_partSRN}. Then \eqref{problemRN} has a ground state solution. Any ground state solution $(u,v)$ of \eqref{problemRN} is such that $uv>0$ in $\R^N$. Moreover, if $(u,v)$ is a positive ground state solution of \eqref{problemRN} then, up to a common translation, $u$ and $v$ are Schwarz symmetric.
\end{theorem}

Here our approach is based on the reduction by inversion as in Section \ref{sec:QuartaOrdem}, see also \cite{BonheureSantosRamosTAMS}. Set $L : = -\Delta + I$. Then \eqref{problemRN} is equivalent to 
\[
L (|Lu|^{\frac{1}{q}-1} Lu) = |u|^{p-1} u, \qquad u \in W^{2, \frac{q+1}{q}}(\R^N),
\]
and the study of ground state solutions of \eqref{problemRN} is then reduced to the study of minimizers for the best Sobolev constant
\begin{equation}\label{sobolevconstantRN}
\alpha_{p,q} : = \inf \left\{  \int_{\R^N} |Lu|^{\frac{q+1}{q}} dx; \; u \in W^{2, \frac{q+1}{q}}(\R^N), \; \| u \|_ {p+1} =1 \right\}.
\end{equation}

The results listed below are proved in \cite[Section 3]{BonheureSantosRamosTAMS}:
\begin{enumerate}[-]
\item there exists at least one minimizer to \eqref{sobolevconstantRN};
\item any such minimizer is such that $u, Lu >0$ in $\R^N$ or $u, Lu< 0$ in $\R^N$;
\item at least one minimizer is Schwarz symmetric.
 \end{enumerate}
 
Therefore, the conclusion of the proof of Theorem \ref{Th:radialsymmetryRN} reduces to proving the following proposition.

\begin{proposition}\label{Prop:RadialSymmetryMinimizerRN}
Assume \eqref{eq:(p,q)_for_Ederson_partSRN}. Let $u  \in W^{2, \frac{q+1}{q}}(\R^N)$ be a minimizer of \eqref{sobolevconstantRN} such that $u, Lu >0$ in $\R^N$. Then, up to translation, $u, Lu$ are Schwarz symmetric.
\end{proposition}

We denote by $T_{\R^N}: W^{2, \frac{q+1}{q}}(\R^N)\to \left(W^{2, \frac{q+1}{q}}(\R^N)\right)'$ the operator given by
\begin{equation}\label{operator_TRN}
\langle T_{\R^N}(u), \varphi \rangle = \int_{\R^N}| L u|^{\frac{1}{q} - 1} L u L \varphi  dx, \qquad \forall \, u, \varphi \in W^{2, \frac{q+1}{q}}(\R^N).
\end{equation}
Then $T_{\R^N}$ is a nonlinear homeomorphism, cf. \cite[Lemma 3.2]{Ederson2008} for a similar result. For every $w \in L^{\frac{p+1}{p}}(\R^N)$, the imbedding $W^{2, \frac{q+1}{q}}(\R^N)\hookrightarrow L^{p+1}(\R^N)$ guarantees that the map
\[
\varphi \mapsto \int_{\R^N} w \varphi \, dx, \qquad \varphi \in W^{2, \frac{q+1}{q}}(\R^N),
\]
defines a continuous linear functional on $W^{2, \frac{q+1}{q}}(\R^N)$, and so there exists a unique $u \in W^{2, \frac{q+1}{q}}(\R^N)$ such that 
$T_{\R^N}(u)=w$, that is
\begin{equation}\label{u w rieszRN}
\int_{\R^N} | L u|^{\frac{1}{q} - 1} L u L \varphi dx = \int_{\R^N} w \varphi \, dx, \qquad \forall \, \varphi \in W^{2, \frac{q+1}{q}}(\R^N).
\end{equation}

\begin{lemma} \label{lemma_radialSRN}
Let $H$ be any half-space in $\R^N$, $w\in L^{\frac{p+1}{p}}(\R^N)$ be nonnegative and $u, \widetilde{u} \in W^{2, \frac{q+1}{q}}(\R^N)$ be such that $T_{\R^N}(u)=w$ and  $T_{\R^N}(\widetilde{u})=w_H$. Then  $\langle T_{\R^N}(u),u\rangle\leq \langle T_{\R^N}(\widetilde{u}),\widetilde{u}\rangle$.
\end{lemma}
\begin{proof} Let $v$ and $\widetilde{v}$ be the strong solutions of
$$
L v  = w, \quad L \widetilde{v}  = w_{H}\quad \text{in} \quad \R^N.
$$
Then, $u$ and $\widetilde{u}$ are the strong solutions of
$$
L u  = |v|^{q-1}v, \quad L \widetilde{u}  =  |\widetilde{v}|^{q-1}\widetilde{v} \quad \text{in} \quad \R^N,
$$
and, by definition,
\[
\langle T_{\R^N}(u), u \rangle = \int_{\R^N} |L u|^{\frac{q+1}{q}} dx = \int_{\R^N}  |v|^{q+1} dx \quad \hbox{and} \quad \langle T_{\R^N}(\widetilde{u}), \widetilde{u} \rangle = \int_{\R^N} |\widetilde{v}|^{q+1} dx.
\]
The conclusion follows then from \eqref{lemma_brock2SRN} with $\varphi \equiv 1$.
\end{proof}

\begin{proof}[Proof of Proposition \ref{Prop:RadialSymmetryMinimizerRN} completed] By definition we have that
\[
\alpha_{p,q}:=\inf_{u\in W^{2, \frac{q+1}{q}}(\R^N), u\neq 0}\frac{\langle T_{\R^N}(u),u \rangle}{\left( \int_{\R^N}|u|^{p+1}\, dx \right)^{(q+1)/q(p+1)}}\cdot
\]
Let $\alpha_{p,q}$ be achieved by a (positive) function $u$ such that  $\int_{\R^N}u^{p+1}\, dx=1$.

\medbreak 

\noindent{\it Step 1:} For every half-space $H$ in $\R^N$,  $u_H$ is also a minimizer for $\alpha_{p,q}$. 

\noindent Indeed, since $u$ is a minimizer for $\alpha_{p,q}$, we have that $T_{\R^N}(u)=\alpha_{p,q}u^p$. Let $\widetilde{u}$ be such that $T_{\R^N}(\widetilde{u})=\alpha_{p,q}u_H^p$. Then, by Lemma \ref{lemma_radialSRN}, $\langle T_{\R^N}(u),u \rangle \leq \langle T_{\R^N}(\widetilde{u}),\widetilde{u} \rangle$. By using the H\"older inequality we deduce that
\begin{multline*}
\alpha_{p,q}= \alpha_{p,q}\int_{\R^N} |u|^{p+1} dx = \langle T_{\R^N}(u),u \rangle \leq \langle T_{\R^N}(\widetilde{u}),\widetilde{u} \rangle =\alpha_{p,q}\int_{\R^N}  \widetilde{u} u_H^p dx\\
\leq \alpha_{p,q}\left(  \int_{\R^N} \widetilde{u}^{p+1}dx\right)^{1/(p+1)}  \left(\int_{\R^N} u_H^{p+1}dx\right)^{p/(p+1)} = \alpha_{p,q}\left(  \int_{\R^N} \widetilde{u}^{p+1}dx\right)^{1/(p+1)}
\end{multline*}
yielding that  $\int_{\R^N} \widetilde{u}^{p+1}dx\geq 1$ and so
\begin{multline*}
\alpha_{p,q} \leq  \frac{\langle T_{\R^N}(\widetilde{u}),\widetilde{u} \rangle}{\left( \int_{\R^N}|\widetilde{u}|^{p+1}dx\right)^{(q+1)/q(p+1)}} 
\leq \alpha_{p,q} \frac{\left(  \int_{\R^N} \widetilde{u}^{p+1}dx\right)^{1/(p+1)}}{\left( \int_{\R^N}|\widetilde{u}|^{p+1}dx\right)^{(q+1)/q(p+1)}} \\
=\alpha_{p,q} \left( \int_{\R^N} \widetilde{u}^{p+1}dx\right)^{-1/q(p+1)}\leq \alpha_{p,q}.
\end{multline*}
It follows that $\int_{\R^N} \widetilde{u}^{p+1}=1$  and $\widetilde{u}$ is a minimizer for $\alpha_{p,q}$, so that $T_{\R^N}(\widetilde{u})=\alpha_{p,q}\widetilde{u}^p$.  Hence $\widetilde{u}=u_H$ and we conclude that $u_H$ is a minimizer for $\alpha_{p,q}$.

\medbreak
\noindent{\it Step 2:} For every half-space $H$ in $\R^N$, we have $u>\overline{u}$ in $int(H)$, $u<\overline{u}$ in $int(H)$ or else $u=\overline{u}$ in $H$.

\noindent Indeed, up to normalization, with $v:=|L u|^{\frac{1}{q}-1}L u$ and $w:=|L u_H|^{\frac{1}{q}-1}L u_H$, we have that $L u= v^q$, $L v=u^p$, $L u_H=w^q$, $L w=u_H^p$ in $\R^N$. In particular, we infer from the equations $L v=u^p$ and $L w=u_H^p$ that $w\geq v_H $ in $H$, cf.  Lemma \ref{laplaciano polarizadoSRN}. Then, since, by definition, $|u-\overline{u}|=2u_H-u-\overline{u}$ in $H$, we see that
$$
L(|u-\overline{u}|)=( (w^q-v^q)+(w^q-\overline{v}^q))\geq 0 \qquad \text{in} \quad int(H).
$$
This implies that either $u>\overline{u}$ in $int(H)$, $u<\overline{u}$ in $int(H)$ or else $u=\overline{u}$ in $H$. Going back to the system, we must have that either $v>\overline{v}$ in $int(H)$, $v<\overline{v}$ in $int(H)$ or else $v=\overline{v}$ in $H$ respectively.

\medbreak
\noindent{\it Step 3:} Up to translation, $u, Lu$ are Schwarz symmetric.
 
\noindent Up to translation, we may assume that $u(0) = \max_{x \in \R^N}u(x)$. Now take any half-space $H \in \mathcal{H_*}$. Then, from Step 2, we have that $u \geq \overline{u}$ in $H$, that is, $u = u_{H}$. Hence, by Proposition \ref{Prop:polarizationXsymmetry} (i), and going back to the system, it follows that $u, Lu$ are Schwarz symmetric.
\end{proof}

\subsection{A partial symmetry results}
We provide here an example where the reduction by inversion approach was used to derive a partial symmetry result for the ground state solutions of the Hénon type system
\begin{equation}\label{eq:henonsysPS}
\left\{
\begin{array}{l}
 - \Delta u  = |x|^{\beta}|v|^{q-1}v \quad \text{in} \quad B,\\
 -\Delta v = |x|^{\alpha} |u|^{p-1}u \quad \text{in} \quad B,\\
 u, v = 0 \quad \text{on} \quad \partial B,
 \end{array}
 \right.
\end{equation}
where $B$ is the open ball in $\R^{N}$ centered at the origin of radius one and $\alpha,\beta\geq0$. In this part we assume again that the powers $p,q$ satisfy the hypothesis \eqref{eq:(p,q)_for_Ederson_partSRN}. The procedure in this part is borrowed from \cite[Theorems 1.2 and 1.3]{BonheureSantosRamosJFA} from where we quote:

\begin{theorem}[\cite{BonheureSantosRamosJFA}]\label{GSHENON}
Assume \eqref{eq:(p,q)_for_Ederson_partSRN}. Then \eqref{eq:henonsysPS} has a ground state (classical) solution. In addition, any ground state solution of \eqref{eq:henonsysPS} has definite sign, i.e. either $u,v> 0$ in $B$ or $u,v < 0$ in $B$. 

Moreover, for any $\alpha, \beta\geq 0$, every (positive) ground state solution $(u,v)$ of \eqref{eq:henonsysPS} is such that $u$ and $v$ are both foliated Schwarz symmetric with respect to the same unit vector $ e \in \R^N$. Furthermore, either $u$ and $v$ are radially symmetric, or $u$ and $v$ are strictly decreasing in $\theta= \arccos( x\cdot  e) \in [0, \pi]$ for $0< |x| < 1$.
\end{theorem}

In this part we will present the arguments involved to prove the Schwarz foliated symmetry of (positive) ground state solutions of \eqref{eq:henonsysPS}. For the complete proof of Theorem \ref{GSHENON}, we refer to \cite{BonheureSantosRamosJFA}. For further results on systems like \eqref{eq:henonsysPS} we refer to \cite{Bidaut-VeronFrancoise,CalanchiRuf,deFigueiredoPeralRossi,deFigueiredoSantosMiyagaki,LiuYang} and references therein.

\medbreak

Arguing as in Section \ref{sec:QuartaOrdem}, the system \eqref{eq:henonsysPS} can be rewritten as the scalar equation 
\begin{equation}\label{EPS}
\Delta(| x|^{-\frac{\beta}{q}} | \Delta u |^{\frac{1}{q} -1}\Delta u)  = | x|^{\alpha}|u |^{p-1} u \quad \hbox{in} \quad B, \quad \hbox{with} \quad u, \Delta u = 0 \quad \hbox{on} \quad \partial B.
\end{equation}
and the study of ground state solutions of \eqref{problemRN} is then reduced to the study of minimizers for the Sobolev constant
\begin{equation}\label{sobolevconstantBOLA}
S : = \inf \left\{  \int_{B} |\Delta u|^{\frac{q+1}{q}} |x|^{-\frac{\beta}{q}} dx; \; u \in E_{\frac{q+1}{q}, \frac{\beta}{q}}, \; \int |u |^{p+1} |x|^{\alpha}dx =1 \right\},
\end{equation}
where, for each $1< s <\infty$ and $\gamma >0$ 
\[
E_{s, \gamma} := \left\{ u \in W^{2,s}(B) \cap W^{1,s}_0(B); \int_B | \Delta u |^{s} |x |^{-\gamma}dx < + \infty  \right\}.
\]

To simplify notation, we will denote $E_{\frac{q+1}{q}, \frac{\beta}{q}}$ simply by $E$. It was proved in \cite[Lemma 3.4]{BonheureSantosRamosJFA} that the nonlinear operator $T: E\to E'$
given by
\begin{equation}\label{operator_TBOLA}
\langle T(u), \varphi \rangle = \int_B| \Delta u|^{\frac{1}{q} - 1} \Delta u \Delta \varphi |x|^{-\beta/q} dx, \qquad \forall \, u, \varphi \in E.
\end{equation}
in a nonlinear homeomorphism. On the other hand, for every $w \in L^{\frac{p+1}{p}}(B)$, the imbedding $E \hookrightarrow L^{p+1}(B)$ guarantees that the map
\[
\varphi \mapsto \int_B w \varphi \, dx, \qquad \varphi \in E,
\]
defines a continuous linear functional on $E$. So there exists a unique $u \in E$ such that 
$T(u)=w$, that is
\begin{equation}\label{u w riesz}
\int_B | \Delta u|^{\frac{1}{q} - 1} \Delta u \Delta \varphi |x|^{-\beta/q} dx = \int_B w \varphi \, dx, \qquad \forall \, \varphi \in E.
\end{equation}

 The next lemma helps to complete the proof of Theorem \ref{GSHENON} in which concerns the symmetry of (positive) ground state solution.

\begin{lemma} \label{lemma_foliated}
Let $w\in L^{\frac{p+1}{p}}(B)$ be nonnegative and $u, \widetilde{u} \in E$ be such that $T(u)=w$ and  $T(\widetilde{u})=w_H$. Then  $\langle T(u),u\rangle\leq \langle T(\widetilde{u}),\widetilde{u}\rangle$.
\end{lemma}
\begin{proof} Let $v$ and $\widetilde{v}$ be the strong solutions of
$$
-\Delta v  = w, \quad -\Delta \widetilde{v}  = w_{H}\quad \mbox{ in } B,\qquad 
v, \widetilde{v}= 0 \quad \mbox{ on } \partial B.$$
Then, $u$ and $\widetilde{u}$ are the strong solutions of
$$
-\Delta u  = |x|^{\beta} v^q, \quad -\Delta \widetilde{u}  = |x|^{\beta} (\widetilde{v})^q\quad \mbox{ in } B,\qquad 
u, \widetilde{u}=0 \quad \mbox{ on } \partial B$$
and, by definition,
\[
\langle T(u), u \rangle = \int_B |\Delta u|^{\frac{q+1}{q}} |x|^{-\beta/q} dx = \int_B |x|^{\beta} v^{q+1} dx 
\]
and
\[
\langle T(\widetilde{u}), \widetilde{u} \rangle = \int_B |x|^{\beta} (\widetilde{v})^{q+1} dx.
\]
The conclusion follows then from \eqref{lemma_brock2} with $\varphi(x)=|x|^{\beta}$.
\end{proof}

\begin{proof}[Proof of the foliated Schwarz symmetry of ground state solution of \eqref{eq:henonsysPS}]
Let
\[
S:=\inf_{u\in E, u\neq 0}\frac{\langle T(u),u \rangle}{\left( \int_{B}|u|^{p+1}|x|^{\alpha}\, dx \right)^{(q+1)/q(p+1)}}\cdot
\]
Let $S$ be achieved by a (positive) function $u$ such that  $\int_{B}u^{p+1}|x|^{\alpha}\, dx=1$. 

\medbreak 
\noindent{\it Step 1:} For every half-space $H \in \mathcal{H}$,  $u_H$ is also a minimizer for $\alpha_{p,q}$. 

\noindent The proof follows the same procedure as in Step 1 of the last section, and so will be omitted here.

\medbreak
\noindent{\it Step 2:} For every half-space $H \in \mathcal{H}$, we have $u>\overline{u}$ in $int(H)\cap B$, $u<\overline{u}$ in $int(H)\cap B$ or else $u=\overline{u}$ in $H\cap B$.

\noindent Again the proof follows exactly as the proof of Step 2 in the last section, based on the strong maximum principle, and so will also be omitted here.

\medbreak
\noindent{\it Step 3:} Fix any point $x_0\in B\menos\{0\}$ with $u(x_0)=\max\{ u(x): x\in B, \, |x|=|x_0|\}$. Set $e = \frac{x_0}{|x_0|}$. Then $u, -\Delta u$ are Schwarz foliated symmetric with respect to $e$.

\noindent Indeed, let $H \in \mathcal{H}_e$. Then, by Step 2, it follows that $u \geq \overline{u}$ in $H \cap B$, that is, $u = u_{H}$. Hence, by Proposition \ref{Prop:polarizationXsymmetry} (ii), and going back to the system, it follows that $u, -\Delta u$ are foliated Schwarz symmetric with respect to $e$.
\end{proof}

\subsection{An example of symmetry breaking}\label{sec:symmetrybreaking}
We provide here an example where two different reduction approaches were used to derive a loss of symmetry of the ground state solutions. Namely we consider once again the Hénon type system \eqref{eq:henonsysPS}, where $B$ is the unit ball in $\R^{N}$, $\alpha,\beta>0$, and $(p,q)$ satisfy \eqref{eq:(p,q)_for_Ederson_partSRN}. 

First we show that a symmetry breaking can be deduced by direct energy comparison using the reduction by inversion from Section \ref{sec:QuartaOrdem}. Secondly we will use the Lyapunov-Schmidt type reduction from Section \ref{sec:Reduced Functional} to deduce a symmetry breaking in a different range of the parameters $\alpha,\beta,p,q$ from a Morse type argument, namely deducing some estimates from the positivity of the second derivative of the family of reduced functional $\J_{\lambda}$ with an adequate choice of the parameter $\lambda$. 


Arguing as in Section \ref{sec:QuartaOrdem}, the system \eqref{eq:henonsysPS} can be rewritten as the scalar equation 
\begin{equation}\label{E}
\Delta(| x|^{-\frac{\beta}{q}} | \Delta u |^{\frac{1}{q} -1}\Delta u)  = | x|^{\alpha}|u |^{p-1} u \quad \hbox{in} \quad B, \quad \hbox{with} \quad u, \Delta u = 0 \quad \hbox{on} \quad \partial B.
\end{equation}
The functional associated to this equation is
\[
J(u) = \frac{q}{q+1} \int_B | \Delta u |^{\frac{q+1}{q}} |x|^{-\frac{\beta}{q}} dx - \frac{1}{p+1}\int_B |u|^{p+1}|x|^{\alpha}dx, 
\]
defined on the functional space
\[
E_{s, \gamma} = \left\{ u \in W^{2,s}(B) \cap W^{1,s}_0(B): \int_B | \Delta u |^{s} |x |^{-\gamma}dx < + \infty  \right\},
\]
with $s=\frac{q+1}{q}$ and $\gamma=\frac{\beta}{q}$.
Radial solutions can be obtained by working with the same functional in the same Sobolev  space restricted to radially symmetric functions, that is
\[
E_{s, \gamma, {\rm rad}} = \left\{ u \in W^{2,s}_{{\rm rad}}(B) \cap W^{1,s}_0(B): \int_B | \Delta u |^{s} |x |^{-\gamma}dx < + \infty  \right\},
\]
still with $s=\frac{q+1}{q}$ and $\gamma=\frac{\beta}{q}$. In order to deduce a symmetry breaking of the least energy solution, the simplest strategy consists in providing estimates of the least energy level (among solutions or equivalently on the corresponding Nehari manifold). 
Assume 
$$\frac{N + \alpha}{p+1} + \frac{N + \beta}{q+1} > N-2.$$
It is proved in \cite{BonheureSantosRamosJFA} that there exist $a,b,\delta,\alpha_0>0$ such that, for every
$\alpha\geq \alpha_0$, $0\leq \beta\leq \delta \alpha$, we have
\begin{equation}\label{eq:crad}
a  \alpha^{\frac{(p+2)(q+1)}{pq-1}} \left(1+\beta^{\frac{p+1}{pq-1}}\right)\leq c_{{\rm rad}}\leq b \alpha^{\frac{(p+2)(q+1)}{pq-1}} \left(1+\beta^{\frac{p+1}{pq-1}}\right), 
\end{equation}
$c_{{\rm rad}}$ being the least energy level among radial solutions. On the other hand, It is shown in \cite[Theorem 2 (c)]{CalanchiRuf} that 
\begin{equation}\label{upper_estimate}
{c}_{\alpha,\beta} \leq C_0\alpha^{\frac{2(p+1)(q+1)-N(pq-1)}{pq-1}}\qquad (C_0>0, \; \beta\leq \alpha, \; \alpha\geq \alpha_0),
\end{equation}
where ${c}_{\alpha,\beta}$ is the least energy level among all (not necessarily radial) solutions of \eqref{eq:henonsysPS}.
Actually, in \cite{CalanchiRuf} it is assumed  $p>1$, $q>1$ and  $\beta<(q+1)N$ but a close inspection of their proof shows that \eqref{upper_estimate} remains valid as long as $\beta \leq \alpha$ and $\alpha$ is sufficiently large.
These estimates lead to the following conclusion: a symmetry breaking occurs when $\alpha$ and $\beta$ are comparable. 
\begin{theorem}\label{alpha_and_beta_comparable} Assume \eqref{eq:(p,q)_for_Ederson_part}, $N>1$ and there exists $C>0$ such that $\beta\leq \alpha\leq C \beta$ as $\alpha\to \infty$. Then  there exists $\alpha_0>0$ such that for $\alpha\geq \alpha_0$, no ground state solution of \eqref{eq:henonsysPS} is radially symmetric.
\end{theorem}
\begin{proof}
This follows from a simple comparison of the levels ${c}_{\alpha,\beta}$ and $c_{{\rm rad}}$.
\end{proof}
In the case where $\alpha$ and $\beta$ are no longer comparable, namely $\beta={\rm o}(\alpha)$ as $\alpha\to \infty$, the situation is much more delicate to handle. If $\beta$ is fixed, we still deduce a loss of symmetry when  $(p,q)$ is sufficiently close to the critical hyperbola, so that
\[
2(p+1)(q+1)-N(pq-1)< (p+2)(q+1).
\] 
This follows from Theorem \ref{alpha_and_beta_comparable} by taking the estimates \eqref{upper_estimate}
and \eqref{eq:crad} into account.

As shown in \cite{SmetsSuWillem}, the simplest way to prove symmetry breaking for the scalar Hénon equation is to observe that the ground critical level of the associated functional is asymptotically strictly smaller than the action on any radial solution; for the system \eqref{eq:henonsysPS} the situation is more tricky since both the corresponding ground critical levels, the radial one and the global one, may grow asymptotically at the same rate, see \cite{BonheureSantosRamosJFA}, and this is in great contrast with the case of a single equation, as treated by \cite{SmetsSuWillem}. We mention that the estimate \eqref{eq:crad} corrects the wrong estimate \cite[Theorem 2c]{CalanchiRuf}. 

To deduce a loss of symmetry in a regime where for instance $\beta\geq 0$ is fixed and $\alpha\to \infty$, another strategy can be exploited, namely another one from \cite{SmetsSuWillem} (adapted for the single equation) which is based on the computation of the second derivative of the underlying energy functional. Here the Lyapunov type reduction of Section \ref{sec:Reduced Functional} is really the good approach to be used since the other ones are not convenient for Morse index type arguments. 

\begin{theorem}
{\rm ($\beta={\rm o}(\alpha)$ and $\beta\to \infty$ or $\beta$ fixed)}  
Assume \eqref{eq:(p,q)_for_Ederson_part}, $p\geq 1$, $q\geq 1$, $N\geq 3$ and that $\beta={\rm o}(\alpha)$ and $\beta\to \infty$ as $\alpha\to \infty$ or $\beta$ is fixed but taken sufficiently large. Then  there exists $\alpha_0>0$ such that for $\alpha\geq \alpha_0$, no ground state solution of \eqref{eq:henonsysPS} is radially symmetric.

\end{theorem}

We still denote by $c_{\alpha,\beta}$ the radial ground state level associated to a least energy radial solution $(u_{\alpha,\beta},v_{\alpha,\beta})$. The main ingredient in the proofs of the last wo theorems is the following estimate 
\begin{equation}\label{calpha,betaleqslant}
c_{\alpha,\beta}\leq C_0 \left( \int_{B} |\nabla u_{\alpha,\beta}|^2\, dx \right)^{1/2}  \left( \int_{B}\frac{v_{\alpha,\beta}^2}{|x|^2}\, dx\right)^{1/2}\qquad \forall  \alpha, \beta\geq 0;
\end{equation}
where the constant $C_0=C_0(p,q,N)$ in  \eqref{calpha,betaleqslant} is independent of $\alpha$ and $\beta$. The proof is rather long and basically follows from the non negativity of the second derivative of the reduced functional $\J_\lambda$ at a minimizer with a good choice of $\lambda$. We refer to \cite{BonheureSantosRamosJFA}. Notice that it is in the proof of this estimate that one requires $p\geq 1$, $q\geq 1$ and $N\geq 3$. On the other hand, one can show \cite{BonheureSantosRamosJFA} that there exists $\alpha_0>0$ such that 
\begin{equation}\label{calpha,betageqslant}
\left( \int_{B} |\nabla u_{\alpha,\beta}|^2\, dx\right)^{1/2}  \left( \int_{B}\frac{v_{\alpha,\beta}^2}{|x|^2}\, dx\right)^{1/2}\leq
 \frac{C_1}{1+\sqrt{\beta}} \; c_{\alpha,\beta}  
 \qquad \forall \alpha\geq \alpha_0, \beta\geq 0
\end{equation}
 for some positive constant $C_1=C_1(p,q,N)$, as long as $\beta/\alpha\to 0$. In fact, it is sufficient to have  $\beta/\alpha\leq \delta $ for some $\delta<1/(p+2)$.
 
 By comparing \eqref{calpha,betaleqslant} and \eqref{calpha,betageqslant} we obtain a contradiction provided $\beta$ is taken sufficiently large so that $(u_{\alpha,\beta},v_{\alpha,\beta})$ is not a ground state solution. 
\section{Concentration phenomena}\label{sec:concentration}

Throughout this section we will assume $N\geq 3$. To start with, let us recall some concentration results for solutions of
\begin{equation}\label{eq:single_equation_conc}
-\eps^2\Delta u+V(x)u=f(u) \qquad \text{ in } \Omega,
\end{equation}
with $f(s)$ being a power type nonlinearity, superlinear at the origin and subcritical at infinity. One of the interesting phenomena concerning \eqref{eq:single_equation_conc} is the existence of concentration for some classes of solutions under different boundary conditions and assumptions on $V$. The task of enumerating all contributions to this subject would give a survey paper by itself, hence here we just state some of the most relevant papers whose statements have guided the corresponding results for Hamiltonian systems.

One of the first results regarding concentration phenomena for \eqref{eq:single_equation_conc} was the one by Floer and Weinstein \cite{FloerWeinstein}; for $\Omega=\R$, $f(s)=|s|^2s$ and $V$ bounded, the authors show (through a Liapunov-Schmidt reduction argument) that, given a non degenerate critical point of $V$, there exists a solution of \eqref{eq:single_equation_conc} which concentrates around that point. Later, Oh \cite{Oh1,Oh2,Oh3} uses a similar approach while extending the results to higher dimensions, dealing also with multispike solutions.

The first results with a non-degeneracy assumption on $V$ go back to the works by Rabinowitz \cite{Rabinowitz} and Wang \cite{Wang}. In \cite{Rabinowitz}, it is shown the existence of a ground state solutions when $\Omega=\R^N$ and $V$ satisfies $0<\inf_{x\in \R^N} V(x)<\liminf_{|x|\to \infty}V(x)$ for sufficiently small $\eps\to0$. In \cite{Wang}, it is proved that these solutions concentrate around a global minimum point of $V$, as $\eps\to 0$.

A further step in the study of these questions was done by del Pino and Felmer \cite{delPinoFelmerUmpico}, where through a penalization method the authors find solutions which concentrate around a local minimum of $V$ (possibly degenerate). This was extended in \cite{delPinoFelmerMultipicos} to find multiple spike solutions concentrating around a finite prescribed number of local minima of $V$, or around given topologically nontrivial critical points of $V$ \cite{delPinoFelmerTop2,delPinoFelmerTop1}. We refer to the introduction of \cite{RamosJMAA} for more references on this subject.

Similar phenomena where studied for $V(x)\equiv 1$ and $\Omega$ bounded. For the Dirichlet case, Ni and Wei \cite{NiWei} proved that ground state solutions concentrate at the point which is at the maximum distance from the boundary, whereas in the Neumann case Ni and Takagi \cite{NiTakagi} showed that they concentrate around a point of $\partial \Omega$ having maximum mean curvature. The proof of these two results were simplified and extended to more general nonlinearities in \cite{delPinoFelmerDirichletNeumann}, and we refer to the introduction of the latter paper for a more detailed description of the results and for an excellent review of the subject. We would like also to refer to \cite{AmbrosettiMalchiodiSurvey} for more recent results.

At this point, it is natural to ask whether these results extend to the Hamiltonian system
\begin{equation}\label{eq:anothersystem}
\left\{
\begin{array}{ll}
-\Delta u+V(x)u=g(v) & \text{ in } \Omega\\
-\Delta v+V(x)v=f(u) & \text{ in } \Omega.
 \end{array}
\right.
\end{equation}
The answer is yes, and this task was done mainly in the works of Ramos et al \cite{PistoiaRamosNeumann,PistoiaRamosDirichlet,RamosJMAA,RamosSoares,RamosTavares,RamosYang}. We will write ahead the exact statements, mainly because some of the results in the mentioned papers need clarification, and also because by using the arguments of Section \ref{sec:Reduced Functional}, we can make nowadays less restrictive assumptions on $f,g$. In all statements it is required that $f,g$ satisfy $(fg1)$--$(fg2)$--$(fg3)$.

The proofs in these papers use mainly the approach and ideas followed on Section \ref{sec:Reduced Functional}, and they \emph{do not} consist on simple adaptations of the arguments used in the single equation case. In addition to what has been already said thoughtout this paper, there are mainly three difficulties that one needs to take in consideration when passing to the system. Firstly, as far as we know no uniqueness result seems to be known for the limiting problem \eqref{eq:anothersystem} with $\Omega=\R^N$ and $V\equiv 1$, which is a crucial assumption in some of the papers dealing with the single equation case (see for instance \cite[Assumption $(f4)$]{delPinoFelmerMultipicos}). Secondly, as said before, the energy functional \eqref{eq:usual_action_functional} is strongly indefinite and the underlying linking theorems are of more complex nature; this fact makes energy estimates much harder to obtain. The third issue is the fact that for $(p,q)$ satisfying assumption \eqref{eq:p_and_q_reduzidoGERAL}, it might happen that either $p$ or $q$ is larger than the critical Sobolev exponent. However, if one has suitable  upper bounds for the energy, one can then argue as in Subsection \ref{sec:p_diferente_q}.

Consider the system
\begin{equation}\label{eq:sistema_concentration}
\left\{
\begin{array}{ll}
-\eps^2\Delta u + u=g(v) & \text{ in } \Omega\\
-\eps^2\Delta v+v=f(u) & \text{ in } \Omega
 \end{array}
\right.
\end{equation}
under Neumann
\begin{equation}\label{eq:Neumann_bcond}
\frac{\partial u}{\partial n}=\frac{\partial v}{\partial n}=0 \qquad \text{ on } \partial \Omega
\end{equation}
or Dirichlet boundary conditions:
\begin{equation}\label{eq:Dirichlet_bcond}
u=v=0 \qquad \text{ on }\partial \Omega.
\end{equation}

As far as we know, the first paper to obtain concentration phenomena in Hamiltonian system is due to \'Avila and Yang \cite{AvilaYang}, using the dual variational formulation we presented in Section \ref{sec:DualMethod} (related results using this approach can be found in \cite{AlvesSoaresYang,AlvesSoares}). Their results were later on improved in \cite{PistoiaRamosNeumann,RamosYang}.

\begin{theorem}[\cite{PistoiaRamosNeumann,RamosYang}] \label{thm:primeiroteorema_conc}
Let $\Omega$ be a bounded domain. Under $(fg1)$--$(fg2)$--$(fg3)$, there exists $\eps_0>0$ such that for any $0<\eps<\eps_0$ problem \eqref{eq:sistema_concentration}--\eqref{eq:Neumann_bcond} has non constant positive solutions $u_\eps,v_\eps$. Moreover, $u_\eps+v_\eps$ attains its maximum value at some unique  point $x_\eps\in \partial \Omega$. Up to a subsequence, $x_\eps\to \bar x\in \partial \Omega$ with
\[
H(\bar x)=\max_{x\in \partial \Omega} H(x),
\]
where $H(x)$ denotes the mean curvature at a point $x\in \partial \Omega$.
\end{theorem}

At this point we would like to stress the differences between the later statement and the actual statement made in \cite{PistoiaRamosNeumann,RamosYang}. First of all, in the last papers the assumptions of $f,g$ were more restrictive (for instance a requirement that $f^2(s)\leq 2f'(s)F(s)$ and $g^2(s)\leq 2g'(s)G(s)$ was made, among others). A close look at the proofs, however, show that, having proved the results of Section \ref{sec:Reduced Functional} under $(fg1)$--$(fg2)$--$(fg3)$, the only thing missing is to check that the following stronger result holds under these assumptions:
\begin{quote}
For any given Palais-Smale sequence $\I$, that is $\I(u_n,v_n)$ bounded with $\mu_n:=\|\I'(u_n,v_n)\|\to 0$, then
\[
\sup\{\I(t(u_n,v_n)+(\phi,-\phi):\ t\geq 0,\ \phi\in H^1_0 \}=\I(u_n,v_n)+\textrm{O}(\mu_n^2).
\]
\end{quote}
This was proved in \cite[Proposition 2.5]{RamosTavares} assuming only $(fg1)$--$(fg2)$--$(fg3)$.

The second and last difference is that the theorem we stated speaks about concentration of $u_\eps+v_\eps$, while \cite{PistoiaRamosNeumann,RamosYang} state that both $u_\eps$, $v_\eps$ concentrate at the \emph{same} unique local maximum $x_\eps$. It seems however that the proofs there do not imply such a strong result\footnote{This fact was also confirmed by M. Ramos in a private communication.}, but only the one we stated on Theorem \ref{thm:primeiroteorema_conc} (more precisely, step 4 in the argument of the proof of \cite[Theorem 3.1]{RamosYang} works only for the sum $u_\eps+v_\eps$). Alternatively, one can prove the weaker result that $u_\eps$ and $v_\eps$ admit global maximum points $x_\eps$ and $y_\eps$, and these satisfy $|x_\eps-y_\eps|/\eps\to 0$ as $\eps\to 0$. In particular, $x_\eps,y_\eps\to \bar x\in \partial \Omega$ which maximizes the mean curvature at the boundary.

Similar comments can be made for the remaining statements of this section, and we refer to \cite{RamosTavares} for more comments regarding the assumptions on $f$ and $g$, and \cite[Section 5]{TavaresThesis} for more details concerning both questions.

As far as Dirichlet boundary conditions are concerned, the generalization of \cite{NiWei} for Hamiltonian systems is the following.

\begin{theorem}[\cite{PistoiaRamosDirichlet}]
Let $\Omega$ be a bounded domain. Under $(fg1)$--$(fg2)$--$(fg3)$, there exists $\eps_0>0$ such that for any $0<\eps<\eps_0$ the problem \eqref{eq:sistema_concentration}--\eqref{eq:Dirichlet_bcond} admits a positive ground state solution $(u_\eps,v_\eps)$, which satisfies the following properties, as $\eps\to 0$:
\begin{enumerate}
\item[(i)] $u_\eps+v_\eps$ attains its maximum value at some unique point $x_\eps\in \Omega$;
\item[(ii)] $d_\eps:=\textrm{dist}(x_\eps,\partial \Omega)\to \max_{x\in \Omega}\textrm{dist}(x,\partial \Omega)$;
\item[(iii)] $c_\eps(\Omega)=\eps^N\left(c(\R^N)+e^{-2(1+\textrm{o}(1))d_\eps/\eps}\right)$,
\end{enumerate}
where $c_\eps(\Omega)$ denotes the ground state level of \eqref{eq:sistema_concentration}, and $c(\R^N)$ denotes the groundstate level of the limiting problem \eqref{eq:sistema_concentration} with $\eps=1$ and $\Omega=\R^N$.
\end{theorem}

Next, we turn to 
\begin{equation}\label{eq:sistema_concentration_with_V}
\left\{
\begin{array}{ll}
-\eps^2\Delta u + V(x)u=g(v) & \text{ in } \Omega\\
-\eps^2\Delta v+V(x)v=f(u) & \text{ in } \Omega.
 \end{array}
\right.
\end{equation}
In case $\Omega=\R^N$, we have the following generalization of \cite{Rabinowitz,Wang}:
\begin{theorem}[\cite{RamosSoares}]
Let $\Omega=\R^N$. Assume that the nonlinearities $f$ and $g$ satisfy $(fg1)$--$(fg2)$--$(fg3)$, and that the potential $V$ is continuous and satisfies
\[
0<\inf_{x\in \R^N}V(x)<\liminf_{|x|\to \infty} V(x).
\]
Then there exists $\eps_0>0$ such that for any $0<\eps<\eps_0$ there exists a positive ground state solution $(u_\eps,v_\eps)$ of \eqref{eq:sistema_concentration_with_V} such that $u_\eps+v_\eps$ attains its maximum value at some unique point $x_\eps\in \R^N$. Moreover, $\{x_\eps\}_\eps$ is bounded and, up to a subsequence, it converges to $x_0\in \R^N$ satisfying
\[
V(x_0)=\min_{x\in \R^N} V(x).
\]
\end{theorem}

Finally, we want to state the generalization for Hamiltonian systems of the results by del Pino and Felmer \cite{delPinoFelmerUmpico,delPinoFelmerMultipicos}. Let us assume that $V$ is locally H\"older continuous and 
\begin{enumerate}
\item[$(V1)$] $V(x)\geq \alpha>0\text{ for all } x\in \Omega$;
\item[$(V2)$] there exists bounded domains $\Lambda_i$, mutually disjoint, compactly contained in $\Omega$ ($i=1,\ldots,k$) such that
\[
\inf_{\Lambda_i} V<\inf_{\partial \Lambda_i} V
\]
(that is, $V$ admits at least $k$ local strict minimum points, possibly degenerate). 
\end{enumerate}

\begin{theorem}[{\cite{RamosTavares},\cite[Section 5]{TavaresThesis}}]\label{thm:ramostavares}
Take $\Omega$ a regular (eventually unbounded) domain. Assume that $(fg1)$--$(fg3)$ and $(V1)$--$(V2)$ holds. Then by taking $\eps>0$ small enough we have that \eqref{eq:sistema_concentration_with_V} admits a positive solution $(u_\eps,v_\eps)$ having the following properties:
\begin{enumerate}
\item[$(i)$] $u_\eps+v_\eps$ possesses exactly $k$ local maximum points $x_{i,\eps}\in \Omega_i$, $i=1,\ldots,k$;
\item[$(ii)$] $u_\eps(x_{i,\eps})+v_\eps(x_{i,\eps})\geq b>0$, and $V(x_{i,\eps})\to \inf_{\Lambda_i} V$ as $\eps\to 0$;
\item[$(iii)$] $u_\eps(x),v_\eps(x)\leq \gamma e^{-\frac{\beta}{\eps}|x-x_{i,\eps}|}$, $\forall x\in \Omega\setminus \cup_{j\neq i} \Lambda_j$;
\end{enumerate}
for some constants $b,\gamma,\beta>0$.
\end{theorem}
\begin{proof}[Comments on the proof] As explained in Section \ref{sec:Reduced Functional}, one can work without loss of generality with $1<p=q<2^*$. We will denote by $\mathcal{J}_\eps$ the reduced functional associated with \eqref{eq:sistema_concentration_with_V} in the sense of Subsection \ref{subset:reducedFunctional}. 
Fix bounded domains $\widetilde \Lambda_i$ ($i=1,\ldots,k$), such that $\Lambda_i\Subset \widetilde\Lambda_i $. Take cut-off functions $\phi_i$ such that $\phi_i=1$ in $\Lambda_i$ and $\phi_i=0$ in $\R^N\setminus \widetilde \Lambda_i$. Since $\Omega$ might be an unbounded domain, following \cite{delPinoFelmerUmpico,delPinoFelmerMultipicos}, one can truncate the functions $f$ and $g$ outside $\Lambda_i$ in such a way that one recovers the Palais-Smale condition. Consider the following Nehari--type set
\[
\Ncal_\eps:=\left\{u\in H^1_0(\Omega)\mid \mathcal{J}_\eps'(u)(u\phi_i)=0,\ \text{ and } \int_{\Lambda_i}u^2\, dx>\eps^{N+1} \right\}.
\]
Roughly speaking, this set (which can be proved to be a manifold) localizes the functional $\mathcal{J}_\eps$ near each $H^1_0(\Lambda_i)\times H^1_0(\Lambda_i)$. The technical condition $\int_{\Lambda_i} u^2\, dx>\eps^{N+1}$ insures that the set is closed (actually, one proves that $\int_{\Lambda_i} u^2\, dx\geq \eta \epsilon^N$ for some $\eta>0$). It can be proved that $c_\eps=\inf_{\Ncal_\eps}\mathcal{J}_\eps$ is a critical point, corresponding to a solution of \eqref{eq:sistema_concentration_with_V} satisfying all the desired properties for sufficiently small $\eps>0$. As it was said before, we would like to observe that the statement in \cite{RamosTavares} is slightly incorrect, because it concludes that $u_\eps$ and $v_\eps$ have \emph{common} local maximums. The veracity of this stronger statement is not known to hold; this is related with the fact that the if $(u,v)$ is a \emph{nontrivial} solution of the limiting problem 
\begin{align*}
&-u''-\frac{N-1}{r} u'+u=g(v),\quad -v''-\frac{N-1}{r}v'+v=f(u),\\  &Nu''(0)=u(0)-g(v(0)),\quad Nv''(0)=v(0)-g(u(0)),
\end{align*} 
then one can only guarantee that either $u''(0)\neq 0$ or $v''(0)\neq 0$. The full proof of the statement of Theorem \ref{thm:ramostavares} is presented in \cite[Section 5]{TavaresThesis}.
\end{proof}

We would like to close this section referring to the work of Ramos \cite{RamosJMAA}, where the author (under the dimensional restriction $3\leq N\leq 6$) exhibits solutions of \eqref{eq:sistema_concentration_with_V} which concentrate around a prescribed critical point of $V$, which is not necessarily a minimum. It remais an open question whether this extends to higher dimensions or if there exist multi-peak solutions of the system for small $\eps>0$ concentrating around topologically nontrivial critical points of $V$ (in the sense of \cite[Theorem 1.2]{delPinoFelmerTop2}).

%
%

\section{Multiplicity results in the spirit of the Symmetric Mountain Pass Lemma}\label{sec:MultiplicityResults}

Let $\Omega$ be a bounded domain of $\R^N$, $N\geq 3$. It is well known that the superlinear and subcritical problem
\begin{equation*}
-\Delta u = |u|^{p-1}u \; \mbox{ in } \Omega, \qquad u=0 \;  \mbox{ on } \partial \Omega,
\end{equation*}
possesses infinitely many solutions in $H^{1}_{0}(\Omega)$ as follows from the natural $\IZ_2$-symmetry and the symmetric Mountain Pass Lemma of Ambrosetti and Rabinowitz \cite{AmbrosettiRabinowitz}. 
The non homogeneous problem
\begin{equation}\label{-Delta u = |u|^p-2u+h(x)}
-\Delta u = |u|^{p-1}u + h(x)\; \mbox{ in } \Omega, \qquad u=0 \;  \mbox{ on } \partial \Omega,
\end{equation}
where $h\in L^{2}(\Omega)$, can therefore be seen as a (large) perturbation of a symmetric situation and thus a large number of solutions is expected. One can indeed obtain infinitely many solutions, provided the growth range of the nonlinearity is suitably restricted. Namely, Bahri and Berestycki \cite{BahriBerestycki}, Struwe \cite{Struwe}, and, with a different approach, Rabinowitz \cite{Rabinowitz2, RabinowitzBook} proved the existence of infinitely many solutions for problem \eqref{-Delta u = |u|^p-2u+h(x)} under the restriction
\begin{equation}\label{frac2p + frac1p-1 > frac2N-2N}
p>1,\qquad \frac{2}{p+1} + \frac{1}{p} > \frac{2N-2}{N},
\end{equation}
while, later on, Bahri and Lions \cite{BahriLions} and Tanaka \cite{Tanaka} (see also \cite{KajikiyaTanaka}) showed that it is sufficient to assume
\begin{equation}\label{p<frac2N-2N-2}
p>1,\qquad p+1<\frac{2N-2}{N-2}.
\end{equation}
The main ingredient in Bahri and Lions \cite{BahriLions} and Tanaka \cite{Tanaka} is the use of the Morse index leading to more precise estimates and a better conclusion.
Moreover, assuming the ``natural" growth restriction $(p+1)(N-2)<2N$, Bahri \cite{Bahri} proved that there is an open dense set of functions $h\in H^{-1}(\Omega)$ for which
\eqref{-Delta u = |u|^p-2u+h(x)} admits infinitely many weak solutions.

For the corresponding Hamiltonian elliptic system 
\begin{equation}\label{S2}
\left\{
\begin{array}{rcll}
-\Delta u & = & \left| v\right|^{q-1} v + k(x) & \hbox{in} \,\, \Omega, \\
-\Delta v & = & \left| u\right|^{p-1} u  + h(x) & \hbox{in} \,\, \Omega, \\
u, v &= & 0 & \hbox{on} \,\, \partial \Omega,
\end{array}
\right.
\end{equation}
the symmetric case $h(x)\equiv k(x)\equiv0$ has been studied by several authors. By means of a Galerkin type approximation combined with the approach presented in Subsection \ref{subset:Fractional}, one can reduce the strongly indefinite functional to a semidefinite situation. We sketch this approach in Section \ref{se:Galer}. A different approach to the problem of symmetric indefinite functional was given by Angenent and van der Vorst in \cite{AngenentvanderVorst}, who applied Floer's version of Morse theory to Hamiltonian elliptic systems, in the spirit of \cite{BahriLions} (see also \cite{AngenentvanderVorst2000}).

We show in Subsection \ref{sec:symreduc} and Subsection \ref{sec:symfount} ahead that the variational approaches presented in the preceding sections allow one to derive the simplest proofs for the symmetric situation. Moreover, in the case where $h(x)\not\equiv 0$ and $k(x)\not \equiv0$, the reduction method of Section \ref{sec:Reduced Functional} is successful to obtain the equivalent of the Bahri-Lions' result for the scalar equation (for $p,q>1$). Indeed, Rabinowitz's approach mainly relies on an estimate of the deviation from symmetry and the use of an auxiliary functional. It turns out that the reduced functional $\J$ suits very well in Rabinowitz's approach (with some adaptation). Moreover, Morse index information \emph{a la} Bahri-Lions can be considered since with $\J$ we recover the geometry of the single equation case functional.

In dimension $1$ (see Subsection \ref{sec:pertusym}), we show that the reduction by inversion allows to treat the more general case $pq>1$. 

It must be stressed that since in general no a priori bounds for positive solutions are known to hold for \eqref{S2}, the results in this section do not give any information about the sign of the infinitely many solutions. In Subsection \ref{sec:infinitely_sign_changing} we will discuss this in more detail, providing also multiplicity of sign-changing solutions.

\subsection{The direct approach and Galerkin type approximation}\label{se:Galer}

In this section we consider the symmetric problem
\begin{equation}\label{symmsystem}
\left\{
\begin{array}{rcll}
-\Delta u & = & \left| v\right|^{q-1} v  & \hbox{in} \,\, \Omega, \\
-\Delta v & = & \left| u\right|^{p-1} u  & \hbox{in} \,\, \Omega, \\
u, v &= & 0 & \hbox{on} \,\, \partial \Omega,
\end{array}
\right.
\end{equation}
following closely the presentation by Tarsi \cite{Tarsi}. We will use the fractional Sobolev space approach of Subsection \ref{subset:Fractional}, and so we take $(p,q)$ satisfying \eqref{eq:p_and_q_for_E^s_approach}, which yields the existence of $0<s<2$ so that the functional $\I_s$ in \eqref{eq:the_use_of_s} is well defined. We show that infinitely many solutions can be found as critical points of $\I_s$ by means of a version of the symmetric Mountain Pass Lemma of Ambrosetti and Rabinowitz \cite{AmbrosettiRabinowitz}, valid for strongly indefinite functionals. Let us introduce some notations.

Take a Banach space $E$ with norm $\|\cdot\|$. Suppose that $E=E^+\oplus E^-$ with both $E^+$,
$E^-$ having infinite dimension, spanned respectively by $(e^+_i)_i$ and $(e^-_i)_i$. Set, for $n,m\in \N$
\begin{equation}\label{defX_n}
X_n=\spann \{e^+_1,...,e^+_n\} \oplus E^-, \qquad
X^m=E^+\oplus \spann \{e^-_1,...,e^-_m\},
\end{equation}
and let $(X^m)^{\bot}=\spann \{e_{m+1}^-,e_{m+2}^-,\ldots\}$ denote the complement of $X^m$ in $E$. For
a functional $I\in\mathcal{C}^1(E,\mathbb{R})$, we define $I_n:=I|_{X_n}$ as
the restriction of $I$ on $X_n$. 
Then we have the following theorem due to de Figueiredo and Ding \cite[Proposition 2.1]{deFigueiredoDing}, see also Bartsch and de Figueiredo \cite{BartschdeFigueiredo}.
\begin{theorem} \label{simMP}
Let E be as above and let $I\in \mathcal{C}^1(E, \mathbb{R})$ be even
with $I(0)=0$. In addition, suppose that for each $m\in \mathbb{N}$,
the following conditions hold:
\begin{itemize}
\item[(I$_1$)] there exists $R_m >0$ such that $I(z)\leq 0$
 for all $z\in X^m$ with $\|z\| \geq R_m$; 
\item[(I$_2$)] there exist $r_m>0$ and $a_m \rightarrow +\infty$ such that
$I(z)\geq a_m$ for all $z\in (X^{m-1})^{\bot}$
with $\|z\|=r_m$;
\item[(I$_3$)] $I$ is bounded from
above on bounded sets of $X^m$;
\end{itemize}
 and that
 \begin{itemize}
\item[(I$_4$)] $I$ satisfies the
$(PS)_c^{*}$ condition for any $c\geq 0$, that is, any sequence
$\{ z_n \} \subset E$ such that $z_n \in X_n$
for any $n \in \mathbb{N}$, $I(z_n) \rightarrow c$, and
$I'_n(z_n) \equiv I|_{X_n}'(z_n)
\rightarrow 0$ as $n \rightarrow +\infty$ possesses a convergent
subsequence.
\end{itemize}
Then the functional $I$ possesses an unbounded sequence $\{c_m\}$ of critical values.
\end{theorem}

\begin{proof}[Idea of the proof] The sequence of critical values can be constructed by means of
a Galerkin approximation. Using the previous notations,
set
\begin{equation*}\label{defB_k}
B^m:= \left\{z\in X^m:\| z\| \leq R_m \right\}
\end{equation*}
as being the ball of radius $R_m$ in $X^m$, take
\begin{equation*}\label{defB_k^n}
B^m_n:= B^m \cap X_n = \left\{z\in X^m \cap X_n:\| z\| \leq R_m \right\},
\end{equation*}
and define the following sets of continuous maps
\begin{equation*}\label{defGamma_k^n}
\Gamma_n^m:= \left\{h\in \mathcal{C} (B^m_n, X_n): h(-z)=-h(z)\ \forall z\in B^m_n,%
\quad h(z)=z\ \forall z\in \partial B^m_n \right\}.
\end{equation*}
Finally define the values
\begin{equation*}\label{c_k^n}
c_n^m:=\inf_{h\in \Gamma _n^m}{\sup_{z \in B_n^m}{I(h(z))}}.
\end{equation*}
It can be proved that, for sufficiently large $m\in \mathbb{N}$, the sequences $c_n^m$ converge to
critical values $c^m$ of the functional $I$ as $n\to +\infty$. Thus the limits
\begin{equation*}\label{defc_k}
c_m:= \lim_{n\rightarrow +\infty} c_n^m
\end{equation*}
are critical values of the symmetric functional $I$ for large $m$.
\end{proof}

Other versions of the same theorem are known, where the $(PS)_c^{*}$
condition is replaced by other variants, or by the usual Palais-Smale condition (see for instance
\cite{Bartsch, BenciRabinowitz, Ding} and references therein).

We briefly present the functional framework in which the functional $\I_s$ associated to the system \eqref{symmsystem} satisfies the hypotheses of Theorem
\ref{simMP}. First, fix again in $H^1_0 (\Omega )$ a system of orthogonal and $L^2$-normalized eigenfunctions $\phi
_1, \phi _2, \phi _3,...$, of $-\Delta $, $\phi _1 >0$,
corresponding to positive eigenvalues $\lambda _1 <\lambda _2 \leq
\lambda _3 \leq ...\uparrow +\infty $, counted with their
multiplicity. Remember the definition of 
\[
A^su=A^s\left( \sum_{n=1}^\infty a_n \phi_n \right)=\sum_{n=1}^\infty \lambda_n^{s/2}a_n \phi_n
\]
defined in the space $E^s(\Omega)$, see \eqref{eq:Es}. Like in Subsection \ref{subset:Fractional}, we use the notation $E_s=E^s(\Omega)\times E^t(\Omega)$ with $t=2-s$.
In order to apply Theorem \ref{simMP}, one uses  the decomposition $E_s=E_s^+\oplus E_s^-$, with
\[
E_s^+=\{(u,A^{-t}A^s u):\  u\in E^s(\Omega)\}, \qquad  E_s^{-}=\{(u,-A^{-t}A^s u):\ u\in E^s(\Omega)\},
\]
and basis $(\phi_n,\pm A^{-t}A^s \phi_n)_n$. For the details we refer to Tarsi \cite[Theorem 1.1]{Tarsi} or to de Figueiredo and Ding \cite{deFigueiredoDing} for a related but different problem. 

\subsection{The Symmetric Mountain Pass Lemma combined with the Lyapunov-Schmidt type reduction}\label{sec:symreduc}
We now show how the reduction approach of Section \ref{sec:Reduced Functional} combined with the Symmetric Mountain Pass Lemma \cite{AmbrosettiRabinowitz} provides a short proof of the existence of infinitely many solution in a symmetric framework. Observe that we have to restrict ourselves to the case $p,q>1$.
\begin{theorem} \label{h=k=0} 
Assume $(p,q)$ satisfies \eqref{eq:p_and_q_reduzidoGERAL}. 
Then the system \eqref{symmsystem} admits an unbounded sequence of solutions $(u_k,v_k)_{k}\subset H^1_0(\Omega)\times H^1_0(\Omega)$.
\end{theorem}
\begin{proof}
We can assume that $1<p= q<2^*-1$ (cf. Remark \ref{rem:p=q} or the discussion ahead) so that we can work with the functional space $H^1_0(\Omega)\times H^1_0(\Omega)$.
Remember from \eqref{eq:reducedfunctional_J} the definition of reduced functional
\[
\J(u) = \sup\left\{ \I\left(u + \psi,u-\psi\right): \psi\in H^1_0(\Omega)\right\}=\I(u+\Psi_u,u-\Psi_u),
\]
so that 
$$\J(u) \geq \I(u,u) \geq \|u\|_{H^1_0}^2$$ 
provided $ \|u\|=\rho$ with $\rho>0$ small enough. 
Now, take a finite dimensional subspace $X\subset H^1_0(\Omega)$. Assume by contradiction that there exists an unbounded sequence $(u_n)_n\subset X$ such that 
$$\liminf_{n\to\infty}\J(u_n)>-\infty.$$  
By computing $\J(u_n)$, we easily see that the sequence $(\|\Psi_{u_{n}}\|_{H^1_0}/\|u_n\|_{H^1_0})_n$ is bounded and
$$\lim_{n\to\infty}\int \left| \frac{u_n}{\|u_n\|_{H^1_0}}\pm \frac{\Psi_{u_{n}}}{\|u_n\|_{H^1_0}}\right| ^{p} = 0.$$
It then follows that $$\lim_{n\to\infty}\int \left|\frac{u_n}{\|u_n\|_{H^1_0}}\right| ^{p} = 0,$$ which is impossible since $X$ has finite dimension. Now since $\J$ satisfies the Palais-Smale condition and it is an even functional, we can therefore apply the $\IZ_2$-version of the Mountain Pass Theorem \cite[Theorem 2.8 \& Corollary 2.9]{AmbrosettiRabinowitz} to the functional $\J$, and the conclusion in the case $1<p= q<2^*-1$ follows from Lemma \ref{lem:homeo}.

Next, we observe that assuming $1<p=q < 2^*-1$ is not restrictive. Indeed, if for instance $p<2^*-1<q$, we define $g_{n}(s)$ as in Subsection \ref{sec:p_diferente_q}. Since $p<2^*-1$, extending the case of pure powers to our new settings, it is easily seen that for every $n\in \mathbb N$, the modified system 
\begin{equation}\label{eq:another_eq_with_n}
\left\{
\begin{array}{ll}
-\Delta u=g_{n}(v) & \text{ in } \Omega\\
 -\Delta v=|u|^{p-1}u & \text{ in } \Omega\\
 u,v=0  & \text{ on } \partial \Omega
 \end{array}
\right.
\end{equation}
has an unbounded sequence of solutions $(u_k,v_k)_{k}\subset H^1_0(\Omega)\times H^1_0(\Omega)$. Finally, arguing as in Section \ref{sec:p_diferente_q}, it comes out that those solutions are bounded in the $L^\infty$ norm independently of $n$. This means that for every $k\in \mathbb N$, the first $k$ solutions of \eqref{eq:another_eq_with_n} are indeed solutions of the original system provided $n$ is chosen large enough. Since this is true for every $k\in \mathbb N$, the conclusion follows. 
\end{proof}

\subsection{The Symmetric Mountain Pass Lemma combined with the reduction by inversion}\label{sec:symfount}
Here we show how the reduction by inversion approach of Section~\ref{sec:QuartaOrdem} combined with the Symmetric Mountain Pass Lemma \cite{AmbrosettiRabinowitz}, see also \cite[p. 5]{RabinowitzBook}, provide yet another short proof, even simpler than that of Theorem \ref{h=k=0}, of the existence of countable infinitely many solutions in a symmetric framework. Here we closely follow the presentation in \cite{dosSantosNodea2009}.

Let $\Omega$ be a smooth bounded domain in $\R^N$ with $N \geq 1$. We consider the system
\begin{equation}\label{S8}
\left\{
\begin{array}{rcll}
-\Delta u & = & \left| v\right|^{q-1} v & \hbox{in} \,\, \Omega, \\
-\Delta v & = & \left| u\right|^{p-1} u & \hbox{in} \,\, \Omega, \\
u, v &= & 0 & \hbox{on} \,\, \partial \Omega,
\end{array}
\right.
\end{equation}
under the same hypothesis made in Section \ref{sec:DualMethod}, namely 
\begin{equation}\label{eq:(p,q)_for_Ederson_partsection8}
p,q >0,\qquad 1> \gd{\frac{1}{p+1} + \frac{1}{q+1} > \frac{N- 2}{N}}.  \tag{H3}
\end{equation}

As we have seen at Section \ref{sec:QuartaOrdem}, the system \eqref{S8} is equivalent to the fourth-order equation
\begin{equation}\label{ENL8}
\left\{
\begin{array}{rcll}
\Delta \left( | \Delta u |^{\frac{1}{q} -1}\Delta u \right) & = & | u |^{p-1}u & \hbox{in}\,\,\Omega\\
u , \Delta u & = & 0 & \hbox{on}\,\,\partial \Omega.
\end{array}
\right.
\end{equation}

Let $E = W^{2,\frac{q+1}{q}}(\Omega) \cap W^{1, \frac{q+1}{q}}_0(\Omega)$. So, under \eqref{eq:(p,q)_for_Ederson_part2}, classical solutions $(u,v)$ of \eqref{S8} are the pairs such that $u$ is a critical point of the $C^1(E, \R)$ functional $J: E\to\R$ defined by
\[
J(u) = \frac{q}{q+1} \int_{\Omega} \left| \Delta u \right|^{\frac{q+1}{q}}\, dx - \frac{1}{p+1} \int_{\Omega}\left| u \right|^{p+1}\, dx.
\]

Here we prove the following theorem which improves Theorem \ref{h=k=0}.

\begin{theorem}\label{theorem}
Assume \eqref{eq:(p,q)_for_Ederson_partsection8}. Then \eqref{S8} has a sequence of classical solutions $(u_n,v_n)$ such that $J(u_n) \rt \infty$ as $n\rt \infty$.
\end{theorem}
\begin{proof}
By using Lemma \ref{lemmaps}, since \eqref{eq:(p,q)_for_Ederson_partsection8} is satisfied, it is possible to show that $J$ satisfies the Palais-Smale condition; cf. \cite[Lemma 3.4]{Ederson2008}. 

On the other hand, by using the superlinear hypothesis $1> \frac{1}{p+1} + \frac{1}{q+1}$, that is $pq>1$, and standard arguments it follows that $J$ has a mountain pass geometry around its local minimum at origin. Also, its is clear that $J$ is an even functional.

Then, by \cite[p. 5]{RabinowitzBook}, we just need to prove that for every finite dimensional subspace $F \con E$, there exists $R = R(F)> 0$ such that $J(u) \leq 0$ for $u \in F\menos B_{R(F)}$. This follows from the hypothesis $p > \frac{1}{q}$, since
\[
J(u) = \frac{q}{q+1} \| u \|_E^{\frac{q+1}{q}} - \frac{1}{p+1}\| u \|_{p+1}^{p+1}
\]
and that on finite dimensional subspaces all norms are equivalent.
\end{proof}

\subsection{Perturbation from symmetry}\label{sec:pertusym}

Throughout this subsection we will restrict the discussion to the case $N\geq 3$. Using the direct variational approach and adapting Rabinowitz's arguments \cite{RabinowitzBook}, Tarsi \cite{Tarsi} obtained a first result in the vein of  \cite{BahriBerestycki,RabinowitzBook,Rabinowitz,Struwe} for the system \eqref{symmsystem}.  As in \cite{deFigueiredoDing}, the approach relies on Galerkin type arguments. Tarsi \cite{Tarsi} proved the existence of infinitely many solutions for the perturbed system \eqref{S2} under the restriction (assuming also that $1<p\leq q$)
\begin{equation}\label{frac1p+frac1q+fracp(p-1)q}
\frac{1}{p+1}+\frac{1}{q+1}+\frac{p+1}{p(q+1)}> \frac{2N-2}{N}.
\end{equation}
We observe that this condition implies condition \eqref{frac2p + frac1p-1 > frac2N-2N} and it reduces to \eqref{frac2p + frac1p-1 > frac2N-2N} in case $p=q$. In particular, $p+1$ is not allowed to be close to the critical range $(2N-2)/(N-2)$ which appears in \eqref{p<frac2N-2N-2}. Observe also that both $p+1$ and $q+1$ have to be smaller than the critical Sobolev exponent $2N/(N-2)$.

Using the Lyapunov-Shmidt type reduction of Section \ref{sec:Reduced Functional}, Bonheure and Ramos \cite{BonheureRamosErratum,BonheureRamos} get rid of the indefiniteness of the energy functional associated to the system, giving rise to critical points whose energy is controlled (from below) by their Morse indices. This allows to obtain a result in the vein of Bahri and Lions \cite{BahriLions} and Tanaka \cite{Tanaka} improving \eqref{frac1p+frac1q+fracp(p-1)q}. 

\begin{theorem}[\cite{BonheureRamosErratum,BonheureRamos}]\label{main_theorem}
Let $h, k \in L^2(\Omega)$ and take $(p,q)$ satisfying 
\begin{equation}\label{fracN2(1-frac1p-frac1q)<fracp-1p}
1<p\leq q,\qquad \frac{N}{2} (1-\frac{1}{p+1}-\frac{1}{q+1})<\frac{p}{p+1}.
\end{equation}
Then the system \eqref{S2} admits an unbounded sequence of solutions $(u_k,v_k)_{k}\subset H^1_0(\Omega)\times H^1_0(\Omega)$.
\end{theorem}

\vspace{-.4cm}
\begin{center}
\includegraphics[scale=.26]{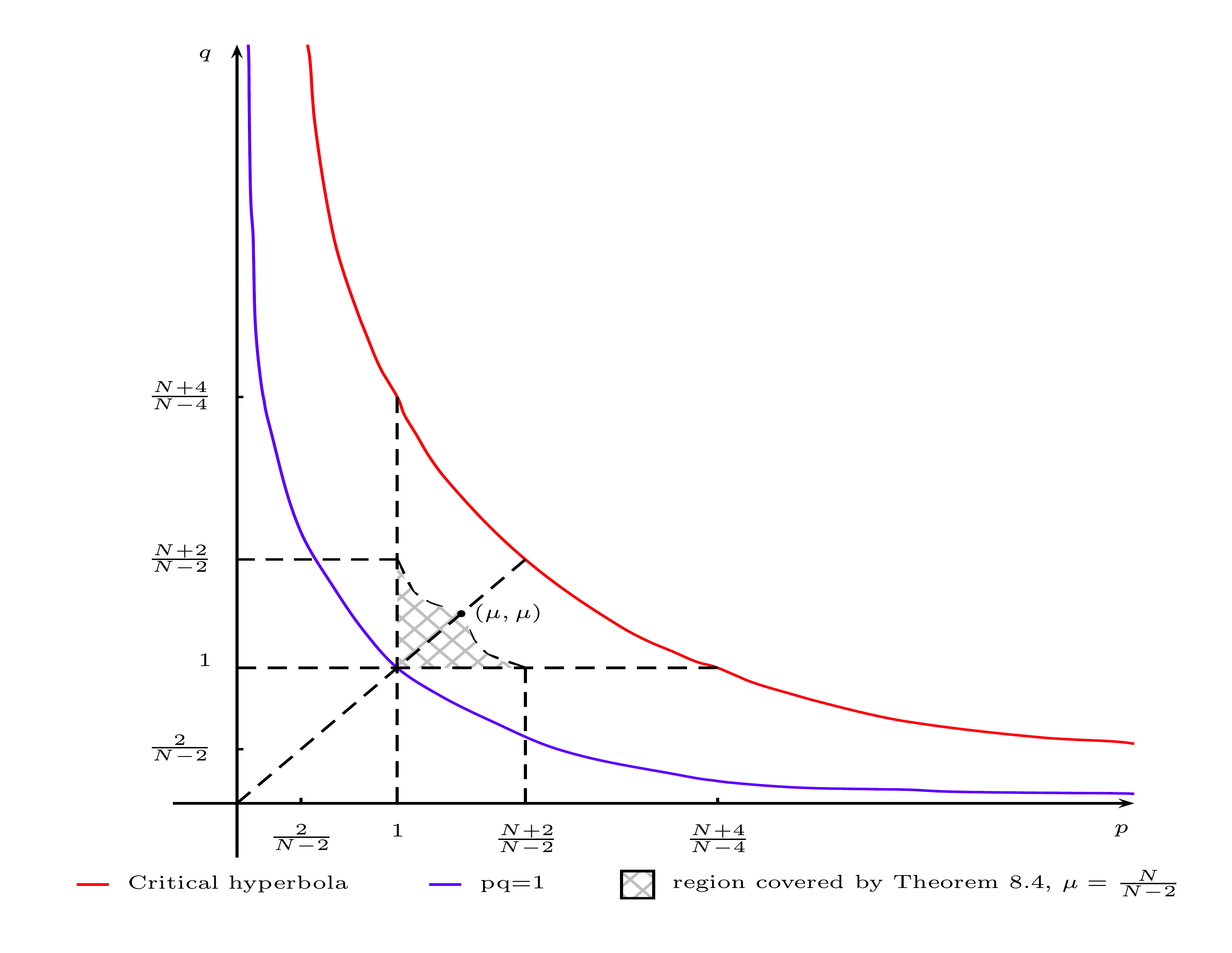}
\end{center}
\vspace{-.3cm}

Observe that the condition \eqref{fracN2(1-frac1p-frac1q)<fracp-1p} is sharp in the sense that it reduces to \eqref{p<frac2N-2N-2} in the case $p=q$. Moreover, this condition is implied by that expressed in \eqref{frac1p+frac1q+fracp(p-1)q}. On the other hand, \eqref{fracN2(1-frac1p-frac1q)<fracp-1p}  does force both $p+1$ and $q+1$ to be smaller than the Sobolev exponent $2N/(N-2)$. Observe also that we do assume both equations to be superlinear. Up to our knowledge, it is not known whether Theorem \ref{main_theorem} extends to superlinear systems under the milder assumption $pq>1$, except in dimension $N=1$, see the forthcoming Theorem \ref{thm:Denis-Ederson}.

The proof of Theorem \ref{main_theorem} combines the perturbation argument from Rabinowitz \cite{Rabinowitz} and Tanaka \cite{Tanaka} for the single equation \eqref{-Delta u = |u|^p-2u+h(x)} with the Lyapunov-Schmidt type reduction of Section \ref{sec:Reduced Functional} and makes use of a new estimate of the augmented Morse index of some min-max critical points of the reduced function $\J$. In fact, the main ingredients for proving this estimate is the family of reduced functional depending on a parameter $\J_{\lambda}$ (Subsection \ref{sec:reduced_with_lambda})  and the well-known Cwikel \cite{Cwikel}, Lieb \cite{Lieb} and Rosenbljum \cite{Rosenbljum} inequality  (see \cite{Simon} for a proof), which asserts that 
if $m^{*}_{V}(\alpha)$ denotes the number of eigenvalues $\mu\leq 1$ of the problem
$$
-\Delta \varphi=\mu V(x) \varphi, \qquad \varphi\in H^1_0(\Omega),
$$
with $N\geq 3$, then
\begin{equation*}
m^{*}_{V}(\alpha)\leq C \int V(x)^{N/2}\, dx
\end{equation*}
for some universal constant $C>0$. We refer to \cite{BonheureRamosErratum,BonheureRamos} for the complete proof. 

As discussed above, the condition on the exponents $p$ and $q$ can be improved in dimension $N=1$. The price to pay is that one needs to require more regularity on the perturbations $f,g$.
\begin{theorem}[\cite{BonheureSantos}]\label{thm:Denis-Ederson}
Suppose that $p,q >0$, $pq>1$ and $f,g \in C^1([0,1])$. Then the system 
\begin{equation}\label{sistema-DE}
\left\{
\begin{array}{rcll}
-u'' &= & \left| v \right|^{q-1}v + f(x) & x\in (0,1),\\
- v'' & = & \left| u \right|^{p-1}u + g(x) & x \in (0,1),\\
u(0)=v(0) & = & u(1)=v(1)=0 & \text{on} \,\, \{0,1\}.
\end{array}
\right.
\end{equation}
has infinitely many classical solutions.
\end{theorem}
The proof consists first in reducing (\ref{sistema-DE}) to a single nonlinear fourth-order equation as in Section \ref{sec:QuartaOrdem}. Let $u_{f}$ be the unique solution of $-u'' = f(x),\ x\in (0,1)$ that vanishes on at $x=0$ and $x=1$. Considering $w = u-u_{f}$, one is led to study the equation
\begin{equation*}
\left\{\!\!
\begin{array}{rcll}
\left(\left| w''\right|^{\frac{1}{q}- 1}w''\right)''& =& | w |^{p-1}w + |w+u_{f}|^{p-1}(w+u_{f}) - | w |^{p-1}w +g(x) & \text{ in } (0,1),\\
w, w''&= & 0 & \text{ on } \{ 0,1\}.
\end{array}
\right.
\end{equation*}
Assuming that $0 < q \leq 1$, the function
\begin{equation*}
(x,s) \mapsto |s+u_{f}(x)|^{p-1}(s+u_{f}(x)) - | s |^{p-1}s +g(x)
\end{equation*}
is in $L^{\infty}([0,1] \times \R)$ as soon as $u_{f}$ and $g$ are bounded. This motivates to consider the model problem
\begin{equation}\label{eq2}
\left\{
\begin{array}{rcll}
\left(| w''|^{\frac{1}{q}- 1}w''\right)''& = & | w |^{p-1}w + h(x,w) &  \text{ in } (0,1),\\
w, w''&= &0 & \text{ on } \{ 0,1\},
\end{array}
\right.
\end{equation}
with $h \in L^{\infty}([0,1] \times \R)$ and Carathéodory. The proof next follows Rabinowitz's approach \cite{Rabinowitz2} adapted by Garcia Azorero and Peral Alonso \cite{GarciaPeral} to deal with perturbations from symmetry involving the $p$-Laplacian operator. A crucial argument in Rabinowitz's method is the use of the asymptotic estimates for the eigenvalues of the Laplacian. In the case we have the Laplacian operator, since it is linear, these asymptotic estimates lead directly to Poincaré type inequalities on the orthogonal of the spaces generated by the $n$--th first eigenfunctions. When dealing with a nonlinear differential operator, this step is much more delicate and relies on some results on Schauder bases which are derived from Fourier analysis theory and topological isomorphism, for instance between $W^{2,p}((0,1))$ and $L^{p}((0,1))\times \R^2$ when dealing with the fourth order quasilinear operator as in \eqref{eq2}. One of the main ingredient then for the proof of Theorem \ref{sistema-DE} is that for every $1 < p < \infty$, $\{ \sin (n \pi t): n \geq 1 \}$ is a Schauder basis for $W^{1,p}_0((0,1))$ and for $W^{2,p}((0,1)) \cap W^{1,p}_0((0,1))$. \emph{It is this step which would require new ideas if one wants to improve {\rm Theorem \ref{sistema-DE}} to higher dimension}. Indeed, for instance, if $N=2$ and $\Omega = (0,1) \times (0,1)$, the sequence of eigenfunctions of $( -\Delta, H^1_0(\Omega))$, ordered according to the corresponding increasing value of the sequence of eigenvalues, is not a Schauder base for $L^{(p+1)/p} (\Omega)$ if $p \neq 1$, since the process of ``ball summation'' for the double Fourier series does not work; see \cite[Section 3.3 \& Theorem 3.5.6]{Krantz}. We refer to \cite{BonheureSantos} for the complete proof of Theorem \ref{sistema-DE}.

\section{Sign-changing solutions}

In this final section we briefly describe two results about sign-changing solutions of Hamiltonian systems. In the first subsection, we report a recent work dealing with least energy nodal solution for an H\'enon--type system, proving existence and symmetry properties. In the second subsection, we go back to the symmetric problem \eqref{symmsystem}, showing the existence of infinitely many sign-changing solutions. The latter result is proved with the Lyapunov-Schmidt type reduction approach, while the former uses the dual method.

\subsection{Least energy nodal solutions}\label{sec:lens}

In some of the first sections we addressed, from several points of view, the question of existence and symmetry of ground state solutions (or least energy solutions). Recently in \cite{BonheureSantosRamosTavares}, the authors together with Miguel Ramos have succeeded  in proving similar results for least energy nodal solutions, that is, solutions which minimize the energy among the set of all solutions which change sign. More precisely, the results hold for the H\'enon--type systems in a bounded domain $\Omega$:
\begin{equation}\label{eq:Henon_least_energy_nodal_solution}
\left\{
\begin{array}{ll}
-\Delta u=|x|^\beta |v|^{q-1}v & \text{ in } \Omega\\
-\Delta v=|x|^\alpha |u|^{p-1}u & \text{ in } \Omega\\
 u,v=0  & \text{ on } \partial \Omega
 \end{array}
\right.
\end{equation}
under the assumptions
\[
\alpha,\beta\geq 0,\qquad 1>\frac{1}{p+1}+\frac{1}{q+1}>\frac{N-2}{N}
\]
(which include the biharmonic case, namely when $\beta =0$ and $q=1$). We have used the dual method to treat the problem, and we will follow closely the notations from Section \ref{sec:DualMethod}.
Given $r,\gamma>0$, define
\[
L^r(\Omega,|x|^\gamma):=\{u:\Omega\to \R \text{ measurable}:\ \int_\Omega |u|^r|x|^{-\gamma}\, dx<\infty\}
\]
which is a Banach space equipped with the norm
\[
\|u\|_{r,\gamma}:= \left(\int_\Omega |u|^r|x|^{-\gamma}\, dx\right)^{1/r}.
\]
Observe that, since $\Omega$ is bounded and $\gamma>0$, we have the inclusions $L^r(\Omega,|x|^{-\gamma})\subset L^r(\Omega)$, where the last is the usual $L^r$--space. Define
\[
X:=L^{(p+1)/p}(\Omega, |x|^{-\alpha/p})\times L^{(q+1)/q}(\Omega, |x|^{-\beta/q}),
\]
\[ 
\|(w_1,w_2)\|_X:=\|w_1\|_{\frac{p+1}{p},\frac{\alpha}{p}}+ \|w_2\|_{\frac{q+1}{q},\frac{\beta}{q}}\qquad \forall w=(w_1,w_2)\in X
  \]
and consider the map $T: X\to L^1(\Omega)$  given by
  \[
  Tw=w_1Kw_2+w_2Kw_1 \qquad w=(w_1,w_2)\in X
  \]
  where, with some abuse of notations,  $K$ denotes the inverse  of the Laplace operator with zero Dirichlet boundary conditions. Let $I: X\to \IR$ be the associated energy functional
  \begin{multline*}
I(w_1,w_2)=\frac{p}{p+1}\int_{\Omega} |w_1|^{(p+1)/p}|x|^{-\alpha/p}\, dx+\frac{q}{q+1}\int_{\Omega} |w_2|^{(q+1)/q}|x|^{-\beta/q}\, dx\\-\frac{1}{2}\int_{\Omega}Tw\,dx.
\end{multline*}
Then the least energy nodal level can be defined by
\[
c_{\rm nod}=\inf\{I(w_1,w_2):\ w_1^\pm,w_2^\pm\neq 0,\ I'(w_1w_2)=0\}.
\]
As in the case of ground states (cf. Section \ref{sec:DualMethod}), this level can be characterized via a fiber-type map. Having this in mind, consider the constants
 \[
 \lambda:=\frac{2p(q+1)}{p+q+2pq}, \qquad \mu=\frac{2q(p+1)}{p+q+2pq},
 \]
 so that
 \[
 \gamma:=\lambda \, \frac{p+1}{p}=\mu\, \frac{q+1}{q}\in ]1,2[ \quad \mbox{ and } \qquad \lambda+\mu=2.
 \]
Given $w\in X$, define $\theta=\theta_w:\R_0^+\times \R_0^+\to \R$ by
\begin{equation}\label{eq:definition_of_theta}
\theta(t,s)=I(t^\lambda w_1^+-s^\lambda w_1^-,t^\mu w_2^+-s^\mu w_2^-),
\end{equation}
and observe that $(t,s)$ is a critical point of $\theta$ if and only if 
\[
(t^\lambda w_1^+-s^\lambda w_1^-,t^\mu w_2^+-s^\mu w_2^-)\in \Ncal_\textrm{nod},
\] where
\[
\Ncal_{\rm nod}:=\{(w_1,w_2)\in X:\ w_i^\pm \neq 0 \ \text{ and } I'(w)(\lambda w_1^+,\mu w_2^+)=I'(w)(\lambda w_1^-,\mu w_2^-)=0\}.
\]
Since $\theta_w$ may not have a global maximum for some $w\in \Ncal_\textrm{nod}$, we need to consider the following auxiliary set
\[
\Ncal_0=\left\{(w_1,w_2)\in X:\ \begin{array}{c} \lambda \int_\Omega w_1^+ Kw_2\, dx+\mu \int_\Omega w_1 Kw_2^+\, dx >0\\[0.1cm] 
			\lambda \int_\Omega w_1^- Kw_2\, dx+\mu \int_\Omega w_1 Kw_2^-\, dx <0  \end{array}\right\}.
\]
Observe that $w_i^\pm\not\equiv 0$ $\forall i=1,2,\ (w_1,w_2)\in \Ncal_0$. It can be proved that for each $w\in \Ncal_0$, $\theta_w$ admits a unique global maximum, which corresponds to a point in $\Ncal_\textrm{nod}$. The main results in \cite{BonheureSantosRamosTavares} deal with existence and symmetry of least energy nodal solutions, and provide several equivalent characterizations of the corresponding energy level.

\begin{theorem}
The number $c_{\rm nod}$ is attained by a function $w\in \Ncal_\textrm{nod}$, and
\begin{align*}
c_{\rm nod} &= \inf_{\Ncal_{\rm nod}} I =\inf_{w\in \Ncal_0} \sup_{t,s>0} I(t^\lambda w_1^+-s^\lambda w_1^-,t^\mu w_2^+-s^\mu w_2^-).
\end{align*}
Moreover, if $\Omega$ is a ball, then each least energy nodal solution $(u,v)$ is such that both $u,v$ are foliated Schwarz symmetric with respect to the same $p\in \partial B_1(0)$.
\end{theorem}

The proof of foliated Schwarz symmetry uses the notion of polarization, which has been introduced in Subsection \ref{workingwithpolarization}. In \cite{BonheureSantosRamosTavares}, several examples of symmetry breaking are provided.

\begin{remark}
It is not clear if any approach different from the dual method could have been used to solve this problem. For instance if one tried the Lyapunov-Schmidt type reduction {\rm (Section \ref{sec:Reduced Functional})}, one would have to deal with the functions $\Psi_{u^+}$, $\Psi_{u^-}$, which seem difficult to characterize and to compare. On the other hand, by choosing the reduction by inversion approach {\rm (Section \ref{sec:QuartaOrdem})}, one could not have worked with positive and negative parts of functions, since these are no longer in the domain of the corresponding energy functional; instead, one would have to deal with the projections on the cones of positive and negative functions.
\end{remark}

\subsection{Multiplicity results}\label{sec:infinitely_sign_changing}

Next we go back to the task of obtaining  multiplicity results for the system \eqref{eq:main_system_particularcase}, which we repeat here for convenience of the reader:
\begin{equation}\label{eq:system:BahriLions}
\left\{
\begin{array}{rcll}
-\Delta u & = & \left| v\right|^{q-1} v  & \hbox{in} \,\, \Omega, \\
-\Delta v & = & \left| u\right|^{p-1} u  & \hbox{in} \,\, \Omega, \\
u, v &= & 0 & \hbox{on} \,\, \partial \Omega.
\end{array}
\right.
\end{equation}
We will assume \eqref{eq:p_and_q_reduzidoGERAL}, which we recall as being
\begin{equation}\label{eq:p_and_q_reduzidoGERAL9}
p,q>1,\qquad \frac{1}{p+1}+\frac{1}{q+1}>\frac{N-2}{N} \tag{H4}.
\end{equation}

In Subsections \ref{se:Galer}-\ref{sec:symreduc}-\ref{sec:symfount} we have seen several ways of proving that \eqref{eq:system:BahriLions} admits infinitely many solutions (which have increasing energy). On the other hand, a priori bounds for positive solutions are known to hold under \eqref{eq:(p,q)_for_W^1,s_0} for $N\leq 4$ (see \cite{PolacikQuittnerSouplet, Souplet}). For $N\geq 5$, there are some partial results which do not cover entirely the case $(p,q)$ satisfying \eqref{eq:(p,q)_for_W^1,s_0}, namely \cite{BuscaManasevich,Souplet}, to which we refer for a more complete history of the subject\footnote{These questions are strongly related with the so called \emph{Lane-Emden conjecture}, which affirms that there are no positive solutions in the entire space to \eqref{eq:system:BahriLions} under \eqref{eq:(p,q)_for_W^1,s_0}; this is completely established by now in dimension $N\leq 4$.}. In the additional assumptions that $\Omega$ is convex and \eqref{eq:p_and_q_reduzidoGERAL9} holds, a priori bounds of positive solutions are known for all space dimensions \cite{ClementdeFigueiredoMitidieri}. 

In these cases, the a priori bounds combined with the previously mentioned multiplicity results yield the existence of infinitely many \emph{sign-changing} solutions. These results do not cover all space dimensions; nevertheless, the existence of infinitely many sign-changing solutions was proved \emph{directly} with success in \cite[Theorem 4]{RamosTavaresZou} by using the approach introduced in Section \ref{sec:Reduced Functional}:

\begin{theorem}\label{thm:BahriLions}
Assume that $(p,q)$ satisfies \eqref{eq:p_and_q_reduzidoGERAL9}. Then \eqref{eq:system:BahriLions} admits an unbounded sequence of sign changing solutions $(u_k,v_k)$ in the sense that $(u_k+v_k)^+\neq 0$, $(u_k+v_k)^-\neq 0$ for every $k$. 
\end{theorem}

\begin{remark}
For $(u,v)$ solution of the problem, one has that $u+v$ changes sign if and only if both $u$ and $v$ change sign. In fact, the direct implication can be proved by using the maximum principle. As for the reverse implication, suppose that $u^\pm,v^\pm\neq 0$. One must have $\{u>0\}\cap\{v>0\}\neq \emptyset$, otherwise $\{u>0\}\subseteq \{v\leq 0\}$ and, by multiplying the equation of $u$ by $u^+$, we would get $\int_\Omega |\nabla u^+|^2\, dx=\int_\Omega |v|^{q-1}v u^+\, dx\leq 0$ and $u^+\equiv 0$, a contradiction. Analogously, $\{u<0\}\cap \{v<0\}\neq \emptyset$, and the claim follows.
\end{remark}

\begin{proof}[Sketch of the proof of {\rm Theorem \ref{thm:BahriLions}}]
As explained in Section \ref{sec:Reduced Functional} (see in particular Remark \ref{rem:p=q}), we can suppose without loss of generality that $1<p=q<2^*-1$. We will use the notations of that section, recalling for instance the notion of reduced functional
\[
\mathcal{J}(u):=I(u+\Psi_u,u-\Psi_u),\qquad u\in H^1_0(\Omega)
\]
(cf. \eqref{eq:reducedfunctional_J}). For $k\in \N$, let $E_k:=\spann\{\phi_1,\ldots, \phi_k\}$, where $\phi_i$ is again the $i$--th eigenvalue of $(-\Delta,H^1_0(\Omega))$, and $S_k:=\{E_k^\perp:\ \|u\|_{L^p}=1\}$. It can be proved that there exists $c_0>0$ (independent of $k$) such that
\[
\inf_{S_k} \mathcal{J}>-c_0
\]
and that (for large $R_k>0$)
\[
\sup_{\partial Q_k} \mathcal{J}<-c_0,\qquad \text{ for } Q_k:=B_{R_k}\cap E_k.
\]
In order to prove existence of sign-changing solutions, it is important to have some estimates outside the set $\mathcal{P}$, the cone of positive solutions. One proves that there exists $M_k$ and $\mu_k$ such that
\[
\textrm{dist}(u,\mathcal{P})\geq 2\mu_k, \quad \forall u\in S_k,\ \mathcal{J}(u)\leq M_k,
\]
and that the cone $\mathcal{P}$ is invariant for the flow, in the sense that if $\sigma(t,u)$ solves
\[
\frac{\partial \sigma}{\partial t}(t,u)=-\chi(\sigma(t,u))\frac{\nabla \mathcal{J}(\sigma(t,u))}{\|\nabla \mathcal{J}(\sigma(t,u))\|},\quad \sigma(0,u)=0
\]
for some smooth function $\chi:H^1_0(\Omega)\to [0,1]$, then 
\[
\textrm{dist}(u,\mathcal{P})\leq \mu_k \Rightarrow \textrm{dist}(\sigma(t,u),\mathcal{P})\leq \mu_k \qquad \forall t\geq 0.
\]
Then it can be proved that
\[
c_k=\inf_{\gamma\in \Gamma_k}\sup_{\gamma(Q_k)\cap \{\textrm{dist}(u,\mathcal{P})\geq \mu_k\}} \mathcal{J},
\]
with
\[
\Gamma_k=\{\gamma\in C(Q_k,H^1_0(\Omega)):\ \gamma \text{ is odd }, \gamma|_{\partial Q_k}=Id,\ \sup_{\gamma(Q_k)} \mathcal{J}<M_k\},
\]
is a critical level of $\mathcal{J}$, having a sign changing critical point $u_k$ with Morse index less than or equal to $k$. By combining this with a suitable notion of linking, it is proved the existence of a sign-changing solution $u_k^*$  with augmented Morse index $m^*(u_k)$ greater than or equal to $k$, and such that
\[
C' k^\frac{2(p+1)(q+1)}{(pq-1)N}\leq C' (m^*(u))^\frac{2(p+1)(q+1)}{(pq-1)N}\leq \mathcal{J}(u_k^*)\leq c_k \leq C k^\frac{2(p+1)}{N(p-1)}+C,
\]
where in the second inequality \cite[Proposition 9]{BonheureRamos} is used.
\end{proof}

It should be noted that the strategy of the proof is flexible enough to be applied to obtain ``perturbation of symmetry'' results in the case of single equation problems involving the harmonic or the biharmonic operator \cite{RamosTavaresZou}.

\vspace{0.5cm}

\textbf{Acknowledgments}  
D. Bonheure was supported by MIS - ARC - INRIA - FRFC. E. M. dos Santos was supported by CNPq - FAPESP. H. Tavares was supported by FCT/Portugal through project PEst-OE/EEI/LA0009/2013 and through the program \emph{Investigador FCT}. 
\newline We would like to thank B. Sirakov, P. Souplet and M. Willem for their comments and for providing us additional references, and to one of the referees for his valuable suggestions on a first version of this paper.
\newline Miguel Ramos was our colleague, collaborator and most of all friend. For all what he taught and gave to us, this paper is dedicated to him.



\end{document}